\documentclass[11pt]{amsart}
\usepackage{geometry, amsaddr}                		
\numberwithin{equation}{section}
\usepackage{graphicx}				
\usepackage{amssymb}

\begin{document}

\title{Cartan Subalgebras for Quantum Symmetric Pair Coideals}

\author{Gail Letzter}
\address{Mathematics Research Group, National Security Agency}
\email{gletzter@verizon.net}
\date{}

\subjclass[2010]{Primary 17B37, Secondary 17B10, 17B22}
\keywords{Quantized enveloping algebras, symmetric pairs, coideal subalgebras, Cartan subalgebras}

\begin{abstract}{
There is renewed interest in the coideal subalgebras used to form quantum symmetric pairs because of recent discoveries showing that they play a fundamental role in the representation theory of quantized enveloping algebras. However, there is still no general theory of finite-dimensional modules for these coideals. In this paper, we establish an important step in this direction: we show that every quantum symmetric pair coideal subalgebra admits a quantum Cartan subalgebra which is  a polynomial ring  that specializes to its classical counterpart. The construction builds on Kostant and Sugiura's classification of Cartan subalgebras for real semisimple Lie algebras via strongly orthogonal systems of positive roots. We show that these quantum Cartan subalgebras act semisimply on finite-dimensional unitary modules and identify particularly nice generators of the quantum Cartan subalgebra for a family of  examples.
}\end{abstract}

\maketitle

\tableofcontents

\theoremstyle{plain} 
\newtheorem{theorem}{Theorem}[section]
\newtheorem{lemma}[theorem]{Lemma}
\newtheorem{proposition}[theorem]{Proposition}
\newtheorem{corollary}[theorem]{Corollary}
\newtheorem*{thmA}{Theorem B}
\newtheorem*{thmB}{Theorem A}

\theoremstyle{definition}
\newtheorem{definition}[theorem]{Definition}

\theoremstyle{remark}
\newtheorem{remark}[theorem]{Remark}
\newcommand{\ad}{{\rm ad\ }}
\newcommand{\Ad}{{\rm Ad\ }}

\newcommand{\adr}{\ensuremath{{\rm ad}}_r}

\section{Introduction}
In the mid 1980's, Drinfeld and Jimbo (\cite{Dr}, \cite{Ji}) introduced  a family of algebras, referred to as  quantized enveloping algebras, which are Hopf algebra deformations of enveloping algebras of semisimple Lie algebras. Shortly afterwards,   Noumi, Sugitani,  Dijkhuizen and the author  developed  a theory of quantum symmetric spaces  based on the construction of quantum analogs of symmetric pairs (\cite{D},  \cite{NS}, \cite{N}, \cite{L1}, \cite{Let}, \cite{Let2}, \cite{L2}).
Classically, a  symmetric pair consists of a semisimple Lie algebra and a Lie subalgebra fixed by an involution.  Quantum symmetric pairs are formed using  one-sided coideal subalgebras because quantized enveloping algebras contain very few Hopf subalgebras.  Although  quantum symmetric pair coideal subalgebras 
 can be viewed as  quantum analogs of  enveloping algebras of the underlying fixed Lie subalgebras, which are either semisimple or reductive,  there is still no general classification of  finite-dimensional modules for these coideal subalgebras.   In this paper, we establish a significant step in this direction: we show that every quantum symmetric pair coideal subalgebra admits a quantum Cartan subalgebra which is  a polynomial ring  that specializes to its classical counterpart and acts semisimply on finite-dimensional unitary modules.  It should be noted that  Kolb has extended the theory of quantum symmetric pair coideal subalgebras to quantized Kac-Moody algebras (\cite{Ko}). Although we do not consider this more general situation here, we expect many of the ideas and constructions  of this paper to carry over to the Kac-Moody setting.

 One of the original goals in developing the theory of quantum symmetric spaces,  realizing Macdonald-Koornwinder polynomials as quantum zonal spherical functions, has in large part been completed (see \cite{N}, \cite{NS},  \cite{DN}, \cite{NDS}, \cite{DS}, \cite{Let}, \cite{Let2}).  Recently, one-sided coideal subalgebras used to form quantum symmetric pairs have gained renewed attention because of the discovery of their fundamental roles in the representation theory of quantized enveloping algebras. 
 Ehrig and  Stroppel  (\cite{ES}) show that a family of these coideal subalgebras are just the right algebras needed  in the categorification of certain quantized enveloping algebra representations arising in Schur duality and skew Howe duality.   
 In \cite{BW}, Bao and Wang  
 obtain canonical bases for representations of this family of  coideal subalgebras and use it to settle conjectures concerning irreducible characters of  Lie superalgebras. Expanding greatly on these ideas, Bao and Wang present a general theory of  canonical bases that works for all types of quantum symmetric pair coideal subalgebras in \cite{BW2}. Following a program initiated in \cite{BW}, Balagovic and Kolb (\cite{BK}) use quantum symmetric pair coideal subalgebras to construct universal $K$-matrices which provide solutions of the reflection equation arising in inverse scattering theory. These $K$-matrices (are expected to) play a similar role as the universal $R$-matrices for quantized enveloping algebras.

The body of work described in the previous paragraph relies on a variety of representations with respect to the quantum symmetric pair coideal subalgebras. However, these representations are mainly modules or submodules of the larger quantized enveloping algebras. For example, the analysis of zonal spherical functions on quantum symmetric spaces uses spherical modules, a special family of finite-dimensional simple representations for the quantized enveloping algebras (see for example \cite{L0}, \cite{L1}, and \cite{L2}). Ehrig  and Stroppel establish dualities involving quantum symmetric pair coideal subalgebras using exterior products of the natural representations for the larger quantized enveloping algebras (\cite{ES}). Bao and  Wang obtain canonical basis with respect to a quantum symmetric pair coideal subalgebra for  tensor products of finite-dimensional simple modules of the quantized enveloping algebra (\cite{BW2}).
One of the beautiful properties of Lusztig's canonical bases (\cite{Lu}, Chapter 25) is that they yield compatible bases for each finite-dimensional simple module of a quantized enveloping algebra.  Without a highest weight theory and corresponding classification of finite-dimensional simple modules for the quantum symmetric pair coideal subalgebras, a comparable result  remains elusive.

The Cartan subalgebra of a quantized enveloping algebra is the Laurent polynomial ring generated by the group-like elements. 
The group-like property enables one to easily break down many  modules into direct sums of weight spaces in the context of a highest weight module theory for the quantized enveloping algebra.  
This Cartan is unique in the sense that it is the only Laurent polynomial ring contained inside the quantized enveloping algebra of the correct size. Unfortunately, in most cases, a quantum symmetric pair coideal subalgebra does not inherit a  Cartan subalgebra in the form of a Laurent polynomial ring of group-like elements from the larger quantized enveloping algebra. This is because the Cartan subalgebra of the larger quantized enveloping algebra corresponds to a maximally split one with respect to the involution defining the quantum symmetric pair coideal subalgebra.  Hence the Laurent polynomial ring generated by the group-like elements inside a quantum symmetric pair coideal subalgebra is typically too small to  play the role of a Cartan subalgebra.  
This lack of an obvious Cartan subalgebra consisting of group-like elements has been the primary reason so little progress has been made on a general representation theory for  quantum symmetric pair coideal subalgebras.    In this paper, we show that, despite this fact, there are good  Cartan subalgebra candidates for quantum symmetric pair coideal subalgebras (as stated in Theorem B below).  

In order to find  good choices for quantum symmetric pair Cartan subalgebras, we first revisit the classical case. It is well-known that there is a one-to-one correspondence between isomorphism classes of real forms   for a complex semisimple Lie algebra $\mathfrak{g}$ and conjugacy classes of involutions on 
$\mathfrak{g}$. Thus, understanding the Cartan subalgebras of $\mathfrak{g}^{\theta}$ is closely connected to understanding Cartan subalgebras of real semisimple Lie algebras. 
Unlike the complex setting, the set of Cartan subalgebras of a real semisimple  Lie algebra do not form a single conjugacy class with respect to automorphisms of the Lie algebra.  A theorem of Kostant and Sugiura (\cite{OV} Section 4.7,  \cite{Si}, and \cite{Kos}) classifies these conjugacy classes using subsets of the positive roots that satisfy strongly orthogonality properties with respect to the Cartan inner product.  Using Cayley transforms, which are a special family of Lie algebra automorphisms,  the Cartan subalgebra $\mathfrak{h}$ can be converted into $\mathfrak{t}$, a Cartan subalgebra  of $\mathfrak{g}$ whose intersection with $\mathfrak{g}^{\theta}$ is  a Cartan subalgebra of this smaller Lie algebra.     In other words, for any symmetric pair, one can use Cayley transforms to express a Cartan subalgebra of $\mathfrak{g}^{\theta}$ in terms of a fixed triangular decomposition $\mathfrak{g} = \mathfrak{n}^-\oplus \mathfrak{h}\oplus \mathfrak{n}^+$.  

The Cartan subalgebra $\mathfrak{t}\cap \mathfrak{g}^{\theta}$ is spanned  
by elements in the union of 
$\{e_{\beta}+f_{-\beta}, \beta\in \Gamma_{\theta}\}$ and $\mathfrak{g}^{\theta}\cap \mathfrak{h}$    where $\Gamma_{\theta}$ is a system of strongly orthogonal positive roots associated to $\theta$ (See Section \ref{section:csorla}).  Note that elements in $\mathfrak{g}^{\theta}\cap \mathfrak{h}$ correspond in the quantum setting to the group-like elements of $\mathcal{B}_{\theta}$. Lifting elements of the form $e_{\beta} + f_{-\beta}$ is more complicated. In order to accomplish these lifts, we use particularly nice systems of  strongly orthogonal positive roots that satisfy extra orthogonality conditions with respect to the Cartan inner product:

 \begin{thmB}(Theorem \ref{theorem:cases_take2}) For each maximally split symmetric pair $\mathfrak{g}$,$\mathfrak{g}^{\theta}$, there exists a maximum-sized strongly orthogonal subset $\Gamma_{\theta}=\{\beta_1, \dots, \beta_m\}$ of positive roots such that for all $j$  we have  $\theta(\beta_j)=-\beta_j$, and   for all $i>j$ and all simple roots $\alpha$ in the support of $\beta_i$, $\alpha$ is strongly orthogonal to $\beta_j$.  Moreover each $\beta_j$ is strongly orthogonal to all but at most two distinguished simple roots contained in its support.\end{thmB}

For $\mathfrak{g}, \mathfrak{g}^{\theta}$  maximally split and $\Gamma_{\theta}$ chosen as in Theorem A, the elements $e_{\beta}+f_{-\beta}$, $\beta\in \Gamma_{\theta}$ are lifted  to the quantum case in a two-step process.  
The first step lifts $f_{-\beta}$ to the lower triangular part of the quantized enveloping algebra  so that the lift commutes with all generators for $U_q(\mathfrak{g})$ defined by simple roots strongly orthogonal to $\beta$.  
 In many cases, this lift is a member of one of    Lusztig's PBW basis (\cite{Ja}, Chapter 8).  For other cases,
  the construction is more subtle. It relies on showing that elements of special submodules of $U_q(\mathfrak{g})$ with respect to the adjoint action satisfy nice commutativity properties. The desired Cartan element associated to $\beta$ is formed by considering elements of $\mathcal{B}_{\theta}$ with lowest weight term equal to the lift of $f_{-\beta}$.  
The process uses fine information on the possible  weights and biweights in the expansion of elements in $\mathcal{B}_{\theta}$ in terms of weight and biweight vectors and a special projection map.   (See Section \ref{section:basic} for the precise notion of biweight.) We further insist that the quantum analogs  of the $e_{\beta}+f_{-\beta}$ are fixed by a quantum version of the Chevalley antiautomorphism.  This property ensures that elements of the quantum Cartan subalgebra of a quantum symmetric pair coideal subalgebra acts semisimply on a large family of finite-dimensional modules. 
These lifts are the substance behind our main result,  Theorem \ref{theorem:main}, which can be summarized as: 
\begin{thmA}   Every quantum symmetric pair coideal subalgebra  $\mathcal{B}_{\theta}$ contains a commutative subalgebra $\mathcal{H}$ that is a polynomial ring over the algebra of  group-like elements inside $\mathcal{B}_{\theta}$. Moreover, $\mathcal{H}$ specializes to the enveloping algebra of a Cartan subalgebra of the underlying fixed Lie subalgebra, $\mathcal{H}$ acts  semisimply on finite-dimensional simple unitary $\mathcal{B}_{\theta}$-modules, and $\mathcal{H}$ satisfies certain uniqueness conditions related to its expression as a sum of weight vectors.
\end{thmA}

There is one family of quantum symmetric pair coideal subalgebras that has its own separate history, namely, the one consisting of  the nonstandard $q$-deformed enveloping algebras of orthogonal Lie algebras  introduced in a paper by Gavrilik and Klimyk \cite{GK}.  The fact that the nonstandard $q$-deformed enveloping algebra $U'_q(\mathfrak{so}_n)$ agrees with one of the coideal subalgebras was realized from the beginning  of the theory of quantum symmetric pairs (see for example \cite{N}, Section 2.4 and \cite{Let2}, end of Section 2). In a series of papers, Gavrilik, Klimyk, and  Iorgov classify the finite-dimensional simple modules for $U'_q(\mathfrak{so}_n)$ using quantum versions of Gel'fand-Tsetlin basis (see \cite{GK},  \cite{IK} 
and references therein).   
Rowell and Wenzl study  $q$-deformed algebras $U'_q(\mathfrak{so}_n)$ in the context of fusion categories (\cite{W}, \cite{RW}).  Taking advantage of the  finite-dimensional representation theory for $U_q'(\mathfrak{so}_n)$, they prove parts of  a conjecture on braided fusion categories (\cite{RW}).  An interesting aspect of their work is that they redo the classification of finite-dimensional modules for $U'_q(\mathfrak{so}_n)$  by developing a Verma module theory based on actions of a quantum Cartan subalgebra that agrees with the one studied here  (see Remark \ref{nonstandard}).  Preliminary investigations suggest a similar Verma module theory can be realized for other symmetric pair coideal subalgebras by analyzing the action of the quantum Cartan subalgebras presented in this paper. We intend to explore this further in future work.

 There is also  recent work on the representation theory of $\mathcal{B}_{\theta}$  for another  infinite family of quantum symmetric pairs that  form a subset of those of type AIII.   In \cite{Wa1}, H. Watanabe constructs a triangular decomposition for these coideal subalgebras, uses this decomposition to develop a Verma  module theory,  and then classifies irreducible modules living inside a large category of representations.
  The Cartan part of this triangular decompositon plays a standard role:  Cartan  elements act as scalars on  highest weight generating vectors for Verma modules and these scalars are the weights used in the classification.  Building on this foundation, Watanabe studies analogs of Kashiwara's crystal bases and shows that they have particularly nice combinatorial properties.
 We see in Section \ref{section:examples} that the Cartan part of this triangular decomposition in \cite{Wa1}   agrees with the one presented here.  We expect that the Cartan subalgebras of this paper can be used to classify families of  irreducible modules  for all quantum symmetric pair coideal subalgebras of  type AIII using very similar arguments to those in \cite{Wa1}.  This would be a first step towards generalizing  Watanabe's crystal basis results to other symmetric pairs, especially others of type AIII.   
 
 In analogy to Lusztig's well-known braid group action on quantized enveloping algebras,  Kolb and Pellegrini  (\cite{KP}) introduce a braid group action on certain families of quantum symmetric pair coideal subalgebras.  One might expect that it is possible  to construct generators for quantum Cartan subalgebras by starting with one element and taking images of it under braid group operators.   However, as explained in Remark \ref{finalremark}, the situation is more subtle.  Indeed, this method does not necessarily result in  elements that commute with each other.  Nevertheless, the braid group action could be quite useful in computing the action of the Cartan subalgebra on various representations.  In particular, Watanabe  uses the braid group action in order to compute the action of Cartan elements on highest weight vectors and this might generalize to other symmetric pair families (see \cite{Wa1}, Section 4.3).
      
There are many connections with the quantum Cartan subalgebra of this paper and the   center of $\mathcal{B}_{\theta}$ studied in \cite{KL}. Both constructions take advantage of various projection maps,  weight and  biweight  expansions,  the structure of $U_q(\mathfrak{g})$ with respect to the adjoint action, and certain commutativity conditions.    In the last section of this paper, we comment on  a particularly close connection between the center and the quantum Cartan subalgebra of $\mathcal{B}_{\theta}$ for a symmetric pair of type AIII/AIV (see Remark \ref{remark:lastremark}).  In \cite{Ko2}, Kolb shows that central elements yield solutions to reflection equations.  It would be interesting to see whether  the quantum Cartan subalgebra elements of this paper, which share many properties of central elements, yield solutions to a related family of equations.  

This paper is organized as follows.  Section \ref{section:csorla} is devoted to the classical case. After introducing basic notation in Section \ref{section:basic_classical}, 
we  review in Section \ref{section:realforms} some of the classical theory of Cartan subalgebras in the real setting, explain the notion of   strongly orthogonal systems, and present part of the  Kostant-Sugiura classification of real Cartan subalgebras in terms of strongly orthogonal systems of positive roots. In Section \ref{section:cayley_transforms}, we show how to use this classification and Cayley transforms to identify the Cartan subalgebra of a fixed Lie subalgebra $\mathfrak{g}^{\theta}$ with respect to a given triangular decomposition of $\mathfrak{g}$.  Maximally split involutions are the focus of Section \ref{section:max_split_inv} where we present Theorem A (Theorem \ref{theorem:cases_take2}) with  case-by-case explicit descriptions of the desired maximum-sized strongly orthogonal systems of positive roots. 

 Section \ref{section:qea} reviews basic properties of   quantized enveloping algebras.   We set notation  in Section \ref{section:basic},   discuss  adjoint module structures and dual Vermas in Section \ref{section:dual_vermas} and describe Lusztig's automorphisms in Section \ref{section:lusztig}.  In Section \ref{section:chevalley}, we review the quantum Chevalley antiautomorphism and the fact that subalgebras invariant with respect to this automorphism act semisimply on finite-dimensional unitary modules.  

In Section \ref{section:qsp}, we turn our attention to quantum symmetric pair coideal subalgebras.  After setting notation in Section \ref{section:Definitions}, we establish in Section \ref{section:weight_space} a result on the biweight space expansions of elements of a coideal subalgebra $\mathcal{B}_{\theta}$ associated to a maximally split involution $\theta$.  In Section \ref{section:decomp_proj}, we introduce a projection map on the quantized enveloping algebra. We then show how to construct elements of $\mathcal{B}_{\theta}$  equal to a given lower triangular element plus terms that vanish with respect to this projection map.  

The main result of Section \ref{section:lower} is  a version of the first assertion of Theorem B for the lower triangular part of $U_q(\mathfrak{g})$.  The proof is based on lifting root vectors associated to weights in strongly orthogonal root systems (as specified by Theorem A) to elements in the lower triangular part that satisfy strong commutativity properties.  These lifts relies on an  analysis of  certain finite-dimensional simple modules in the classical setting (Section \ref{section:siv}) and lifts of special families of root vectors in  Section \ref{section:sie} and Section \ref{section:lowtri}. Details of the lift are given in Section \ref{section:lowtri2}.

Section \ref{section:cond_comm} sets up the basic tools needed to create  Cartan elements in $\mathcal{B}_{\theta}$. Section \ref{section:gen} introduces the notion of generalized normalizers in order to analyze when elements of $\mathcal{B}_{\theta}$ commute with elements of carefully chosen sub-quantized enveloping algebras 
of $U_q(\mathfrak{g})$. Section \ref{section:weight_cons}   and Section \ref{section:com_crit} establish properties of  lowest and highest weight summands respectively of elements inside special generalized normalizers associated to  $\mathcal{B}_{\theta}$. Building on the prior sections, Section \ref{section:det_comm} presents methods for constructing  Cartan elements with desired commutativity properties in $\mathcal{B}_{\theta}$ from the lift of the root vectors $f_{-\beta}$. 

Theorem B, the main result of the paper,  is the heart of Section \ref{section:cart_cons}. First, we establish in Section \ref{section:spec} a  specialization criteria  that is  later used  to show that the quantum Cartan subalgebras specialize to their classical counterparts.    
We show how to construct the desired Cartan elements, and thus prove the detailed version of Theorem B (Theorem \ref{theorem:main} and Corollary \ref{corollary:main}), in Section \ref{section:qcs}. 

In the final section (Section \ref{section:examples}) of the paper, we analyze in detail the quantum Cartan subalgebras for a family of symmetric pairs of type AIII/AIV.  After a brief overview of the generators and relations for  this family  in Section \ref{section:ex_overview}, we  present the $n=2$ and $n=3$ cases in  Sections \ref{section:nequals2} and \ref{section:nequals3} respectively.   For  general $n$ (Section \ref{section:ex_general}), we give explicit formulas for  a set of generators for the quantum Cartan subalgebra.  We also establish a close connection between the quantum Cartan subalgebra and central elements of subalgebras of $\mathcal{B}_{\theta}$. 

We conclude the paper with an appendix listing commonly used  notation  by order of appearance. 

\medskip
\noindent
{\bf Acknowledgements:}  I especially want to thank Catharina Stroppel and Stefan Kolb for their intriguing questions and stimulating mathematical discussions that inspired me to think again about quantum symmetric pairs.   I am very grateful to the organizers Istv\'an Heckenberger, Stefan Kolb, and Jasper Stokman for inviting me to   the Mini-Workshop: Coideal Subalgebras of Quantum Groups, held at the Mathematisches Forschungsinstitut Oberwolfach.   I would also like to express my appreciation to the referee for carefully reading this manuscript and providing insightful and constructive suggestions. Some of this work is related to open questions  formulated after attending  the conference (see  \emph{An Overview of Quantum Symmetric Pairs} in \cite{Ob}).  A preliminary version of the results in this paper were presented at the Joint Mathematics Meetings, Atlanta 2017.  

 \section{Cartan Subalgebras of Real Semisimple Lie Algebras}\label{section:csorla}

 Let $\mathbb{C}$ denote the complex numbers, $\mathbb{R}$ denote the real numbers, and $\mathbb{N}$ denote the nonnegative integers.    We  recall facts about   normal real forms of a  complex semisimple Lie algebra $\mathfrak{g}$ and present
 part of Kostant and Sugiura's results on the classification of Cartan subalgebras in the real setting in terms of strongly orthogonal systems of positive roots.  Much of the  material in this section   is based on the presentation in \cite{OV} Section 4.7 and \cite{K} Chapter VI, Sections 6 and 7 (see also \cite{Si},  \cite{AK}, and \cite{DFG}).  However, we take a slightly different perspective, focusing on Cartan subalgebras of a subalgebra $\mathfrak{g}^{\theta}$  of $\mathfrak{g}$ fixed by an involution $\theta$ rather than Cartan subalgebras  of related real forms.   We conclude this section with Theorem \ref{theorem:cases_take2}, the detailed version of Theorem A, which identifies  strongly orthogonal systems of positive roots with special properties.

\subsection{Basic Notation: the Classical Case}\label{section:basic_classical}
Let $\mathfrak{g} = \mathfrak{n^-}\oplus \mathfrak{h}\oplus \mathfrak{n}^+$ be  a fixed triangular decomposition of the complex semisimple Lie algebra $\mathfrak{g}$. Let $\Delta$ denote the set of roots of $\mathfrak{g}$ and recall that $\Delta$ is a subset of the dual space $\mathfrak{h}^*$  to $\mathfrak{h}$. Write $\Delta^+$ for the subset of positive roots.  Let $h_i,e_{\alpha}, f_{-\alpha}$, $i=1,\dots, n$, $\alpha\in \Delta^+$ be  a Chevalley basis for $\mathfrak{g}$.
Set $\pi=\{\alpha_1,\dots, \alpha_n\}$ equal to a fixed choice of simple roots for the  root system $\Delta$. As is standard, we view roots as elements of $\mathfrak{h}^*$ (see for example \cite{H}, Chapter 14). Let $\mathcal{W}$ denote the Weyl group associated to the root system $\Delta$ and let $w_0$ denote the longest element of $\mathcal{W}$. Given a simple root $\alpha_i$, we often write $e_i$ for $e_{\alpha_i}$ and $f_i$ for $f_{-\alpha_i}$. Given a positive root $\alpha\in \mathfrak{h}^*$, we write $h_{\alpha}$ for the coroot in $\mathfrak{h}$. Let $(\ , \ )$ denote the Cartan inner product on $\mathfrak{h}^*$.

We define notation associated to  a subset  $\pi'$  of $\pi$ as follows. Let $\mathfrak{g}_{\pi'}$ be the semisimple Lie subalgebra of $\mathfrak{g}$ generated by $e_i,f_i,h_i$ for all  $\alpha_i\in \pi'$. Write $Q(\pi')$ (resp. $Q^+(\pi')$) for the 
set of linear combinations of elements in $\pi'$ with integer (resp. nonnegative integer) coefficients. In particular, $Q(\pi)$ is just the root lattice  for $\mathfrak{g}$ with respect to the choice of simple roots $\pi$.    Let $\Delta(\pi')$ be the subset of $\Delta$ equal to the root system generated by $\pi'$ with respect to the Cartan inner product $(\ ,\ )$. Write $\mathcal{W}_{\pi'}$  for the subgroup of $\mathcal{W}$ generated by the reflections defined by the simple roots in $\pi'$.  In other words, $\mathcal{W}_{\pi'}$ can be viewed as the Weyl group for the root system $\Delta(\pi')$.  Write $w(\pi')_0$ for the  longest element of $\mathcal{W}_{\pi'}$.  If $\pi'$ is the empty set, we set $w(\pi')_0= 1$.

 Let $P(\pi)$ denote the weight lattice, and let $P^+(\pi)$ denote the subset of $P(\pi)$ consisting of dominant integral weights with respect to the choice of simple roots $\pi$ of the root system $\Delta$.  Recall that $P(\pi)$ has a standard partial order defined by $\beta\geq \gamma$ if and only if $\beta-\gamma\in Q^+(\pi)$.  Similarly, let $P(\pi')$ denote the weight lattice 
 and let  $P^+(\pi')$ denote the set of dominant integral weights associated to the root system $\Delta(\pi')$ with set of simple roots $\pi'$.   Since $\pi\subset \mathfrak{h}^*$, $P(\pi)$ and $Q(\pi)$ are subsets of $\mathfrak{h}^*$.  Thus, given say $\lambda\in P(\pi)$, it makes sense to also consider $\lambda/2\in \mathfrak{h}^*$ as we do at the beginning of Section \ref{section:dual_vermas}.
 
Given   a weight $\beta=\sum_im_i\alpha_i\in \mathfrak{h}^*$ and $\alpha_i\in \pi$, we set ${\rm mult}_{\alpha_i}(\beta) = m_i$. Write ${\rm Supp}(\beta) $ for the subset of simple roots $\alpha_i$ in $\pi$ such that ${\rm mult}_{\alpha_i}(\beta) =m_i\neq 0$.  Define the   height of $\beta$   by  ${\rm ht}(\beta) = \sum_im_i$. Since the set $\pi$ of simple roots form a basis for $\mathfrak{h}^*$, the notion of height works for all elements of $\mathfrak{h}^*$ and hence for all elements in the weight lattice $P(\pi)$ and the root lattice $Q(\pi)$.

\subsection{Cartan Subalgebras and the Normal Real Form}\label{section:realforms}

The normal real form  of  the semisimple Lie algebra $\mathfrak{g}$  is the real Lie subalgebra 
\begin{align*}
\mathfrak{g}_{\mathbb{R}} = \sum_{i=1}^n\mathbb{R}h_i+\sum_{\alpha\in \Delta^+}\mathbb{R} e_{\alpha} + \sum_{\alpha\in \Delta^+}\mathbb{R}f_{-\alpha}.
\end{align*}  Set $\mathfrak{h}_{\mathbb{R}}= \mathfrak{g}_{\mathbb{R}}\cap \mathfrak{h}=\sum_{i=1}^n\mathbb{R}h_i$.

Consider a Lie algebra involution $\theta$ of $\mathfrak{g}$. Conjugating with a Lie algebra automorphism of $\mathfrak{g}$ if necessary, we may assume that 
both $\mathfrak{h}$ and $\mathfrak{h}_{\mathbb{R}}$ are stable under the action of $\theta$, and, moreover, $\theta$ restricts to an involution of $\mathfrak{g}_{\mathbb{R}}$.  The Cartan decomposition associated to $\theta$ is the decomposition of $\mathfrak{g}_{\mathbb{R}}$ into subspaces 
\begin{align*} 
\mathfrak{g}_{\mathbb{R}} = \mathfrak{t}_{\mathbb{R}} \oplus \mathfrak{p}_{\mathbb{R}}
\end{align*}
where 
$\mathfrak{t}_{\mathbb{R}} = \{t\in \mathfrak{g}_{\mathbb{R}} |\ \theta(t) = t\}
$ and 
$\mathfrak{p}_{\mathbb{R}} = \{t\in \mathfrak{g}_{\mathbb{R}}|\ \theta(t) =- t\}.
$ 

A Cartan subalgebra $\mathfrak{s}_{\mathbb{R}}$ of the real semisimple Lie algebra $\mathfrak{g}_{\mathbb{R}}$ is called \emph{standard} if
\begin{align*}
\mathfrak{s}_{\mathbb{R}} = \mathfrak{s}_{\mathbb{R}}^+ \oplus \mathfrak{s}_{\mathbb{R}}^-
\end{align*}
where $\mathfrak{s}_{\mathbb{R}}^+\subseteq \mathfrak{t}_{\mathbb{R}}$
and $\mathfrak{s}_{\mathbb{R}}^-\subseteq \mathfrak{p}_{\mathbb{R}}$.
The space $\mathfrak{s}_{\mathbb{R}}^-$ is called the \emph{vector part} of the
Cartan subalgebra $\mathfrak{s}_{\mathbb{R}}$.  The dimension of $\mathfrak{s}_{\mathbb{R}}^+$ is called the \emph{compact dimension} of $\mathfrak{s}_{\mathbb{R}}$ and
the dimension of  
$\mathfrak{s}_{\mathbb{R}}^-$ is called the \emph{noncompact dimension}
of $\mathfrak{s}_{\mathbb{R}}$.  A standard Cartan subalgebra is called 
\emph{maximally compact} if its compact dimension is the maximum possible value and 
is called \emph{maximally noncompact} if its noncompact dimension is the largest possible value. Maximally noncompact Cartan subalgebras are also called maximally split Cartan subalgebras.  

Let $\mathfrak{a}_{\mathbb{R}}$ be a maximal commutative real Lie algebra that is a subset of $\mathfrak{p}_{\mathbb{R}}$. 
 Note that $\mathfrak{s}_{\mathbb{R}}$ is maximally split if and only if 
\begin{align*}\mathfrak{s}_{\mathbb{R}}^-=\mathfrak{s}_{\mathbb{R}}\cap \mathfrak{p}_{\mathbb{R}} =\mathfrak{s}_{\mathbb{R}}\cap \mathfrak{a}_{\mathbb{R}} = \mathfrak{a}_{\mathbb{R}}.
\end{align*}

Recall that $\theta$ has been adjusted so that the  Cartan subalgebra $\mathfrak{h}$ is stable with respect to $\theta$.  It follows that $\theta$ sends root vectors to root vectors.  Hence $\theta$ induces an involution, which we also call $\theta$, on the root system of $\mathfrak{g}$. In particular,   $\theta(\alpha)$ equals the weight of $\theta(e_{\alpha})$ for all $\alpha\in \Delta$. This involution extends to the lattices $Q(\pi)$ and $P(\pi)$ by insisting that $\theta(\sum_i\eta_i\alpha_i) = \sum_i\eta_i\theta(\alpha_i)$ for all $\eta=\sum_i\eta_i\alpha_i$. We write $Q(\pi)^{\theta}$  (resp. $P(\pi)^{\theta}$) for the set of elements in $Q(\pi)$  (resp. $P(\pi)$) fixed by $\theta$. Let  $\Delta_{\theta}$ denote the subset of the positive roots  $\Delta^+$  such that $\theta(\alpha) = -\alpha$.

We say that two elements $\alpha$ and $\beta$ in $\mathfrak{h}^*$ are orthogonal if they are orthogonal with respect to the Cartan inner product (i.e. $(\alpha,\beta)= 0$).  The following definition introduces a stronger form of orthogonality and specifies  important subsets of positive roots used in the  classification of the Cartan subalgebras of real semisimple Lie algebras. 

\begin{definition}
Let  $\Gamma$ be a subset of the positive roots $\Delta^+$.
\begin{itemize}
\item[(i)] Two positive roots $\alpha$ and $\beta$ are strongly orthogonal if $(\alpha,\beta) = 0$ and 
$\alpha+\beta$ is not a root.
\item[(ii)] $\Gamma$ is a  strongly orthogonal system of positive roots  if every pair  $\alpha,\beta$ in $\Gamma$   is strongly orthogonal.
\item[(iii)]  $\Gamma$ is a strongly orthogonal system of positive roots associated to the involution $\theta$ if $\Gamma\subseteq \Delta_{\theta}$ and $\Gamma$ is a strongly orthogonal system of positive roots.
\end{itemize}
\end{definition}

Given $\beta\in \mathfrak{h}^*$, write  ${\rm Orth}(\beta)$ for the set consisting of  all  $\alpha\in \pi$ such that  $\alpha, \beta$ are  orthogonal.  For $\beta\in \Delta^+$, we write  ${\rm StrOrth}(\beta)$ for the set of all $\alpha\in \pi$ such that  $\alpha, \beta$ are strongly orthogonal. 

Now consider a strongly orthogonal system of positive roots  $\Gamma$.  To simplify notation, we refer to $\Gamma$ as  a \emph{strongly orthogonal system} with the property that it consists of positive roots understood.  Similarly, we call $\Gamma$ a \emph{strongly orthogonal $\theta$-system} if it is a strongly orthogonal system of positive roots associated to the involution $\theta$.
A classification of strongly orthogonal $\theta$-systems  is given in \cite{AK} and \cite{DFG}, Section 4.2. 

The next theorem gives necessary and sufficient conditions 
for two Cartan subalgebras of a normal real form $\mathfrak{g}_{\mathbb{R}}$   to be conjugate to each other.
 
\begin{theorem}\label{theorem:KS} (Kostant-Sugiura's theorem \cite{Kos}, \cite{Si}, see also \cite{OV}, Section 4.7]) A subspace $\mathfrak{b}_{\mathbb{R}}$ of $\mathfrak{a}_{\mathbb{R}}$ is a vector part of a standard Cartan subalgebra $\mathfrak{s}_{\mathbb{R}}$ if and only if
\begin{align*}
\mathfrak{b}_{\mathbb{R}} = \{x\in \mathfrak{a}_{\mathbb{R}}|\ \alpha(x) = 0 {\rm \ for \ all \ }\alpha\in \Gamma\}
\end{align*}
where $\Gamma$ is a strongly orthogonal $\theta$-system.
Moreover, if  $\mathfrak{b}'_{\mathbb{R}}$ is the  vector part of another Cartan subalgebras $\mathfrak{s}'_{\mathbb{R}}$ of $\mathfrak{g}_{\mathbb{R}}$  defined by the strongly orthogonal $\theta$-system 
$\Omega$, then $\mathfrak{s}_{\mathbb{R}}$ and $\mathfrak{s}'_{\mathbb{R}}$ are conjugate if and only if 
there is an element $w$ in the Weyl group $\mathcal{W}$ that commutes with $\theta$ and $w\Omega = \Gamma$.
\end{theorem}

Since $\theta$ restricts to an involution of $\mathfrak{g}_{\mathbb{R}}$, we  can lift the Cartan decomposition to the complex setting, yielding a decomposition 
\begin{align*}
\mathfrak{g} = \mathfrak{g}^{\theta} \oplus \mathfrak{p}
\end{align*}
where
$\mathfrak{g}^{\theta} = \{\theta(t)  = t|\ t\in \mathfrak{g}\} = \mathfrak{t}_{\mathbb{R}}\oplus i\mathfrak{t}_{\mathbb{R}}$
and
$\mathfrak{p} = \{\theta(t) = -t |\ t\in \mathfrak{g}\} = \mathfrak{p}_{\mathbb{R}} \oplus i\mathfrak{p}_{\mathbb{R}}.$  We say  that a Cartan subalgebra  $\mathfrak{s}$ of $\mathfrak{g}$ is \emph{standard} if  it 
is the complexification of a standard Cartan subalgebra $\mathfrak{s}_{\mathbb{R}}$ of $\mathfrak{g}_{\mathbb{R}}$. In particular, a 
Cartan subalgebra $\mathfrak{s}$ of $\mathfrak{g}$ is standard if and only if  
$\mathfrak{s} =( \mathfrak{s}\cap \mathfrak{g}^{\theta})\oplus (\mathfrak{s}\cap \mathfrak{p})$. 

  Since $\theta$ restricts to an involution on $\mathfrak{h}_{\mathbb{R}}$, it follows that $\mathfrak{h}_{\mathbb{R}}$ is a standard Cartan subalgebra of $\mathfrak{g}_{\mathbb{R}}$ and so $\mathfrak{h}$ is a standard Cartan subalgebra of $\mathfrak{g}$.
Write $\mathfrak{h}_{\theta}$ for $\mathfrak{h}\cap \mathfrak{g}^{\theta}$ and $\mathfrak{h}^-$ for $\mathfrak{h}\cap \mathfrak{p}$.   We have    
\begin{align*}
\mathfrak{h}= \mathfrak{h}_{\theta}\oplus \mathfrak{h}^-
\end{align*} 
and  we refer to $\mathfrak{h}^-$ as the vector part of $\mathfrak{h}$  since it is  the  complexification of the vector part  of $\mathfrak{h}_{\mathbb{R}}$.  Set $\mathfrak{a} = \mathfrak{a}_{\mathbb{R}} + i \mathfrak{a}_{\mathbb{R}}$ and note that $\mathfrak{a}$ is a maximal commutative complex Lie algebra subset of $\mathfrak{p}$.

\subsection{Cayley Transforms}\label{section:cayley_transforms}
We review here the process for converting a Cartan subalgebra into a  maximally compact one via Cayley transforms as presented in \cite{K}, Chapter VI, Section 7.  Our setup is as before.  We start with an involution $\theta$ of $\mathfrak{g}$ that restricts to an involution of $\mathfrak{g}_{\mathbb{R}}$ and to an involution of $\mathfrak{h}_{\mathbb{R}}$.  

Let $\alpha\in \Delta_{\theta}$ and so   $\theta(\alpha) = -\alpha$.
Note that $e_{\alpha}, f_{-\alpha}, h_{\alpha}$ are all in $\mathfrak{p}$ and, moreover, $h_{\alpha}\in \mathfrak{h}$ and so $h_{\alpha}\in \mathfrak{a}$.
Applying an automorphsim  to $\mathfrak{g}$ that
fixes elements of $\mathfrak{h}$ and sends root vectors to scalar multiples of themselves if necessary, we may assume that
\begin{align*}
\theta(e_{\alpha}) = f_{-\alpha}
\end{align*}
and thus $e_{\alpha}+f_{-\alpha}\in \mathfrak{g}^{\theta}$
for all $\alpha\in \Delta_{\theta}$.
Given a Lie subalgebra $\mathfrak{s}$ of $\mathfrak{h}$,
let ${\rm ker}_{\alpha}(\mathfrak{s})$ denote the kernel of $\alpha$ restricted to
$\mathfrak{s}$.
Since $\theta(\alpha) = -\alpha$, we have $\mathfrak{h}_{\theta}\subset {\rm ker}_{\alpha}(\mathfrak{h})$. Hence
${\rm ker}_{\alpha}(\mathfrak{h}) = \mathfrak{h}_{\theta}\oplus {\rm ker}_{\alpha}(\mathfrak{a})$
and so
\begin{align*}\mathfrak{h} = \mathbb{C}h_{\alpha} \oplus \mathfrak{h}_{\theta}\oplus
{\rm ker}_{\alpha}(\mathfrak{a}).
\end{align*}
 \begin{definition} Given a positive root $\gamma$, the Cayley transform ${\bf d}_{\gamma}$ (as  
in \cite{K}, Chapter VI, Section 7) is the Lie algebra isomorphism of  $\mathfrak{g}$ defined by 
${\bf d}_{\gamma}  =\Ad ({\rm exp} {{\pi}\over{4}} (f_{-\gamma} - e_{\gamma}))$.
\end{definition}

It follows from the above definition that  the Cayley transform  ${\bf d}_{\gamma}$ sends 
$h_{\gamma} $ to $e_{\gamma}+f_{-\gamma}$.
Also ${\bf d}_{\gamma}(h) = h$
for all $h\in {\rm ker}_{\gamma}(\mathfrak{h})$. 
Therefore, for $\alpha\in \Delta_{\theta}$ we have
\begin{align*}
{\bf d}_{\alpha}(\mathfrak{h}) = \mathbb{C}(e_{\alpha}+f_{-\alpha}) \oplus \mathfrak{h}_{\theta}\oplus
{\rm ker}_{\alpha}(\mathfrak{a}).
\end{align*} Moreover,  this direct sum decomposition restricts to the analogous one for $\mathfrak{h}_{\mathbb{R}}$.
We have
\begin{align*}
{\bf d}_{\alpha}(\mathfrak{h})\cap \mathfrak{g}^{\theta}
= \mathbb{C}(e_{\alpha}+f_{-\alpha}) \oplus \mathfrak{h}_{\theta}
\end{align*}
Hence ${\bf d}_{\alpha}$ transforms $\mathfrak{h}$ into a Cartan subalgebra
${\bf d}_{\alpha}(\mathfrak{h})$ of compact dimension one higher than that of $\mathfrak{h}$. 

We can repeat this process using the root space decomposition 
of $\mathfrak{g}$ with respect to  this new Cartan subalgebra $\mathfrak{s}= {\bf d}_{\alpha}(\mathfrak{h})$.  In particular, one picks a positive root $\beta$ (if possible)
    with respect to $\mathfrak{s}$
 that vanishes on $\mathfrak{s}\cap \mathfrak{g}^{\theta}$.  Applying the Cayley transform defined by $\beta$ yields another Cartan subalgebra where once again the compact dimension increases by one.   
Thus we have the following result as explained in the presentation of Cayley transforms in  \cite{K}, Chapter VI, Section 7  (see also \cite{DFG}).

\begin{lemma}  A standard Cartan subalgebra $\mathfrak{h}$ with respect to the involution $\theta$ can be transformed into a maximally compact Cartan subalgebra by repeated applications of Cayley transforms defined by roots that vanish on the nonvector part of the Cartan subalgebra.  
\end{lemma}

Let $\Gamma$ be a strongly orthogonal $\theta$-system. We can express $\mathfrak{h}$ as
\begin{align*}
\mathfrak{h} = \mathfrak{h}_{\theta} \oplus \sum_{\alpha\in \Gamma}\mathbb{C}h_{\alpha} \oplus \mathfrak{b}
\end{align*}
where $\mathfrak{b}\subseteq \mathfrak{a}$ and $[h,e_{\alpha}] = 0$
(and so $\alpha(h) = 0$)   for all $h\in \mathfrak{b}$ and $\alpha\in \Gamma$.  Note also
that 
 the Cayley transforms ${\bf d}_{\alpha}, \alpha\in \Gamma$ are a commutative set of automorphisms on $\mathfrak{g}$.  Applying all of them to $\mathfrak{g}$ yields a Cartan subalgebra  of the form 
\begin{align*}
 \mathfrak{h}_{\theta} \oplus
\left(\bigoplus_{\alpha\in \Gamma}\mathbb{C}(e_{\alpha}+f_{-\alpha})\right)
\oplus \mathfrak{b}.
\end{align*}
It is straightforward to see that this Cartan subalgebra is the complexification of a real Cartan subalgebra of $\mathfrak{g}_{\mathbb{R}}$, and, moreover, $\mathfrak{b}$ is the complexification of the vector part.

We say that a strongly orthogonal $\theta$-system $\Gamma$ is a maximum strongly orthogonal $\theta$-system if $|\Gamma|\geq |\Omega|$ for all strongly orthogonal $\theta$-systems $\Omega$.

\begin{theorem}\label{theorem:nice_cartans}
Let $\theta$ be an involution of $\mathfrak{g}=\mathfrak{n}^-\oplus \mathfrak{h}\oplus \mathfrak{n}^+$ compatible with the normal real structure such that $\mathfrak{h}$ is a standard Cartan subalgebra with respect to $\theta$ and $\theta(e_{\alpha}) = f_{-\alpha}$ whenever 
$\theta(\alpha) = -\alpha$.
  Then   $\Gamma$ is a maximum strongly orthogonal $\theta$-system   if and only if 
\begin{align*}
\mathfrak{h}_{\theta} \oplus \left(\bigoplus_{\alpha\in \Gamma}\mathbb{C}(e_{\alpha}+f_{-\alpha})\right)
\end{align*} is a Cartan subalgebra of $\mathfrak{g}^{\theta}$.
Moreover, if $\Omega$ is another maximum strongly orthogonal $\theta$-system, then $\Omega= w\Gamma$ for some  element $w$ in $\mathcal{W}$ that commutes with $\theta$.
\end{theorem}

\begin{proof} Consider a strongly orthogonal $\theta$-system $\Gamma$.  Since $\Gamma\subseteq \Delta_{\theta}$, we have $\theta(\alpha) = -\alpha$ and so $h_{\alpha}\in \mathfrak{p}$ for all $\alpha\in \Gamma$.   It follows that 
$\mathfrak{h}_{\Gamma} =\sum_{\alpha\in \Gamma}{\mathbb{C}}h_{\alpha}$ is a subspace of $\mathfrak{a}$.  The fact that $\Gamma$ is a strongly orthogonal system ensures that $\{h_{\alpha}|\ \alpha\in \Gamma\}$ is a linearly independent set and so $\dim(\mathfrak{h}_{\Gamma})=|\Gamma|$.  Moreover the subspace \begin{align*}
\bigoplus_{\alpha\in \Gamma}\mathbb{C}(e_{\alpha}+f_{-\alpha})
 \end{align*}
 of $\mathfrak{g}^{\theta}$ is a commutative Lie algebra and has the same dimension $|\Gamma|$.
 Since 
 \begin{align}\label{Cartan_possibility}
 \mathfrak{h}_{\theta}\oplus \left(\bigoplus_{\alpha\in \Gamma}\mathbb{C}(e_{\alpha}+f_{-\alpha})\right)
 \end{align} is a commutative Lie subalgebra  of $\mathfrak{g}^{\theta}$, and hence a commutative Lie subalgebra of a Cartan subalgebra of $\mathfrak{g}^{\theta}$, we must have
\begin{align}\label{rank-inequality} {\rm rank}(\mathfrak{g}^{\theta})-\dim(\mathfrak{h}_{\theta}) \geq |\Gamma|
\end{align}
where ${\rm rank}(\mathfrak{g}^{\theta})$ is the rank of $\mathfrak{g}^{\theta}$ which is just the dimension of any Cartan subalgebra of $\mathfrak{g}^{\theta}$.
It follows that equality holds in (\ref{rank-inequality}) if and only if (\ref{Cartan_possibility})  is a Cartan subalgebra of $\mathfrak{g}^{\theta}$.  The first part of the theorem then follows once we show that there exists a strongly orthogonal $\theta$-system of maximum possible size $ {\rm rank}(\mathfrak{g}^{\theta})-\dim(\mathfrak{h}_{\theta}) $.

 Let $\mathfrak{s}$ be a maximally compact Cartan subalgebra of $\mathfrak{g}$ with respect to $\theta$.  It follows that $\mathfrak{s}\cap \mathfrak{g}^{\theta}$ is a Cartan subalgebra of $\mathfrak{g}^{\theta}$ and so $\dim(\mathfrak{s}\cap \mathfrak{g}^{\theta}) ={\rm rank}(\mathfrak{g}^{\theta})$.  Let $\mathfrak{b}$ be the complexification of the vector part of  $\mathfrak{s}$ and let $\mathfrak{a}$ be a maximal commutative Lie algebra containing $\mathfrak{b}$ and contained in $\mathfrak{p}$.   
 We can write 
$\mathfrak{h} = \mathfrak{h}_{\theta}\oplus \mathfrak{a}$
and 
$\mathfrak{s} = (\mathfrak{s}\cap \mathfrak{g}^{\theta})\oplus \mathfrak{b}.$

By Theorem \ref{theorem:KS}, 
there exists  a strongly orthogonal $\theta$-system $\Gamma$  such that 
\begin{align}\label{bdecomp}
\mathfrak{b}  = \{x\in \mathfrak{a}| \alpha(x) = 0 {\rm \ for \ all \ }\alpha\in \Gamma\}.
\end{align}  
The decomposition (\ref{bdecomp}) ensures that
$\mathfrak{h}_{\Gamma}\cap \mathfrak{b} = 0$. Hence 
\begin{align}\label{Gamma-inclusion}
\mathfrak{h}_{\Gamma}+\mathfrak{b}= \mathfrak{h}_{\Gamma}\oplus \mathfrak{b} \subseteq \mathfrak{a}.
\end{align}
Let $\mathfrak{b}'$ be a subspace of $\mathfrak{a}$ such that 
\begin{align*}
\mathfrak{h}_{\Gamma}\oplus \mathfrak{b}'\oplus\mathfrak{b}= \mathfrak{a}.
\end{align*}
Recall that $\alpha(h_{\alpha}) >0$ for all $\alpha\in \Gamma$.  Since all pairs of roots in $\Gamma$  are strongly orthogonal, it follows that $\alpha(h_{\beta}) = 0$ for all $\alpha,\beta\in \Gamma$ with $\alpha\neq \beta$.  Thus the map $(\alpha, x) \mapsto \alpha(x)$ defines a nondegenerate pairing on $\mathbb{C}\Gamma\times \mathfrak{h}_{\Gamma}$.  This nondegenerate pairing allows us to adjust the space $\mathfrak{b}'$ if necessary  so that $\alpha(x) = 0$ for all $x\in \mathfrak{b}'$ and $\alpha\in\Gamma$.  But then we see by (\ref{bdecomp}), $\mathfrak{b}'$ must equal zero and the inclusion of (\ref{Gamma-inclusion}) is actually an equality.  Thus we have
\begin{align*}
\mathfrak{h} = \mathfrak{h}_{\theta}\oplus \mathfrak{h}_{\Gamma}\oplus\mathfrak{b}.
\end{align*}
It follows that 
\begin{align*}
{\rm rank}(\mathfrak{g}^{\theta}) = \dim(\mathfrak{s}\cap \mathfrak{g}^{\theta}) = \dim(\mathfrak{h}) - \dim(\mathfrak{b}) = \dim(\mathfrak{h}_{\theta}\oplus\mathfrak{h}_{\Gamma}) = \dim(\mathfrak{h}_{\theta})\oplus |\Gamma|.
\end{align*}
Hence if $\mathfrak{s}$ is a maximally compact Cartan subalgebra with respect to $\theta$ and $\Gamma$ is a strongly orthogonal system satisfying (\ref{bdecomp}) then 
\begin{align*}
|\Gamma| = {\rm rank}(\mathfrak{g}^{\theta}) - \dim(\mathfrak{h}_{\theta}).
\end{align*}
This completes the proof of the first assertion.

The second assertion follows from the fact that all maximally compact Cartan subalgebras with respect to $\theta$ are conjugate (\cite{K}, Proposition 6.61  and Theorem \ref{theorem:KS}).
\end{proof}

\subsection{Maximally Split Involutions}\label{section:max_split_inv}

    In this section, we 
define  the notion of  maximally split involutions  as presented in \cite{L1}, Section 7.  
After reviewing basic properties from Section 7 of \cite{L1}, we  identify a maximum strongly orthogonal $\theta$-system 
  $\Gamma_{\theta}$ for each choice of simple Lie algebra $\mathfrak{g}$ and maximally split involution $\theta$.   
The choice of $\Gamma_{\theta}$ also satisfies certain conditions that will allow us to 
lift the Cartan subalgebra of Theorem \ref{theorem:nice_cartans} to the quantum setting compatible with quantum analogs of $U(\mathfrak{g}^{\theta})$.

\begin{definition}(\cite{L1}, Section 7) The involution 
  $\theta$ is called maximally split with respect to a fixed triangular decomposition  $\mathfrak{g}=\mathfrak{n}^-\oplus \mathfrak{h}\oplus \mathfrak{n}^+$ if it satisfies the following three conditions: 
 \begin{itemize}
\item[(i)] $\theta(\mathfrak{h}) = \mathfrak{h}$
\item[(ii)] if $\theta(h_i) = h_i$, then $\theta(e_i) = e_i$ and $\theta(f_i) = f_i$
\item[(iii)] if $\theta(h_i)\neq h_i$, then $\theta(e_i)$ (resp. $\theta(f_i)$) is a root vector in $\mathfrak{n}^-$ (resp. $\mathfrak{n}^+$).
\end{itemize}
\end{definition}

We call a pair $\mathfrak{g}, \mathfrak{g}^{\theta}$ a maximally split symmetric pair provided that $\mathfrak{g}$ is a  complex semisimple Lie algebra and $\theta$ is a maximally split involution with respect to a known
triangular decomposition $\mathfrak{g} = \mathfrak{n}^-\oplus\mathfrak{h}\oplus\mathfrak{n}^+$. (The triangular decomposition is not included in the notation of such symmetric pairs, but rather understood from context.)  By Section 7 of \cite{L1}, every involution $\theta$  of the semisimple Lie algebra $\mathfrak{g}$ is conjugate to a maximally split involution.   

The pair $\mathfrak{g}, \mathfrak{g}^{\theta}$ is irreducible if $\mathfrak{g}$ cannot be written as the direct sum of two complex semisimple Lie subalgebras which both admit $\theta$ as an involution.  By  \cite{A},
the pair $\mathfrak{g}, \mathfrak{g}^{\theta}$ is irreducible if and only if $\mathfrak{g}$ is simple or 
$\mathfrak{g} = \mathfrak{g}_1\oplus \mathfrak{g}_2$  where $\mathfrak{g}_1$ and $\mathfrak{g}_2$ are isomorphic as Lie algebras.  Moreover, in this latter case, $\theta$ is the involution that sends $e_i$ to $f_i^*$, $f_i$ to $e_i^*$, and $h_i$ to $-h_i^*$ where $e_i,h_i, f_i, i=1,\dots, n$ (resp. $e_i^*, h_i^*, f_i^*, i=1, \dots, n$) are the generators for $\mathfrak{g}_1$ (resp. $\mathfrak{g}_2$).

As in previous sections, we modify $\theta$ via conjugation with an automorphism of $\mathfrak{g}$ preserving root spaces so that  
$\theta(e_{\alpha}) = f_{-\alpha}$
and thus $e_{\alpha}+f_{-\alpha}\in \mathfrak{g}^{\theta}$ 
for all $\alpha\in \Delta_{\theta}$.
The vector space  $\mathfrak{h}\cap \mathfrak{p}$ is a maximal abelian Lie subalgebra of $\mathfrak{p}$ which we can take for $\mathfrak{a}$.  In particular,
$\mathfrak{h} = \mathfrak{h}_{\theta} \oplus {\mathfrak{a}}$ and $\mathfrak{h}$
 is a maximally split Cartan subalgebra with respect to $\theta$.

Let $\pi_{\theta}$ denote the subset of $\pi$ fixed by $\theta$.   
 There exists a permutation $p$ on $\{1, \dots, n\}$ which induces a diagram automorphism (that we also call $p$) on $\pi$ such that
\begin{align*}\theta(-\alpha_i) -\alpha_{p(i)}\in Q^+(\pi_{\theta})
\end{align*}
for all $\alpha_i\in \pi \setminus \pi_{\theta}$.  Note that $p$ restricts to a permutation on $\pi\setminus \pi_{\theta}$.  We  sometimes write $p(\alpha_i)$ for $\alpha_{p(i)}$.    We extend $p$  to a map on $\mathfrak{h}^*$ by setting $p(\sum_im_i\alpha_i) = \sum_im_ip(\alpha_i)$.

\begin{theorem}\label{theorem:cases_take2}
For each maximally split symmetric pair  $\mathfrak{g}, \mathfrak{g}^{\theta}$, there exists a maximum strongly orthogonal $\theta$-system $\Gamma_{\theta}=\{\beta_1,\dots, \beta_m\}$   and a set of simple roots $\{\alpha_{\beta}, \alpha_{\beta}'|\ \beta\in \Gamma_{\theta}\}$ such that for all $\beta \in \Gamma_{\theta}$, we have 
 \begin{align}\label{beta_structure}
 \beta = \alpha_{\beta}, \quad \beta =\alpha'_{\beta} + w_{\beta}\alpha_{\beta},
 \quad{\rm or}\quad \beta = \alpha_{\beta}' + \alpha_{\beta}+ w_{\beta}\alpha_{\beta}
 \end{align}
where $w_{\beta} = w({\rm Supp}(\beta) \setminus\{\alpha_{\beta}', \alpha_{\beta}\})_0$
 and for all $j = 1, \dots, m$, the following conditions hold
 \begin{itemize}
 \item[(i)]  for all  $i>j$,  ${\rm Supp}(\beta_i) \subset {\rm StrOrth}(\beta_j)$
  \item[(ii)]    ${\rm Supp}(\beta_j) \setminus \{\alpha_{\beta_j},\alpha'_{\beta_j}\}\subseteq {\rm StrOrth}(\beta_j)$
  \item[(iii)] $\theta$ restricts to an involution on ${\rm Supp}(\beta_j)$
  \item[(iv)] $-w_{\beta_j}$ restricts to a permutation on $\pi_{\theta}\cap {\rm Supp}(\beta_j)$
  \end{itemize}
  Moreover, we can further break down the expression for   each $\beta$ in $\Gamma_{\theta}$ given in (\ref{beta_structure}) into the following five cases:  \begin{itemize}
  \item[(1)] $\beta = \alpha_{\beta} = \alpha'_{\beta}$.
  \item[(2)] $\beta = \alpha_{\beta} + w_{\beta}\alpha_{\beta}$ and $\alpha_{\beta} = \alpha_{\beta}'$.
  \item[(3)] $\beta =p(\alpha_{\beta}) + w_{\beta}\alpha_{\beta}=w\alpha_{\beta}$ where $w=w({\rm Supp}(\beta)\setminus\{\alpha_{\beta}\})_0$ . In this case, $\Delta({\rm Supp}(\beta))$ is a root system of type A$_r$ where $r=|{\rm Supp}(\beta)|$ and  $\alpha_{\beta}\neq \alpha_{\beta}' = p(\alpha_{\beta})$.
  \item[(4)] $\beta = \alpha'_{\beta}+w_{\beta}\alpha_{\beta}=w\alpha_{\beta}'$ where $w = w({\rm Supp}(\beta) \setminus\{\alpha_{\beta}'\})_0$. In this case, $\Delta({\rm Supp}(\beta))$ is a root system of type B$_r$ where $r = |{\rm Supp}(\beta)|$,  $\{\alpha_{\beta}\}={\rm Supp}(\beta)\setminus \pi_{\theta}$ and $\alpha_{\beta}'$ is the unique short root in ${\rm Supp}(\beta)$.  
 \item[(5)] $\beta = \alpha_{\beta}' + \alpha_{\beta} + w_{\beta}\alpha_{\beta} $ and $\alpha_{\beta}'\neq \alpha_{\beta}$. In this case,  $\alpha_{\beta}'\in \pi_{\theta}$, 
 $\alpha_{\beta}'$ is strongly orthogonal to all $\alpha\in {\rm Supp}(\beta)\setminus\{\alpha_{\beta}',\alpha_{\beta}\}$,  and $\alpha_{\beta}'$ is not strongly orthogonal to $\beta$ but the two roots are orthogonal (i.e. $(\alpha_{\beta}', \beta) = 0$).
 \end{itemize}
\end{theorem}
\begin{proof} Note that it is sufficient to prove the theorem for the case where $\mathfrak{g}, \mathfrak{g}^{\theta}$ is irreducible since, in the general case,  we can simply take a disjoint union of strongly orthogonal systems corresponding to each irreducible component that satisfies the desired properties of the theorem.  Consider the irreducible symmetric pair  $\mathfrak{g}, \mathfrak{g}^{\theta}$ where 
$\mathfrak{g}$ is the direct sum of two isomorphic copies of the same Lie algebra and $\theta$ is the involution descibed above.  Then $\theta$ sends roots of $\mathfrak{g}_1$ to those of $\mathfrak{g}_2$ and so $\Delta_{\theta}$ is the empty set.  It follows that the only choice for $\Gamma_{\theta}$  in this case is  the empty set which clearly satisfies the  conditions of the theorem.  Thus, for the remainder of the proof, we assume that $\mathfrak{g}$ is simple.

 The proof for $\mathfrak{g}$ simple  is by cases.  We use the notation in \cite{H} Section 12.1 for the set of simple roots and orthonormal unit vectors associated to  root systems of various types of simple Lie algebras. 
For each case, we define $\theta$, define $p$ if $p$ does not equal the identity, state $\Delta_{\theta}, \mathfrak{h}_{\theta}$, and  provide a choice
for $\Gamma_{\theta}$  that satisfies the above conditions. If $\beta$ is a simple root, then $\alpha_{\beta} = \beta=\alpha_{\beta}'$. 
For $\beta$ not equal to a simple root in $\pi$, we  specify $\alpha_{\beta}$ and $\alpha'_{\beta}$.  For the latter root, we also state whether $\alpha'_{\beta}\in \pi_{\theta}$ or whether it equals $\alpha_{\beta}$ or $p(\alpha_{\beta})$. The relationship between $\beta$, $\alpha_{\beta}$, and $\alpha'_{\beta}$ can then be read from context. For example, 
$\beta= \alpha'_{\beta} + w_{\beta}\alpha_{\beta}$ if $\alpha'_{\beta} = p(\alpha_{\beta})$ while  
 $\beta= \alpha_{\beta} + w_{\beta}\alpha_{\beta}$ if ${\rm mult}_{\alpha_{\beta}}(\beta) = 2$ and $\alpha'_{\beta} = \alpha_{\beta}$.
 
Recall that the automorphism group of a root system associated to a semisimple Lie algebra is the semidirect product of the Weyl group and the group of diagram automorphisms (\cite{H}, Section 12.2).  Moreover, $-1$ is not in the Weyl group  associated to a simple Lie algebra if and only if this simple Lie algebra is one of the following  types A$_n (n>1)$, D$_n$ (for $n\geq 5$ and $n$ odd), and E$_6$ (\cite{H2}, Section 4). Hence, $-w_{\beta}$ restricts to the identity for all other cases, including D$_n$, $n$ even.  For type A$_n$ and type E$_6$, $-w_{\beta}$ restricts to the unique permutation on the simple roots corresponding to the unique non-identity diagram automorphism.  Using this information,   it is straightforward to check that (iv) holds in all cases.

\medskip
\noindent
{\bf Type AI}: $\mathfrak{g}$ is of type A$_n$, $\theta(\alpha_i) = -\alpha_i$ for $1\leq i\leq n$,  $\Delta_{\theta} = \Delta^+$,  $\mathfrak{h}_{\theta} = 0$, and we may choose
$\Gamma_{\theta} =  \{\beta_1,\dots, \beta_s\}$ where $\beta_j=\alpha_{2j-1}$  for $j=1,\dots, s$,  $s =\lfloor (n+1)/2\rfloor$.

\medskip
\noindent
 {\bf Type AII}: $\mathfrak{g}$ is of type A$_n$ where $n = 2m+1$ is odd and $n\geq 3$, $\theta(\alpha_i) = \alpha_i$ for $i=2j+1$, $0\leq j\leq m$, 
$\theta(\alpha_i) = -\alpha_{i-1} -\alpha_i - \alpha_{i+1}$ for $i = 2j$, $1\leq j\leq m$.  In this case, 
$\mathfrak{h}_{\theta}={\rm span}_{\mathbb{C}}\{h_1, h_3, \dots, h_{2m+1}\}$ 
and $\Delta_{\theta}$ is empty, so $\Gamma_{\theta}$ is also the empty set.  

\medskip
\noindent
{\bf Type AIII/AIV}: $\mathfrak{g}$ is of type A$_n$,
  $r$  is an integer with $1\leq r\leq (n+1)/2$,   $\theta(\alpha_j) = \alpha_j$ for all $r+1\leq j\leq n-r$,
  $\theta(\alpha_i) = -\alpha_{n-i+1}$ for $1\leq i\leq r-1$ and $n-r+2\leq i\leq n$, 
$\theta(\alpha_r)= -\alpha_{r+1} -\alpha_{r+2} - \cdots - \alpha_{n-r} -\alpha_{n-r+1}$ and $\theta(\alpha_{n-r+1}) 
 = -\alpha_{n-r} - \alpha_{n-r-1} -\cdots - \alpha_{r+1} -\alpha_r.$ 
In this case, $p(i) = n-i+1$ for all $i=1,\dots, n$, $\mathfrak{h}_{\theta}={\rm span}_{\mathbb{C}}
\{h_j, h_i-h_{n-i+1}|\ r+1\leq j\leq n-r, 1\leq i\leq r\}$,    
$\Delta_{\theta} = \{\alpha_{k} + \alpha_{k+1} + \cdots + \alpha_{n-k+1}|\ 1\leq k\leq r\}$, and the only choice for $\Gamma_{\theta}$ is $\Gamma_{\theta} =  \Delta_{\theta}$. Hence 
we have $\Gamma_{\theta} = \{\beta_1,\cdots, \beta_r\}$,  where
$\beta_j = \alpha_{j} + \alpha_{j+1} + \cdots + \alpha_{n-j+1}$
  for $j=1, \cdots, r$. For each $j$, 
we may  set $\alpha_{\beta_j}= \alpha_{j}$ and $\alpha'_{\beta_j}=p(\alpha_{\beta_j}) = \alpha_{n-j+1}$.  

\medskip
\noindent
 {\bf Type BI/BII}: $\mathfrak{g}$ is of type B$_n$, $r$ is an integer such that $1\leq r\leq n$, 
$\theta(\alpha_i) = \alpha_i$ for all $r+1\leq i \leq n$, 
 $\theta(\alpha_i) = -\alpha_i$ for all $1\leq i\leq r-1$ if $r<n$ and for all $1\leq i\leq n$ if $r=n$,
and  $\theta(\alpha_r) = -\alpha_r - 2\alpha_{r+1} - \cdots - 2\alpha_{n-1} -2\alpha_n$ if $r<n$.
In this case, $\mathfrak{h}_{\theta}= {\rm span}_{\mathbb{C}}\{h_{r+1},\dots, h_n\}$, $\Delta_{\theta}$ is the subset of the positive roots $\Delta^+$ generated by the set  
  $\{\alpha_1,\alpha_2,\dots, \alpha_{r-1}, 
\alpha_r+\cdots +\alpha_n\}$ which is a root system of type B$_r$, and 
  a choice for $\Gamma_{\theta}$ is  
$\Gamma_{\theta} = \{\beta_1, \cdots, \beta_r\}$
where 
 $\beta_{2j-1} = \alpha_{2j-1}+2\alpha_{2j}+\cdots  + 2\alpha_n$  with $\alpha_{\beta_{2j-1}} = 
\alpha_{2j}=\alpha'_{\beta_{2j-1}}$ and $\beta_{2j} = \alpha_{2j-1}$   for $j = 1, \dots, \lfloor r/2\rfloor$,
and if $r$ is odd then 
$\beta_{r} = \alpha_{r}+\cdots +\alpha_n$ with $\alpha_{\beta_{r}} = \alpha_r$ and $\alpha'_{\beta_r} = \alpha_n$.

\medskip
\noindent
{\bf Type CI}: $\mathfrak{g}$ is of type C$_n$, $\theta(\alpha_i) = -\alpha_i$ for $1\leq i\leq n$.   We have $\Delta_{\theta} =\Delta^+$ and $\mathfrak{h}_{\theta} = 0$.  There is one choice for $\Gamma_{\theta}$, namely 
$\{2\epsilon_1, 2\epsilon_2, \dots, 2\epsilon_n\}$. Rewriting this in terms of the simple roots $\{\alpha_1,\dots, \alpha_n\}$, we have 
$\Gamma_{\theta} = \{\beta_1,\dots, \beta_n\}$ where 
$\beta_j= 2\alpha_j + 2\alpha_{j+1} + \cdots 2\alpha_{n-1}+\alpha_n$
for $ 1\leq j\leq n-1$ and $\beta_n  = \alpha_n$.
We have $\alpha_{\beta_j} =\alpha_j=\alpha'_{\beta_j}$    for $j=1,\dots,n$.  

\medskip
\noindent
{\bf Type CII}, {Case 1}: $\mathfrak{g}$ is of type C$_n$, $r$ is  even and satisfies $1\leq r \leq n-1$, 
$\theta(\alpha_{j}) = \alpha_j$ for $j=1, 3, \cdots, r-1$ and $j=r+1, r+2,\dots, n$,  
$\theta(\alpha_i) = -\alpha_{i-1} -\alpha_i-\alpha_{i+1}$ for $i =2,4, \dots, r-2$,
and  $\theta(\alpha_r) = -\alpha_{r-1}-\alpha_r-2\alpha_{r+1}-\cdots -2\alpha_{n-1} -\alpha_n$.
We have 
$\Delta_{\theta}= \{\alpha_j + 2\alpha_{j+1} + \cdots + 2\alpha_{n-1} + \alpha_n|\ j= 1,3,\dots, r-1\}$
and  
 $\mathfrak{h}_{\theta} = {\rm span}_{\mathbb{C}}\{h_1,h_3, \dots, h_{r-1}, h_{r+1}, h_{r+2}, \dots, h_n\}.$
The only choice for  $\Gamma_{\theta}$ is $\Delta_{\theta}$ and so $\Gamma_{\theta} = \{\beta_1,\beta_2, \dots, \beta_{r/2}\}$ where $\beta_j = \alpha_{2j-1} + 2\alpha_{2j} + \cdots + 2\alpha_{n-1}+\alpha_n$ with  $\alpha'_{\beta_j} = \alpha_{2j-1}$ and $\alpha_{\beta_j} = \alpha_{2j}$ for $j = 1,2,\dots, r/2$.

\medskip
\noindent
{\bf Type CII}, {Case 2}: $\mathfrak{g}$ is of type C$_n$, $n$ is even, 
$\theta(\alpha_{j}) = \alpha_j$ for $j=1, 3, \cdots, n-1$, 
 $\theta(\alpha_i) = -\alpha_{i-1} -\alpha_i-\alpha_{i+1}$ for $i =2,4, \dots, n-2$,
and  $\theta(\alpha_n) = -2\alpha_{n-1}-\alpha_n$.
One checks that  
$\Delta_{\theta} = \{\alpha_j + 2\alpha_{j+1} + \cdots + 2\alpha_{n-1} + \alpha_n|\ j= 1,3,\dots, n-3\}
\cup\{\alpha_{n-1}+\alpha_n\}$
 and 
 $\mathfrak{h}_{\theta}= {\rm span}_{\mathbb{C}}\{h_1,h_3, \dots, h_{n-1}\}$.
Moreover, the only choice for $\Gamma_{\theta}$ is $\Delta_{\theta}$.
Hence we have $\Gamma_{\theta}=\{\beta_1,\beta_2,\dots, \beta_t\}$ 
where  $t = n/2$, 
$\beta_j = \alpha_{2j -1}+ 2\alpha_{2j} + \cdots + 2\alpha_{n-1} + \alpha_n$ for  $j=1,\dots, t-1$
and  
$\beta_t = \alpha_{n-1}+\alpha_n$.  We have 
$\alpha'_{\beta_j} = \alpha_{2j-1}$ and $\alpha_{\beta_j} = \alpha_{2j}$ for $j=1,\dots, t$.

\medskip
\noindent
{\bf Type DI/DII}, {Case 1}: $\mathfrak{g}$ is of type D$_n$, $r$ is an integer such that $1\leq r\leq n-2$, 
$\theta(\alpha_i) = \alpha_i$ for $i = r+1, r+2, \dots, n$,
 $\theta(\alpha_i) = -\alpha_i$ for $i=1, \dots, r-1$,
and $\theta(\alpha_r) = -\alpha_r-2\alpha_{r+1}\cdots - 2\alpha_{n-2}-\alpha_{n-1}-\alpha_n$.
In this case, $\mathfrak{h}_{\theta} = {\rm span}_{\mathbb{C}} \{h_i|\ r+1\leq i \leq n\}$ and  $\Delta_{\theta}$ is the set of all positive roots 
contained in the span of  the union of the following two sets:
$ \{\alpha_j|\ 1\leq j\leq r-1\}$
and 
$\{\alpha_i + 2\alpha_{i+1}+2\alpha_{i+2}+ \cdots + 2\alpha_{n-2}+\alpha_{n-1}+\alpha_n|\  1\leq i\leq r-1\}.$   
A choice for $\Gamma_{\theta}$ is $\Gamma_{\theta} = \{\beta_1,\dots, \beta_{2t}\}$ where $\beta_{2j} = \alpha_{2j-e}$ 
and $\beta_{2j-1} = 
 \alpha_{2j-e}+2\alpha_{2j-e+1} + \cdots + 2\alpha_{n-2} + \alpha_{n-1} + \alpha_n$
 with $\alpha_{\beta_{2j-1}}= \alpha_{2j-e+1}=\alpha'_{\beta_{2j-1}}$ for 
 $j=1,\dots, t$ where $t =\lfloor{r/2}\rfloor$, $e = 0$ if $r$ is odd and $e=1$ if $r$ is even.

\medskip
\noindent
{\bf Type DI}, {Case 2}:  $\mathfrak{g}$ is of type D$_n$ with 
$n\geq 4$ (case $n=3$ is the same as type AI), $\theta(\alpha_i) = -\alpha_i$ for $1\leq i\leq n-2$,
$\theta(\alpha_n)= -\alpha_{n-1}$ and $\theta(\alpha_{n-1}) = -\alpha_n$.
In this case, $p(i) = i$ for $i=1, \dots, n-2$, $p(n-1) = n$, $p(n) = n-1$, 
$\mathfrak{h}_{\theta} = \mathbb{C}(h_{n-1}-h_n)$ and $\Delta_{\theta}$ is the set consisting of those positive roots in the span of the union of the following 
two sets: 
$ \{\alpha_i|\ i= 1,\dots, n-2\}$
and 
$\{\alpha_i+2\alpha_{i+1}+\cdots + 2\alpha_{n-2} + \alpha_{n-1}+\alpha_n|\ 1\leq i\leq n-2\}.$
A choice for  $\Gamma_{\theta}$  is $ \{\beta_1,\dots,\beta_{2t}\}$ where 
 $\beta_{2j} = \alpha_{2j-e}$ 
and $\beta_{2j-1} = 
 \alpha_{2j-e}+2\alpha_{2j-e+1} + \cdots + 2\alpha_{n-2} + \alpha_{n-1} + \alpha_n$
  for 
 $j=1,\dots, t$ where $t =\lfloor{(n-1)/2}\rfloor$, $e = 1$ if $n$ is odd and $e=0$ if $n$ is even.  Note that $\beta_{n-3+e} = \alpha_{n-2}+\alpha_{n-1}+\alpha_n$.
 We have $\alpha_{\beta_{2j-1}}= \alpha_{2j-e+1}=\alpha'_{\beta_{2j-1}}$ for 
 $j=1,\dots, t-1$,  $\alpha_{\beta_{n-3+e}} = \alpha_{n-1}$ and $\alpha'_{\beta_{n-3+e}} =p(\alpha_{n-1}) =\alpha_{n}.$

\medskip
\noindent
{\bf Type DI}, {Case 3}: $\mathfrak{g}$ is of type D$_n$, $\theta(\alpha_i) =-\alpha_i$ all $i=1,\dots, n$. Here, we have $\mathfrak{h}_{\theta} =0$
and $\Gamma_{\theta} = \{\beta_1,\dots, \beta_{2t}\}$ where $t= \lfloor{n}/2\rfloor$,
 $\beta_{2j} = \alpha_{2j-e}$ and $\beta_{2j-1} = \alpha_{2j-e}+2\alpha_{2j-e+1} + \cdots + 2\alpha_{n-2} + \alpha_{n-1} +\alpha_n$ for $j=1,\dots, t-1$,  $\beta_{2t} = \alpha_{n-1},$ and $\beta_{2t-1} = \alpha_{n}$ 
 where $e=0$ if $n$ is odd and $e=1$ if $n$ is even.  We have 
$\alpha_{\beta_{2j-1}}= \alpha_{2j-e+1}=\alpha'_{\beta_{2j-1}}$ for 
 $j=1,\dots, t-1$.

\medskip
\noindent
{\bf Type DIII}, {Case 1}: $\mathfrak{g}$ is of type D$_n$, $n$  is even, 
$\theta(\alpha_i) = \alpha_i$ for $i=1,3, \dots, n-1$,
$\theta(\alpha_i) = -\alpha_{i-1} -\alpha_i-\alpha_{i+1}$ for $i=2,4, \dots, n-2$,
and $\theta(\alpha_n) = -\alpha_n$.
In this case,  $
  \Delta_{\theta} = \{\alpha_{k} + 2\alpha_{k+1}  +\cdots
+ 2\alpha_{n-2} +  \alpha_{n-1} + \alpha_n|\
 k = 1,3, 5, \cdots, n-3\} \cup \{\alpha_n\}.
$ and $\mathfrak{h}_{\theta} = {\rm span}_{\mathbb{C}}\{h_1,h_3,\dots, h_{n-1}\}.$ It follows that $\Gamma_{\theta} = \Delta_{\theta}= \{\beta_1,\dots, \beta_{n/2}\}$ where 
$\beta_j = \alpha_{2j-1} + 2\alpha_{2j} + \cdots + 2\alpha_{n-2} + \alpha_{n-1} + \alpha_n$ for $j=1, \dots, n/2 -1$ and $\beta_{n/2} = \alpha_n$. 
We have $\alpha_{\beta_j} = \alpha_{2j}=\alpha'_{\beta_j}$   for $j = 1,\dots, n/2-1$.

\medskip
\noindent
{\bf Type  DIII}, {Case 2}: $\mathfrak{g}$ is of type D$_n$, $n$ is odd, 
$\theta(\alpha_i) = \alpha_i$ for $i \in \{1,3,5,\dots, n-2\}$,
 $\theta(\alpha_i) = -\alpha_{i-1}-\alpha_i-\alpha_{i+1}$ for $i=2,4,6,\dots, n-3$,
 $\theta(\alpha_{n-1}) = -\alpha_{n-2}-\alpha_n$,
and $\theta(\alpha_{n}) = -\alpha_{n-2}-\alpha_{n-1}$.
In this case, $p(i) = i$ for $i=1, \dots, n-2$, $p(n-1) = n$, $p(n) = n-1$. 
We have $\Delta_{\theta} = \{
 \alpha_{k} + 2\alpha_{k+1}  +\cdots 
+ 2\alpha_{n-2} +  \alpha_{n-1} + \alpha_n|\ 
 k = 1,3, 5, \cdots, n-2\}$
and  $\mathfrak{h}_{\theta}={\rm span}_{\mathbb{C}}(\{h_i|\ i=1,3, \dots, n-2\}\cup \{h_{n-1}- h_{n}\}).$
It follows that $\Gamma_{\theta} = \Delta_{\theta}$ and 
so  $\Gamma_{\theta} = \{\beta_1,\dots, \beta_{t}\}$ where $t = (n-1)/2$, $\beta_j = \alpha_{2j-1} + 2 \alpha_{2j} + \cdots + 2\alpha_{n-2} + \alpha_{n-1} + \alpha_n$ 
with $\alpha_{\beta_j}  = \alpha_{2j}=\alpha'_{\beta_j}$  for $j= 1, \dots, t - 1$ and $\beta_{t} = \alpha_{n-2} + \alpha_{n-1} + \alpha_n$ with $\alpha_{\beta_t} = \alpha_{n-1}, \alpha_{\beta_t}' =p( \alpha_{\beta_{n-1}}) = \alpha_{n}$.  

\medskip
\noindent
{\bf Type EI, EV, EVIII}: $\mathfrak{g}$ is of type E6, E7, E8 respectively and 
$\theta(\alpha_i) = -\alpha_i$ for all $i$.  Hence 
$\Delta_{\theta} = \Delta^+$
and $\mathfrak{h}_{\theta} = 0$. 
Set 
\begin{align*}
\gamma_1 = \epsilon_8 + \epsilon_7 &= 2\alpha_1 + 3\alpha_2 + 4\alpha_3 + 6\alpha_4 + 5\alpha_5 + 4\alpha_6 + 3\alpha_7 + 2\alpha_8,\quad \alpha_{\gamma_1} = \alpha_8 \cr\gamma_2 = \epsilon_8 -\epsilon_7&= 2\alpha_1 + 2\alpha_2+ 3\alpha_3 + 4\alpha_4 + 3\alpha_5 + 2\alpha_6 +\alpha_7,\quad \alpha_{\gamma_2} =\alpha_1\cr
 \gamma_3 = \epsilon_6 + \epsilon_5 &=  \alpha_7 + 2\alpha_6 + 2\alpha_5 + 2\alpha_4 + \alpha_3 +\alpha_2,
\quad \alpha_{\gamma_3} = \alpha_6\cr
 \gamma_4= \epsilon_6 -\epsilon_5 &= \alpha_7\cr
\gamma_5 = \epsilon_4 + \epsilon_3 &=  \alpha_5 + 2\alpha_4 + \alpha_3 + \alpha_2,
\quad\alpha_{\gamma_5} = \alpha_4\cr
\gamma_6 = \epsilon_4 -\epsilon_3& =\alpha_5\cr
\gamma_7 = \epsilon_2 - \epsilon_1 &=\alpha_3\cr
\gamma_8 = \epsilon_2+\epsilon_1 &=\alpha_2
\end{align*}

\medskip
\noindent
Case 1:   $\mathfrak{g}$ is type E8.  A choice for $\Gamma_{\theta}$ is the 
set $\{\beta_1,\dots, \beta_8\}$ where $\beta_i = \gamma_i$ for $i=1, 2, \dots, 8$ and $\alpha_{\beta_j} =\alpha_{\gamma_j} =\alpha'_{\beta_j}$ for $j=1,2,3,5$.

\medskip
\noindent
Case 2:  $\mathfrak{g}$ is type E7. A choice for $\Gamma_{\theta}$ is  the set $\{\beta_1,\dots, \beta_7\}$ where $\beta_i = \gamma_{i+1}$ for $i=1, 2, \dots, 7$ and $\alpha_{\beta_j} =\alpha_{\gamma_{j+1}} =\alpha'_{\beta_j}$ for $j=1,2,4$.

\medskip
\noindent
Case 3: 
$\mathfrak{g}$ is of type  E6.  A choice for $\Gamma_{\theta}$ is $\{\beta_1,\beta_2,\beta_3,\beta_4\}$ where   $\beta_i = \gamma_{i+4}$ for $i=1, 2,3, 4$ and $\alpha_{\beta_j} =\alpha_{\gamma_{j+4}} =\alpha'_{\beta_j}$ for $j=1$.

\medskip
\noindent
{\bf Type EII}: $\mathfrak{g}$ is of type E6,  $\theta(\alpha_1) = -\alpha_6, \theta(\alpha_2) = -\alpha_2, \theta(\alpha_3) = -\alpha_5, \theta(\alpha_4) =- \alpha_4, \theta(\alpha_5) =- \alpha_3, \theta(\alpha_6) =- \alpha_1$. In this case, $p(1) = 6, p(2)  =2, p(3) = 5,p(4) = 4,  p(5) = 3, p(6) = 1$,
$\mathfrak{h}_{\theta} = {\rm span}_{\mathbb{C}} \{h_1-h_6, h_3-h_5\}$ and 
$\Delta_{\theta} = \Delta^+\cap {\rm span}\{\alpha_1+\alpha_6, \alpha_3+\alpha_5, \alpha_4\}$. Here, we may choose $\Gamma_{\theta} = \{\beta_1, \beta_2,\beta_3\}$ where 
$\beta_1 = \alpha_1+ \alpha_3 + \alpha_4 + \alpha_5 + \alpha_6$ with $\alpha_{\beta_1} = \alpha_1$  and $\alpha'_{\beta_1} = p(\alpha_{\beta_1}) = \alpha_6$,
 $\beta_2 = \alpha_3 + \alpha_4 + \alpha_5$ with $\alpha_{\beta_2} = \alpha_3$ and $\alpha'_{\beta_2} = p(\alpha_{\beta_2}) = \alpha_5$, and
$\beta_3 = \alpha_4$.

\medskip
\noindent
{\bf Type EIII}: $\mathfrak{g}$ is of type E6, $\theta(\alpha_j) = \alpha_j$ for $j = 3,4,5,$
$\theta(\alpha_1) = -\alpha_3-\alpha_4-\alpha_5- \alpha_6, \theta(\alpha_6) = \alpha_5-\alpha_4-\alpha_3-\alpha_1,$ and $\theta(\alpha_2) =-2\alpha_4-\alpha_3-\alpha_5 -\alpha_2$.
In this case $p(1) = 6, p(2)  =2, p(3) = 5,p(4) = 4,  p(5) = 3, p(6) = 1$, 
$\mathfrak{h}_{\theta} = {\rm span}_{\mathbb{C}}\{h_3, h_4, h_5, h_1-h_6\}$ and
$\Delta_{\theta}=\{\beta_1,\beta_2\} =\Gamma_{\theta}$
where 
$\beta_1 =  \alpha_1  + 2\alpha_2 + 2\alpha_3 + 3\alpha_4 + 2\alpha_5 + \alpha_6$, $\beta_2 = \alpha_1 +\alpha_3 +\alpha_4+\alpha_5+\alpha_6$, 
 $\alpha_{\beta_1} = \alpha_2=\alpha'_{\beta_1}$, $\alpha_{\beta_2} = \alpha_1$, and $\alpha'_{\beta_2} = p(\alpha_{\beta_2}) = \alpha_6$.
 
\medskip
\noindent
{\bf Type EIV}: $\mathfrak{g}$ is of type E6, $\theta(\alpha_1) = -\alpha_1-2\alpha_3-2\alpha_4-\alpha_2-\alpha_5$, $\theta(\alpha_6) = -\alpha_6-2\alpha_5-2\alpha_4-\alpha_2-\alpha_3$, $\theta(\alpha_j) = \alpha_j$ for $j=2,3,4,5$. 
We have $\mathfrak{h}_{\theta} = {\rm span}_{\mathbb{C}}\{h_2,h_3,h_4,h_5\}$. 
One checks that $\Delta_{\theta}$ must be in the span of 
the two vectors $2\alpha_1+ \alpha_2 + 2\alpha_3 + 2\alpha_4 + \alpha_5$ and 
$2\alpha_6 + \alpha_2+ 2\alpha_4 + 2\alpha_5 + \alpha_3$.  Since this span does not contain any positive roots, we get $\Delta_{\theta} $ is the empty set and so $\Gamma_{\theta}$ is also the empty set.

\medskip
\noindent
{\bf Type EVI}: $\mathfrak{g}$ is of type E7, $\theta(\alpha_j) = \alpha_j$ for $j=2,5,7$, $\theta(\alpha_6) = -\alpha_6-\alpha_5-\alpha_7$, $\theta(\alpha_4) = -\alpha_2-\alpha_5-\alpha_4$, $\theta(\alpha_i) = -\alpha_i$ for $i=1,3$.
In this case, $\Delta_{\theta} = {\rm span}\{\alpha_1, \gamma_2, \gamma_3, \gamma_5, \gamma_7\}\cap \Delta^+$ and 
 $\mathfrak{h}_{\theta} = {\rm span}_{\mathbb{C}}\{h_2,h_5,h_7\}$.  A choice for  $\Gamma_{\theta}$ is $\{\beta_1, \beta_2, \beta_3, \beta_4\}$ where $\beta_1 = \gamma_2$ with $\alpha_{\beta_1} = \alpha_{\gamma_{2}}$ and $\beta_{j} = \gamma_{2j-1}$ with 
 $\alpha_{\beta_j} = \alpha_{\gamma_{2j-1}}$ for $j=2,3,4$.

\medskip
\noindent
{\bf Type EVII}: $\mathfrak{g}$ is of type E7, $\theta(\alpha_i) = \alpha_i$ for $i=2,3,4,5$, $\theta(\alpha_1) = -\alpha_1-2\alpha_3-2\alpha_4-\alpha_2-\alpha_5$, $\theta(\alpha_6) = -\alpha_6-2\alpha_5-2\alpha_4-\alpha_2-\alpha_3$, 
$\theta(\alpha_7) = -\alpha_7$. We have $\mathfrak{h}_{\theta} = {\rm span}_{\mathbb{C}}\{h_2,h_3,h_4,h_5\}$ and $\Delta_{\theta} = \Gamma_{\theta} =\{\beta_1,\beta_2,\beta_3\}$ where following the notation for types EI, EV, EVIII, we have $\beta_1 = \gamma_2$ with $\alpha_{\beta_1} = \alpha_1=\alpha'_{\beta_1}$, $\beta_2 = \gamma_3$ with $\alpha_{\beta_2} = \alpha_6=\alpha'_{\beta_2}$, and $\beta_3 =\gamma_4 = \alpha_7$.

\medskip
\noindent
{\bf Type EIX}: $\mathfrak{g}$ is of type E8, $\theta(\alpha_i) = \alpha_i$ for $i=2,3,4,5$, $\theta(\alpha_i) = -\alpha_i$ for $i = 7, 8$, 
$\theta(\alpha_1) = -\alpha_1-2\alpha_3-2\alpha_4-\alpha_2-\alpha_5$,
$\theta(\alpha_6) = - \alpha_6 - 2\alpha_5-2\alpha_4-\alpha_2-\alpha_3$.
The vector space  $\mathfrak{h}_{\theta} = {\rm span}_{\mathbb{C}}\{h_2,h_3,h_4,h_5\}$ and 
$\Delta_{\theta} =  {\rm span}\{\beta_1, \beta_2, \beta_3, \beta_4\}=\Gamma_{\theta}$
where following the notation for types EI, EV, EVIII, we have $\beta_1 = \gamma_1$ with $\alpha_{\beta_1} = \alpha_8=\alpha_{\beta_1}' $, $\beta_2 = \gamma_2$ with $\alpha_{\beta_1} = \alpha_1=\alpha'_{\beta_1}$, $\beta_3 = \gamma_3$ with $\alpha_{\beta_2} = \alpha_6=\alpha'_{\beta_2}$, and $\beta_4 =\gamma_4 = \alpha_7$.

\medskip
\noindent
{\bf Type FI}: $\mathfrak{g}$ is of type F4, 
$\theta(\alpha_i) = -\alpha_i$, $i = 1,2,3,4$. In this case, $\mathfrak{h}_{\theta}= 0$ and $\Delta_{\theta} = \Delta^+$. A choice for $\Gamma_{\theta} = \{\beta_1,\beta_2,\beta_3,\beta_4\}$ where 
 $\beta_1 = \epsilon_1 + \epsilon_2 = 
2\alpha_1 + 3\alpha_2 + 4\alpha_3 +2\alpha_4 $ with $\alpha_{\beta_1} = \alpha_1= \alpha'_{\beta_1}$,
 $\beta_2 = \epsilon_1 -\epsilon_2 = \alpha_2 + 2\alpha_3 + 2\alpha_4$ with $ \alpha_{\beta_2} = \alpha_4=\alpha'_{\beta_2}$,
 $\beta_3 = \epsilon_3 + \epsilon_4 =  \alpha_2 + 2\alpha_3$ with $ \alpha_{\beta_3} = \alpha_3=\alpha'_{\beta_3}$, and
 $\beta_4 = \epsilon_3 - \epsilon_4 = \alpha_2$.

\medskip
\noindent
{\bf Type FII}: $\mathfrak{g}$ is of type F4, $\theta(\alpha_i) = \alpha_i$ for $i=1,2,3$, $\theta(\alpha_4) = -\alpha_4-3\alpha_3-2\alpha_2-\alpha_1$. In this 
case, $\mathfrak{h}_{\theta} = {\rm span}\{h_1,h_2,h_3\}$ and 
$\Delta_{\theta} = \{\beta_1\} = \Gamma_{\theta}$ where $\beta_1 = \epsilon_1 = 2\alpha_4 + 2\alpha_2 + \alpha_1 + 3\alpha_3$ and $\alpha_{\beta_1} = \alpha_4=\alpha'_{\beta_1} $.

\medskip
\noindent
{\bf Type G}: $\mathfrak{g}$ is of type G2, $\theta(\alpha_i) = -\alpha_i$ for $i=1,2$. In this case, $\mathfrak{h}_{\theta} = 0$ and $\Delta_{\theta} = \Delta^+$. There are three possibilities for a maximum strongly orthogonal $\theta$-system  inside $\Delta_{\theta}$, namely
\begin{itemize}
\item[(1)] $\{-2\epsilon_1 +\epsilon_2 + \epsilon_3, -\epsilon_2 + \epsilon_3\} = \{\alpha_2, \alpha_2 + 2\alpha_1\}$
\item[(2)] $\{-2\epsilon_2 +\epsilon_1 + \epsilon_3, -\epsilon_1 + \epsilon_3\} = \{\alpha_2 + 3\alpha_1, \alpha_2 + \alpha_1\}$ 
\item[(3)] $\{2\epsilon_3 - \epsilon_1 - \epsilon_2, \epsilon_1 - \epsilon_2\} =\{2\alpha_2 + 3\alpha_1, \alpha_1\}$. 
\end{itemize}
The first  subset  is a choice for $\Gamma_{\theta}$.  So we have $\Gamma_{\theta}=\{\beta_1,\beta_2\}$ where 
$\beta_1 =\alpha_2+2 \alpha_1$, $\alpha_{\beta_1} = \alpha_1=\alpha'_{\beta_1}$ and $\beta_2 = \alpha_2$.
\end{proof}

The next remark expands on properties of the roots described in the above theorem in light of the specific examples of $\Gamma_{\theta}$ given for each symmetric pair. 

\begin{remark}\label{remark:cases}  Let $\mathfrak{g}, \mathfrak{g}^{\theta}$ be an irreducible maximally split symmetric pair and let $\Gamma_{\theta}$ be a  maximum strongly orthogonal $\theta$-system satisfying the conditions of Theorem \ref{theorem:cases_take2}.  Let $\beta\in \Gamma_{\theta}$ and set $r = |{\rm Supp}(\beta)|$.
Then in the explicit versions of $\Gamma_{\theta}$ given in the proof of the theorem, we see that 
\begin{itemize}
\item if  $\beta$ satisfies (3) then $\Delta({\rm Supp}(\beta))$ is a root system  of type A$_r$ where $r = |{\rm Supp}(\beta)|$ and $\theta$ restricts to an involution on $\mathfrak{g}_{{\rm Supp}(\beta)}$ of type AIII/AIV.  
 \item  if $\beta$ satisfies (4)  then $\beta$ is  the final root in $\Gamma_{\theta}$ where either $\mathfrak{g}, \mathfrak{g}^{\theta}$ is of  type BI/BII and $r$ is odd or  $\mathfrak{g}, \mathfrak{g}^{\theta}$ is of type CII Case 2.
\item  if $\beta$ satisfies (5) then
 $\Delta({\rm Supp}(\beta))$ is a root system of type C$_r$, $\theta$ restricts to an involution of type CII on $\mathfrak{g}_{{\rm Supp}(\beta)}$, and $\mathfrak{g}, \mathfrak{g}^{\theta}$ is of type CII. 
 \end{itemize}
 \end{remark}

\section{Quantized Enveloping Algebras}\label{section:qea}
In this section, we turn our attention to the quantized enveloping algebra of the complex semisimple Lie algebra  $\mathfrak{g}$ with a chosen triangular decomposition $\mathfrak{g} = \mathfrak{n}^-\oplus \mathfrak{h}\oplus \mathfrak{n}^+$. After setting basic notation, we describe the structure of the locally finite part and a realization of dual Verma modules as subspaces of the quantized enveloping algebra. We also discuss Lusztig's automorphisms and Chevalley antiautomorphisms as well as the concept of specialization.  

\subsection{Basic Notation: the Quantum Case} \label{section:basic}
Let $q$ be an indeterminate.  The quantized enveloping algebra,  $U_q(\mathfrak{g})$, is the Hopf algebra over $\mathbb{C}(q)$ generated as an algebra  by $E_i,F_i,K_i^{\pm 1}, i=1,\dots, n$ with relations and Hopf algebra structure as defined in \cite{L1} Section 1 (with $x_i, y_i, t_i$ replaced by $E_i$, $F_i$, $K_i$) or \cite{Ko} Section 3.1. Sometimes we write $E_{\alpha}$ (resp. $F_{-\alpha}$) for $E_j$ (resp. $F_j$) where $\alpha$ is the simple root $\alpha_j$. Note that a counit $\epsilon$, which is an algebra homomorphism from $U_q(\mathfrak{g})$ to the scalars,  is part of the Hopf structure.  In our setting, the counit is defined by  $\epsilon(E_i) = \epsilon(F_i) = 0$ and $\epsilon(K_i^{\pm 1}) = 1$.

  Let $U^+$ be the $\mathbb{C}(q)$ subalgebra of $U_q(\mathfrak{g})$ generated by $E_i, i=1,\dots, n$.  Similarly, let $U^-$  be the $\mathbb{C}(q)$ subalgebra of $U_q(\mathfrak{g})$ generated by 
$F_i, i=1,\dots, n$.   Let $\mathcal{T}$ denote the group generated by the $K_1^{\pm 1}, \dots, K_n^{\pm 1}$ and write $U^0$ for the group algebra over $\mathbb{C}(q)$ generated by $\mathcal{T}$.  The algebra $U_q(\mathfrak{g})$ admits a triangular decomposition which is an isomorphism of vector spaces
\begin{align}\label{triangle}
U_q(\mathfrak{g}) \cong U^-\otimes U^0\otimes U^+
\end{align}
via the multiplication map.
Let $G^+$ (resp. $G^-$) be the subalgebra of $U_q(\mathfrak{g})$ generated by  $E_iK_i^{-1}, i=1,\dots, n$ (resp. $F_iK_i, i = 1, \dots, n$).  One gets similar triangular decompositions upon replacing $U^+$ by $G^+$ or $U^-$ by $G^-$.

 The (left) adjoint action of $U_q(\mathfrak{g})$, which makes $U_q(\mathfrak{g})$ into a left module over itself, comes from the Hopf algebra structure and is defined on generators by 
 \begin{itemize}
\item$ (\ad E_i)a = E_i a - K_iaK_i^{-1} E_i$
\item $ (\ad K_i) a = K_iaK_i^{-1}$ and $ (\ad K_i^{-1}) a = K_i^{-1}aK_i$
\item $ (\ad F_i) a = F_iaK_i  - aF_iK_i = F_iK_iK_i^{-1}a K_i - aF_iK_i$
\end{itemize}
for all $a\in U_q(\mathfrak{g})$ and each $i=1,\dots, n$. 
Temporarily fix $i$.  It is straightforward to check that $({\rm ad}\ E_i)a = \epsilon(E_i)a = 0$, $({\rm ad}\ F_i)a = \epsilon(F_i)a = 0$ and $({\rm ad}\ K_i^{\pm 1}) a = \epsilon(K_i^{\pm 1})a = a$ if and only if $[E_i, a] = [F_i,a] = [K_i^{\pm 1},a]=0$.

Set $A = \mathbb{C}[q,q^{-1}]_{(q-1)}$ and let $\hat U$ be the $A$ subalgebra generated by 
\begin{align*}E_i, F_i, K_i^{\pm 1}, {{(K_i-1)}\over{(q-1)}}
\end{align*} for $i=1,\dots, n$.  The specialization of $U_q(\mathfrak{g})$ at $q=1$ is the $\mathbb{C}$ algebra
$\hat U\otimes_A \mathbb{C}$.  It is well known that $\hat U\otimes_A \mathbb{C}$ is isomorphic to $U(\mathfrak{g})$ (see for example \cite{L1}).  Given a subalgebra $S$ of $U_q(\mathfrak{g})$, we say that $S$ specializes to the subalgebra $\bar S$ of $U(\mathfrak{g})$ provided the image of $S\cap \hat U$ in $\hat U\otimes_A \mathbb{C}$ is $\bar S$.

Given a $U_q(\mathfrak{g})$-module $V$, the $\lambda$ weight space of $V$, denoted $V_{\lambda}$, is the subspace consisting of vectors $v$ such that $K_iv = q^{(\alpha_i, \lambda)}v$ for each $i$.  Elements of $V_{\lambda}$ are called weight vectors with respect to this module structure. We say that $u\in U_q(\mathfrak{g})$ is a weight vector of weight $\lambda$ if it is a weight vector of weight $\lambda$ with respect to the adjoint action.  In other words,  $K_iuK_i^{-1} = q^{(\alpha_i,\lambda)}u$ for all $i=1, \dots, n$.  Given any $({\rm ad}\ U_q(\mathfrak{g}))$-module $M$  and a weight $\lambda\in Q(\pi)$, let  $M_{\lambda}$ denote the $\lambda$ weight space of $M$ with respect to the adjoint action.  

We can write  $U_q(\mathfrak{g})$ as direct sum of weight spaces with respect to the adjoint action where all weights are in $Q(\pi)$.  Suppoe that   $a=\sum_{\lambda}a_{\lambda}\in U_q(\mathfrak{g})$ where each $a_{\lambda}\in (U_q(\mathfrak{g}))_{\lambda}$.  We say that $v$ is a weight summand of $a$ with respect to the expression of $a$ as a sum of weight vectors (or, simply, $v$ is a weight summand of $a$)  if $v=a_{\lambda}$ for some $\lambda$.  Note that $U^-,U^+, G^-, G^+$ can all also be written as a sum of weight spaces with respect to the adjoint action.  Using (\ref{triangle}), we have the following direct sum:
\begin{align}\label{biweight} U_q(\mathfrak{g}) = \bigoplus_{\lambda, \mu\in Q^+(\pi)}G^-_{-\lambda}U^0U^+_{\mu}.
\end{align}
Note that
\begin{align*} G^-_{-\lambda}U^0 = U^-_{-\lambda} U^0
\end{align*}
for each $\lambda\in Q^+(\pi)$ and so we can replace $G^-_{-\lambda}$ with $U^-_{-\lambda}$ in the above direct sum decomposition. Similarly, we could replace $U^+_{\mu}$ with $G^+_{\mu}$ for $\mu\in Q^+(\pi)$.   We refer to $G^-_{-\lambda}U^0U^+_{\mu}$ as a biweight subspace of $U_q(\mathfrak{g})$. An element $v\in G^-_{-\lambda}U^0U^+_{\mu}$ is said to have biweight $(-\lambda, \mu)$. Note that a vector of biweight $(-\lambda, \mu)$ has weight $\mu-\lambda$.  We call  
$a=\sum_{\lambda, \mu\in Q^+(\pi)}a_{\lambda, \mu}$ where each $a_{\lambda, \mu}\in G^-_{-\lambda}U^0U^+_{\mu}$ the expansion of $a$ as a sum of biweight vectors and call $v$ a biweight summand of $a$ if $v=a_{\lambda, \mu}$ for some $\lambda, \mu\in Q(\pi)$.  

We have the following coarse version of (\ref{biweight}):
\begin{align}\label{lweight} U_q(\mathfrak{g}) = \bigoplus_{\lambda\in Q^+(\pi)}G^-_{-\lambda}U^0U^+.
\end{align}
Suppose that $\sum_{\lambda, \mu\in Q^+(\pi)}a_{\lambda, \mu}$ is the expansion of $a$ as a sum of biweight vectors.  Then 
\begin{align*}
a = \sum_{\lambda \in Q^+(\pi)} a_{\lambda}    {\rm\ where\ } a_{\lambda}= \sum_{\mu\in Q^+(\pi)}a_{\lambda, \mu}
\end{align*}
is the expression of $a$ as a sum using the direct sum decomposition (\ref{lweight}).  Given an element $u$ in $G^-_{-\zeta}U^0U^+$, set $l$-weight$(u) = -\zeta$. We say that $u$ is a minimal $l$-weight summand of $a$ if $u \in G^-_{-\zeta}U^0U^+$ and $a-u\in \sum_{\lambda \not\geq\zeta}G^-_{-\lambda}U^0U^+$.

Set  $K_{\xi}=K_1^{\xi_1}\cdots K_n^{\xi_n}$ for each $\xi=\sum_i\xi_i\alpha_i$ in the root lattice $Q(\pi)$. Note that the map $\xi\mapsto K_{\xi}$ defines an isomorphism from $Q(\pi)$ to $\mathcal{T}$.  We can enlarge the group $\mathcal{T}$ to a group $\check{\mathcal{T}}$ so that this isomorphism extends to an isomorphism from the weight lattice $P(\pi)$ to $\check{\mathcal{T}}$.  In particular,    let $r$ be a positive integer such that $r\eta \in Q(\pi)$ for each $\eta \in P(\pi)$ and  form the free abelian group of rank $n$ generated by  $ K_1^{1/r}, \dots, K_n^{1/r}$ where $(K_j^{1/r})^r = K_j$ for each $j$.  Note that $\mathcal{T}$ embeds in  this  group in the obvious way.  Given $\eta\in P(\pi)$, we set $K_{\eta} = K_1^{\eta_1/r}\cdots K_n^{\eta_n/r}$ where 
$\eta =( \eta_1/r)\alpha_1 + \cdots (\eta_n/r)\alpha_n$ and each $\eta_j$ is an integer.  The group $\check {\mathcal{T}}$ is the subgroup of  $\langle K_1^{1/r}, \dots, K_n^{1/r}\rangle$ 
consisting of all elements of the form $K_{\xi}$ where  $\xi\in P(\pi)$ and thus the map $\xi\mapsto K_{\xi}$ defines an isomorphism from $P(\pi)$ to $\check{\mathcal{T}}$.

Let   $N$ be a positive integer so that $(\mu,\alpha_i) \in {{1}\over{N}}\mathbb{Z}$ for all $\mu \in P(\pi)$ and $\alpha_i\in \pi$.  We will often work with elements in the simply connected quantized enveloping algebra $\check U$, a Hopf algebra containing $U_q(\mathfrak{g})$. As an algebra, $\check U$ is generated over $\mathbb{C}(q^{1/N})$ by $U_q(\mathfrak{g})$ and $\check{\mathcal{T}}$ such that $({\rm ad}\ K_{\mu})v = K_{\mu}vK_{\mu}^{-1} = q^{(\mu, \lambda)}v$ for all $v\in \check U_{\lambda}$.  Note that  the counit $\epsilon$ satisfies $\epsilon(K_{\mu}) = 1$ for all $K_{\mu}\in \check{\mathcal{T}}$.  (For more details, see \cite{Jo},  Section 3.2.10).  Thus elements of $\check U$ are linear combination of  terms of the form $uK_{\mu}$ where $u\in U_q(\mathfrak{g})$ and $\mu\in P(\pi)$.
Note that (\ref{biweight}) and thus the notion of biweight extends to $\check U$.

 Given $\lambda \in P(\pi)$, write  $M(\lambda)$ for the  universal highest weight $U_q(\mathfrak{g})$-module (i.e. the Verma module) generated by a highest weight vector of weight $\lambda$ and let $L(\lambda)$ be the simple highest weight $U_q(\mathfrak{g})$-module of highest weight $\lambda$.  
  When $\lambda$ is in $P^+(\pi)$, the simple module $L(\lambda)$ is finite-dimensional.

Suppose that $\pi'$ is a subset  of $\pi$. Set $U_{\pi'}$ equal to the subalgebra of $U_q(\mathfrak{g})$ generated by $E_i, K_i^{\pm 1}, F_i$ for all $\alpha_i\in \pi'$. Note that we can identify
$U_{\pi'}$ with the quantized enveloping algebra  $U_q(\mathfrak{g}_{\pi'})$.
Given a subalgebra $M$ of $U_q(\mathfrak{g})$, set $M_{\pi'} = U_{\pi'}\cap M$.  For example, 
the subalgebra $U^+_{\pi'}$  of  $U^+$ is the algebra generated by $E_i$ with $\alpha_i\in {\pi'}$. Given $\lambda \in P^+(\pi')$, write $L_{\pi'}(\lambda)$ for the simple highest weight $(U_{\pi'}U^0)$-module of highest weight $\lambda$.  Given a subset $S$ of $U_q(\mathfrak{g})$, we write $S^{U_{\pi'}}$ for the set 
\begin{align*}\{s\in S|\ [a,s] =0{\rm  \ for \ all \ }a\in U_{\pi'}\}.
\end{align*}  By the definition of  the adjoint action and the counit $\epsilon$, we see that $S^{U_{\pi'}}$ also equals the set 
\begin{align*}\{s\in S|\ ({\rm ad}\ a)s=\epsilon(a)s{\rm  \ for \ all \ }a\in U_{\pi'}\}.
\end{align*} We refer to $S^{U_{\pi'}}$ as a trivial $({\rm ad}\ U_{\pi'})$-module.  
The set $S^{U_{\pi'}}$ is also called the centralizer of $U_{\pi'}$ inside of $S$.

We call a lowest weight vector $f$ with respect to the action of $U_q(\mathfrak{g})$ (i.e. $F_if = 0$ for all $\alpha_i\in \pi$)  inside a module for this algebra  \emph{a $U_q(\mathfrak{g})$ lowest weight vector}.  In the special case where the action is the adjoint action, we call $f$ an \emph{an $({\rm ad}\ U_q(\mathfrak{g}))$ lowest weight vector}. We use the same language for subsets  $\pi' $ of $ \pi$ with $U_q(\mathfrak{g})$ replaced by  $U_{\pi'}$.

 \subsection{Dual Vermas and the Locally Finite Part}\label{section:dual_vermas}

We recall here the construction of a number of $(\ad U_q(\mathfrak{g}))$-modules.  (A good reference with more details is \cite{Jo}, Section 7.1.)  Given $\lambda\in P(\pi)$ we can realize  $G^-$ as an $({\rm ad}\ U_q(\mathfrak{g}))$-module, which we refer to as $G^-{(\lambda)}$, via the action 
\begin{align*}
({\rm ad}_{\lambda} E_i) 1 = 0\quad ({\rm ad}_{\lambda} K_i)1 = q^{ {{1}\over{2}}(\lambda, \alpha_i)}1 \quad ({\rm ad}_{\lambda} F_i)1 = (1-q^{(\lambda, \alpha_i)})F_iK_i.
\end{align*}
for $i=1,\dots, n$.
(Here we are using a slightly different Hopf structure and so a slightly different adjoint action which is why we have $G^-$ instead of $U^-$ as in \cite{Jo}, Section 7.1.)
With respect to this action, $G^-{(\lambda)}$ is isomorphic to the  dual module $\delta M(\lambda/2)$ of $M(\lambda/2)$.  Thus $G^-{(\lambda)}$ contains a unique simple submodule isomorphic to $ L(\lambda/2)$.  Note that we are following the construction in \cite{Jo}, Section 7.1 in choosing $\lambda \in P(\pi)$ as opposed to $\lambda\in 2P(\pi)$.  This means that there is a possibility that $\lambda/2$ is not in $P(\pi)$. However,  $\lambda/2$ is still a well-defined element of  $\mathfrak{h}^*$ (as explained in Section \ref{section:basic_classical}) and the construction works fine when this happens. 

  If $\gamma$ is dominant integral, then  $L(\gamma)$ is finite-dimensional and hence can also be viewed as a module generated by a lowest weight vector. In this case, a lowest weight  generating vector for $L(\gamma)$ has weight $w_0\gamma$. Note that $1$ has weight $\gamma$ viewed as an element in the module $G^-{(2\gamma)}$.  Moreover, the $U_q(\mathfrak{g})$ lowest weight vector of the submodule of $G^-{(2\gamma)}$ isomorphic to $L(\gamma)$ is of the form $({\rm ad}_{2\gamma} y)1$ where  $y\in U^-_{-\gamma'}$ and  $\gamma' = \gamma-w_0\gamma$.  Now consider an arbitrary weight $\gamma$ and a weight $\beta\in Q^+(\pi)$.  The dimension of the subspace of $\delta M(\gamma)$ spanned by vectors of weight greater than or equal to $\gamma -\beta$ is finite-dimensional.  Therefore any $U_q(\mathfrak{g})$ lowest weight vector of $\delta M(\gamma)$ generates a finite-dimensional submodule of $\delta M(\gamma)$. Thus $\delta M(\gamma)$, and hence $G^-{(2\gamma)}$,  contains a $U_q(\mathfrak{g})$ lowest weight vector of weight $\gamma -\beta$ if and only if $\gamma$ is dominant integral and $\beta = \gamma-w_0\gamma$.

The degree function defined by 
 \begin{itemize}
\item $\deg_{\mathcal{F}} (F_iK_i) = \deg_{\mathcal{F}} (E_i) = 0$ for all $i$.
\item $\deg_{\mathcal{F}} K_i = -1$ and $\deg_{\mathcal{F}} K_i^{-1} = 1$ for all $i$.
\end{itemize}
yields a  $({\rm ad}\ U_q(\mathfrak{g}))$-invariant filtration $\mathcal{F}$ on $\check U$ such that $G^-K_{-\lambda}$ is a homogenous subspace of degree ${\rm ht}(\lambda)$.  
This degree function, and hence this filtration, extends to $\check U$ by setting $\deg_{\mathcal{F}} K_{\xi}=-{\rm ht}(\xi)$ and $\deg_{\mathcal{F}} K^{-1}_{\xi}={\rm ht}(\xi)$
for all $\xi\in P(\pi)$. (This is just the ad-invariant filtration on $U_q(\mathfrak{g})$ as in  \cite{Jo} Section 7.1, though with a slightly different form of the adjoint action.)

Now \begin{align*}
({\rm ad}\ F_i )K_{-\lambda} = F_iK_{-\lambda} K_i - K_{-\lambda}F_iK_i = (1-q^{(\lambda, \alpha_i)})F_iK_iK_{-\lambda}.
\end{align*}
In particular, $G^-K_{-\lambda}$  and $G^-{(\lambda)}$ are isomorphic as $(U^-)$-modules where the action on the first module is via  $({\rm ad}\ U^-)$ and the action on the second module is via $({\rm ad}_{\lambda}U^-)$.   
 This is not true if we replace $U^-$ with $U^+$.  Instead, the twisted adjoint action of $U^+$ corresponds to a version of the graded adjoint action with respect to the $({\rm ad}\ U_q(\mathfrak{g}))$-invariant filtration $\mathcal{F}$.  More precisely, given $f\in G^-$ and $i\in \{1, \dots, n\}$, there exists $f'$ and $f''$ in $G^-$ and a scalar $c$ such that 
 we can write 
 \begin{align*}
 ({\rm ad}\ E_i)fK_{-\lambda} &= f'K_{-\lambda} + f''K_{-\lambda}K_i^2 + cfK_{-\lambda}E_i 
  \end{align*}
 where here we are considering $fK_{-\lambda}$ as an element of $U_q(\mathfrak{g})$ and using the ordinary adjoint action of $E_i$.  We obtain a graded action of $({\rm ad}\ E_i)$ on $G^-K_{-\lambda}$ by dropping the lower degree term with respect to $\mathcal{F}$ (i.e. $f''K_{-\lambda}K_i^2$) as well as the contribution from $\sum_iU_q(\mathfrak{g})E_i$ (i.e. $cfK_{-\lambda}E_i$).  Thus we may equip  $G^-K_{-\lambda}$ with a graded adjoint action of $U^+$.  Using this graded adjoint action of $U^+$ combined with the ordinary action of $({\rm ad}\ U^-)$ and setting $({\rm ad}\ K_{\beta})\cdot K_{-\lambda} =  q^{{{1}\over{2}}(\lambda, \beta)}K_{-\lambda}$  for all $\beta\in Q^+(\pi)$ makes  $G^-K_{-\lambda}$ into an $({\rm ad}\ U_q(\mathfrak{g})$)-module.  (We will often refer to this action as the graded adjoint action, but still use the same notation for the adjoint action, where the fact that $G^-K_{-\lambda}$ is being viewed as a module rather than a subset of the larger module $U_q(\mathfrak{g})$ is understood from context.)
 Note that in the above description of $({\rm ad}\ E_i)fK_{-\lambda}$, the term $f'$ is the top degree term of $({\rm ad}\ E_i)f$ (more precisely,   $({\rm ad}\ E_i)f  = f' +  $  terms of degree strictly less than $0$) and  is independent of $\lambda$.  Hence, we may view the graded adjoint action of $U^+$ on $G^-K_{-\lambda}$ as independent of $\lambda$.  Moreover, $G^-K_{-\lambda}$ and $G^-(\lambda)$ are isomorphic $U_q(\mathfrak{g})$-modules where the action on the  first module is the graded adjoint and on the second module is the twisted adjoint action.  (The fact that we can make $G^-K_{-\lambda}$ into an $({\rm ad}\ U_q(\mathfrak{g}))$-module in this way  is implicit in  \cite{Jo}, Lemma 7.1.1.)

Now consider any weight $\lambda$. By the above discussion,  $G^-K_{-\lambda}$ and  $G^-{(\lambda)}$ are isomorphic as $(\ad U^-)$-modules and the latter is  isomorphic to $\delta M(\lambda/2)$. It follows that    $G^-K_{-\lambda}$ contains a nonzero $({\rm ad}\ U_q(\mathfrak{g}))$ lowest weight vector $gK_{-\lambda}$ where $g$ is of weight $-\beta$ if and only if  $\lambda=2\gamma$ and $\beta = \gamma-w_0\gamma$ for some $\gamma\in P^+(\pi)$.    Moreover, this lowest weight vector is unique (up to scalar multiple) and is contained in $(\ad U^-)K_{-\lambda}$.

Let $\pi'$ be a subset of $\pi$.  For each $\lambda\in P(\pi)$, write $\tilde \lambda$ for the restriction of the weight $\lambda$ to an element of $P(\pi')$, the weight lattice for $\pi'$. In particular, $\tilde\lambda$ is the element in $P(\pi')$ that satisfies $(\lambda, \alpha_i) = (\tilde\lambda, \alpha_i)$ for all $\alpha_i\in \pi'$.  Here, the subset $\pi'$ is understood from context. Note that 
$G^-_{\pi'}K_{-\lambda}$ and $G^-_{\pi'}K_{-\tilde\lambda}$ are isomorphic as $(\ad U_{\pi'})$-modules. Hence $G^-_{\pi'}K_{-\lambda}$ contains a  nonzero $({\rm ad}\ U_{\pi'})$ lowest weight vector, say $gK_{-\lambda}$,   if and only if $\tilde\lambda= 2\tilde\gamma$ for some choice of $\tilde\gamma\in P^+(\pi')$.  Moreover, the weight of $gK_{-\lambda}$  (viewed as an element in the $({\rm ad}\ U_q(\mathfrak{g}))$-module $G^-K_{-\lambda}$) is 
$w_{\pi'}\tilde\gamma$ where $w_{\pi'} = w(\pi')_0$ and so the weight of $g$ (considered as an element of the $({\rm ad}\ U_q(\mathfrak{g}))$-module $U_q(\mathfrak{g})$)  is 
$-\tilde\gamma+w_{\pi'}\tilde\gamma$.  Since $(\gamma-\tilde\gamma, \alpha) = 0$ for all $\alpha\in \pi'$, we have $w_{\pi'}(\gamma-\tilde\gamma) = \gamma-\tilde\gamma$.  Hence 
\begin{align*}
-\gamma + w_{\pi'}\gamma =-(\gamma-\tilde\gamma)-\tilde\gamma+ (\gamma-\tilde\gamma ) + w_{\pi'}\tilde\gamma=-\tilde\gamma+w_{\pi'}\tilde\gamma
\end{align*}
Thus the weight of $g$ is $-\gamma+w_{\pi'}\gamma$.  We may similarly analyze $({\rm ad}\ U_{\pi'})$  lowest weight vectors in $G^-K_{-\lambda}$ where we view this space as an $({\rm ad}\ U_{\pi'})$-module.  In particular, if $g\in G^-$ and $gK_{-\lambda}$ is an $({\rm ad}\ U_{\pi'})$  lowest weight vector then $g$ generates a finite-dimensional $({\rm ad}_{\lambda} U_{\pi'})$-submodule of $G^-(\lambda)$ and the restriction $-\tilde\mu$ of the weight $-\mu$ of $g$ (considered as an element of  $G^-(\lambda)$) to $\pi'$ must satisfy $\tilde\mu\in P^+(\pi')$.  Note that  $gK_{-\lambda}$ generates a finite dimensional $U_q(\mathfrak{g})$-submodule of $G^-K_{-\lambda}$ when we use the graded adjoint action.  Thus if $b\in U_q(\mathfrak{g})$ and $({\rm ad}_{\lambda}b)g=0$, we have 
\begin{align}\label{lower_degree}
\deg_{\mathcal{F}}\left(({\rm ad}\ b) gK_{-\lambda}\right)\leq  \deg_{\mathcal{F}}(K_{-\lambda})={\rm ht}(\lambda)
\end{align}

Let $F(\check U)$ denote the locally finite part of $\check U$ with respect to the (left) adjoint action. By \cite{Jo} Section 7.1,  $F(\check U)$ admits a direct sum decomposition \begin{align*}F(\check U) =  \bigoplus_{\gamma\in P^+(\pi)} ({\rm ad} \ U_q(\mathfrak{g}))K_{-2\gamma}
\end{align*}
as $({\rm ad}\ U_q(\mathfrak{g}))$-modules. The subspace of $G^-K_{-2\gamma}$ corresponding to the finite-dimensional $({\rm ad}_{2\gamma} U_q(\mathfrak{g}))$-submodule of $G^-{(2\gamma)}$ is the subspace $({\rm ad}\ U^-)K_{-2\gamma}$ of $G^-K_{-2\gamma}$.

Let $\pi'$ be a subset of $\pi$ as above. Set $F_{\pi'}(\check U)$ equal to the locally finite part of $\check U$ with respect to the action of $({\rm ad}\ U_{\pi'})$.   By the defining relations for $U_q(\mathfrak{g})$ (see for example  Lemma 4.5 of \cite{JL}) and the definition of the adjoint action, we see that $E_j, F_jK_j \in F_{\pi'}(\check U)$ for all $\alpha_j\notin \pi'$ and $F(\check U) \subseteq F_{\pi'}(\check U)$.  Also, $K_{-2\gamma}\in F_{\pi'}(\check U)$ for all $\gamma$ that restricts to a dominant integral weight in $P^+(\pi')$.  

Since $F_{\pi'}(\check U)$ contains all the finite-dimensional $({\rm ad}\ U_{\pi'})$-submodules of $\check U$, it follows that  the subspace of elements in $F_{\pi'}(\check U)$ that admit a trivial $({\rm ad}\ U_{\pi'})$ action is the same as the subspace of elements in $\check U$ admitting such a trivial action.  In particular, we have 
\begin{align*}
F_{\pi'}(\check U)^{U_{\pi'}} &= \{a\in F_{\pi'}(\check U)|\ ({\rm ad}\ u)a = \epsilon(u)a {\rm \ for \ all \ }u\in U_{\pi'}\} \cr &= \{a\in \check U|\ ({\rm ad}\ u)a = \epsilon(u)a
{\rm \ for \ all \ }u\in U_{\pi'}\} = \check U^{U_{\pi'}}.
\end{align*}
It follows from the discussion at the end of Section \ref{section:basic}, that  this subspace agrees with the centralizer of $U_{\pi'}$ inside $\check U$.  In particular, \begin{align*}\check U^{U_{\pi'}} &= \{a\in \check U| ua-au = 0 {\rm \ for \ all \ }u\in U_{\pi'}\}\cr &= \{a\in \check U|\ ({\rm ad}\ u)a = \epsilon(u)a {\rm \ for \ all \ }u\in U_{\pi'}\}.
\end{align*}

 \subsection{Lusztig's Automorphisms}\label{section:lusztig}
For each $j\in \{1,\dots, n\}$, let $T_j$ denote Lusztig's automorphism associated to the simple root $\alpha_j$ (\cite{Lu}, Section 37.1, see also  Section 3.4 of \cite{Ko}). For each $w$ in $\mathcal{W}$, the automorphism $T_w$ is defined to be the composition 
\begin{align*}
T_w = T_{r_1}\cdots T_{r_m}
\end{align*} where 
$s_{r_1}\cdots s_{r_m}$ is a reduced expression for $w$ and each $s_{r_j}$ is the reflection defined by the simple root $\alpha_{r_j}$.   By \cite{Lu} (or \cite{Ja}, Section 8.18) $T_w$ is independent of reduced expression for $w$.
One of the properties satisfied by the $T_j$, $j=1,\dots, n$ is 
$T_j^{-1} = \sigma T_j \sigma$ where $\sigma$ is the $\mathbb{C}(q)$ algebra antiautomorphism defined by  
\begin{align}\label{sigma}
\sigma(E_j) = E_j\quad \sigma(F_j) = F_j\quad \sigma(K_j) = K_j^{-1} \quad \sigma(K_j^{-1}) = K_j.
\end{align}
 The following result of \cite{Ko} relates Lusztig's automorphisms and submodules of $U_q(\mathfrak{g})$ with respect to the action of $({\rm ad}\ U_{\pi'})$ where $\pi'\subset \pi$.

\begin{lemma}(\cite{Ko}, Lemma 3.5)\label{lemma:lowest_weight}  Let $\pi'$ be a subset of the simple roots $\pi$ and let $\alpha_i$ be a simple root not in $\pi'$. Set $w_{\pi'} = w(\pi')_0$.
\begin{itemize}
\item[(i)] The subspace ${(\ad U_{\pi'})}(E_i)$ of $U^+$  is a simple finite-dimensional $(\ad U_{\pi'})$-submodule of $U_q(\mathfrak{g})$ with highest weight vector $T_{w_{\pi'}}(E_i)$ and lowest weight vector $E_i$.  
\item[(ii)] The subspace ${(\ad U_{\pi'})}(F_iK_i)$ of $G^-$  is a simple finite-dimensional $(\ad U_{\pi'})$-submodule of $U_q(\mathfrak{g})$ with highest weight vector $F_iK_i$ and lowest weight vector $T_{w_{\pi'}}^{-1}(F_iK_i)$.
\end{itemize}
\end{lemma}

Assume $\pi'$, $\alpha_i$, and $w_{\pi'}$ are chosen as in the above lemma. Note that  $\widetilde{w_{\pi'}\alpha_i}$, the restriction of the weight $w_{\pi'}\alpha_i$  with respect to $\pi'$, is an element in $P^+(\pi')$.   In particular, as $(U_{\pi'})$-modules
$(\ad U_{\pi'})(E_i) $ is isomorphic to $L_{\pi'}(w_{\pi'}\alpha_i)$, the finite-dimensional simple $(U_{\pi'}U^0)$-module of highest weight $w_{\pi'}\alpha_i$.  

\subsection{The Quantum Chevalley Antiautomorphism}\label{section:chevalley}
Let $U_{\mathbb{R}(q)} $ be the $\mathbb{R}(q)$ subalgebra of $U_q(\mathfrak{g})$ generated by $E_i, F_i, K_i,$ and $K_i^{-1}$.  
Let $\kappa$ denote  the algebra antiautomorphism of $U_{\mathbb{R}(q)}$
defined by 
\begin{align*}
\kappa(E_i) = F_iK_i \quad \kappa(F_i) = K_i^{-1}E_i\quad \kappa(K) = K
\end{align*}
for all $i = 1, \dots, n$ and $K \in \mathcal{T}$.   Note that $\kappa^2 = {\rm Id}$ and so $\kappa$ is an involution.  We define the notion of complex conjugate for $\mathbb{C}(q)$ by setting $\bar s = s$ for all $s\in \mathbb{C}$ and $\bar q = q$ and insisting that conjugation is an algebra automorphism.  Extend $\kappa$ to a conjugate linear antiautomorphism of $U_q(\mathfrak{g})$ by setting  $\kappa(au) = \bar a \kappa(u)$ 
for all $a\in \mathbb{C}(q)$ and $u\in U_{\mathbb{R}(q)}$.  We refer to $\kappa$ as the quantum Chevalley antiautomorphism as in \cite{L1}.  
Note that 
\begin{align*}
\kappa(({\rm ad}\  E_j) b ) = \kappa(E_jb - K_jbK_j^{-1}E_j) = \kappa(b) F_jK_j - F_j\kappa(b) K_j = -({\rm ad}\ F_j)\kappa(b).
\end{align*}
It follows that 
\begin{align}\label{kappaeq}
\kappa(({\rm ad}\ E_{i_1}\cdots E_{i_m}) E_kK )= (-1)^m({\rm ad}\ F_{i_1}\cdots F_{i_m}) (K F_kK_k)
\end{align}
for all $K\in \mathcal{T}$.

We recall information about unitary modules for quantum groups as presented in \cite{L1}.  Let $C$ be a  subalgebra of $U_q(\mathfrak{g})$ such that $\kappa(C) = C$.  
Given a $C$-module $M$, a mapping $S_M$ from $M\times M$ to ${\mathbb{R}}(q)$ is called sesquilinear if $S_M$ is linear in the first variable and conjugate linear in the second.  We say that the $C$-module $M$ is  \emph{unitary} provided $M$ admits a sesquilinear form $S_M$ that satisfies 
\begin{itemize}
\item $S_M(av,w) =S_M(v,\kappa(a)w)$ for all $a\in C$ and $v,w\in M$
\item $S_M(v,v)$ is a positive element of $\mathbb{R}(q)$ for each nonzero vector $v\in M$
\item $S_M(v,w) = \overline{S_M(w,v)} $ for all $v$ and $w$ in $M$.
\end{itemize} 

\begin{theorem} \label{theorem:unitary} Let $C$ be a subalgebra of $U_q(\mathfrak{g})$ such that $\kappa(C) = C$. Every finite-dimensional unitary $C$-module can be written as a direct sum of finite-dimensional simple $C$-modules.  \end{theorem}
\begin{proof} This is just Theorem 2.4 and Corollary 2.5 of \cite{L1} except that here we only assume that $C$ is a $\kappa$ invariant subalgebra of $U_q(\mathfrak{g})$ where  as in \cite{L1}, $C$ is additionally assumed to be a (left) coideal subalgebra.  Nevertheless, the proofs in \cite{L1} do not use the coideal assumption and thus apply to this more general setting.
\end{proof}

The next result is an immediate consequence of Theorem \ref{theorem:unitary} and the fact that a commutative polynomial ring $H$ acts semisimply on a $H$-module $V$ if and only if $V$ can be written as a direct sum of eigenspaces with respect to the action of $H$.

\begin{corollary} \label{cor:unitary1} Let $C$ be a subalgebra of $U_q(\mathfrak{g})$ such that $\kappa(C) = C$. Let $H$ be a commutative subalgebra of $C$ such that 
\begin{itemize}
\item $\kappa(H)=H$
\item $H\cap U^0$ is the Laurent polynomial ring generated by $H\cap \mathcal{T}$
\item  $H$ is isomorphic to a polynomial ring over $H\cap U^0$
\end{itemize} Then  any finite-dimensional unitary $C$-module can be written as a direct sum of eigenspaces with respect to the action of $H$.
\end{corollary}

Recall that every finite-dimensional simple $U_q(\mathfrak{g})$-module $L(\lambda)$ admits a sesquilinear form $S$ defined by $S(av, bv) = S(v, \mathcal{P}_{HC}(\kappa(a)b)v)$ where $v$ is a highest weight generating vector for $L(\lambda)$ and $\mathcal{P}_{HC}$ is the Harish-Chandra projection as defined in  (1.14) of \cite{L1}.  Thus, the following result is an immediate consequence of 
Theorem \ref{theorem:unitary}.
\begin{corollary} \label{theorem:unitary2} Let $C$ be a subalgebra of $U_q(\mathfrak{g})$ such that $\kappa(C) = C$. The finite-dimensional simple $U_q(\mathfrak{g})$-module $L(\lambda)$ admits a semisimple $C$ action for all $\lambda\in P^+(\pi)$.  
\end{corollary}

The next corollary is a special case of Corollary \ref{cor:unitary1}.

\begin{corollary} \label{Haction2} Let $H$ be a commutative subalgebra of $U_q(\mathfrak{g})$ such that $\kappa(H) = H$.  Assume further that $H\cap U^0$ is the Laurent polynomial ring generated by $H\cap\mathcal{T}$ and that $H$ is isomorphic to a polynomial ring over $H\cap U^0$.  Then any  finite-dimensional $U_q(\mathfrak{g})$-module can be written as a direct sum of eigenspaces with respect to the action of $H$.
\end{corollary}

 \section{Quantum Symmetric Pairs}\label{section:qsp}
We review the construction of quantum symmetric pairs following \cite{Ko} (Sections 4 and 5) which is closely based on   \cite{L1} but uses right coideal subalgebras instead of left ones.
Afterwards, we analyze elements of these coideal subalgebras  through  biweight space expansions and the introduction of  a special projection map.

\subsection{Definitions}\label{section:Definitions}
Let $\theta$ be a maximally split involution for the semisimple Lie algebra $\mathfrak{g}=\mathfrak{n}^-\oplus \mathfrak{h}\oplus\mathfrak{n}^+$  as defined  in  Section \ref{section:max_split_inv}.
Let $\mathcal{M} $ be the subalgebra of $U_q(\mathfrak{g})$ generated by $E_i,F_i, K_i^{\pm 1}$ for all  $\alpha_i\in \pi_{\theta}$. Note that $\mathcal{M} = U_q(\mathfrak{m})$ where $\mathfrak{m}$ is the Lie subalgebra  of $\mathfrak{g}$ generated by $e_i, f_i, h_i$ for all $\alpha_i\in \pi_{\theta}$. By Section 4 of \cite{Ko}, we can lift $\theta$ to an algebra automorphism $\theta_q $  of $U_q(\mathfrak{g})$  and use this automorphism to define quantum analogs of $U(\mathfrak{g}^{\theta})$.  In particular, the quantum symmetric pair (right) coideal subalgebra $\mathcal{B}_{\theta,\bf{c}, \bf{s}}$ is generated by 
\begin{itemize}
\item The group $\mathcal{T}_{\theta}= \{K_{\beta}|\ \beta\in Q(\pi)^{\theta}\}$
\item $\mathcal{M} = U_q(\mathfrak{m})$
\item the elements $B_i  = F_i + c_i\theta_q(F_iK_i)K_i^{-1} + s_iK_i^{-1}$ for all $i$ with $\alpha_i\notin \pi_{\theta}$
\end{itemize}
where ${\bf c} = (c_1,\dots, c_n)$ and ${\bf s}= (s_1,\dots, s_n)$ are  $n$-tuples of scalars with ${\bf{c}} \in \bf{C}$ and ${\bf{s}}\in \bf{S}$
where $\bf{C}$ and $\bf{S}$ are  defined as  in  \cite{Ko} (5.9) and (5.11) respectively (see also \cite{L1}, Section 7, Variations 1 and 2).  We write $\mathcal{B}_{\theta}$ for $\mathcal{B}_{\theta,\bf{c}, \bf{s}}$ where  $\bf{c}$ and $\bf{s}$ are understood from context. We further assume that each entry in ${\bf s}$ and each entry in ${\bf c}$ is in $\mathbb{R}[q,q^{-1}]_{(q-1)}$.  This guarantees that each $B_i$  is contained in 
$U_{\mathbb{R}(q)}$ and  in $\hat U$.  Thus $\mathcal{B}_{\theta}$ admits a conjugate linear antiautomorphism (see discussion below) and specializes to $U(\mathfrak{g}^{\theta})$ as $q$ goes to $1$ (\cite{L1},  Section 7). 

 Note that there are extra conditions on both $\bf{C}$ and ${\bf S}$.  Only the latter affect the arguments of this paper.  For the entries of ${\bf s}$ it is important to note that  $s_i\neq 0$ implies $\alpha_i$ is an element of a distinguished 
subset $\mathcal{S}$ of $\pi\setminus \pi_{\theta}$ (see \cite{L1} Variation 2 for the definition of $\mathcal{S}$).  
By
 Section 7 of \cite{Let},  we have $\mathcal{S}$ is empty except in the following cases:
\begin{itemize}
\item {\bf Type AIII:} $n=2m+1$, $\pi_{\theta}$ is empty and $\mathcal{S} = \{\alpha_{m+1}\}$
\item {\bf Type CI:} $\mathcal{S} = \{\alpha_n\}$
\item {\bf Type DIII:} Case 1, $\mathcal{S} = \{\alpha_n\}$
\item {\bf Type EVII:} $\mathcal{S}=\{\alpha_7\}$
\end{itemize}
Let $\Gamma_{\theta}$ be the choice of strongly orthogonal $\theta$-system of Theorem \ref{theorem:cases_take2}. By the above description for $\mathcal{S}$ with respect to irreducible symmetric pairs and the proof of Theorem \ref{theorem:cases_take2}, we see that  if $\alpha_i\in \mathcal{S}$ and $\alpha_i=\alpha_{\beta}$ for some $\beta\in \Gamma_{\theta}$ then $\beta=\alpha_{\beta}=\alpha_i$.  In other words, $\alpha_{\beta}\in \mathcal{S}$ implies that $\beta$ is a simple root. 

Write $\mathbb{N}\mathcal{S}$ for the set of linear combinations of elements in $\mathcal{S}$ with nonnegative integer coefficients.
Note  that $\mathcal{S}$ is a subset of $\Delta_{\theta}$ and in particular, $\theta(\alpha)=-\alpha$ for all $\alpha\in \mathcal{S}$.  This follows immediately by inspection of the list above, but was also part of the original definition of $\mathcal{S}$ in \cite{Let}, Section 7.  We will find this property useful in some of the arguments below which include generic elements of $\mathbb{N}\mathcal{S}$.

  By \cite{L1}, there exists a Hopf algebra automorphism $\psi$ of $U_q(\mathfrak{g})$ that restricts to a real Hopf algebra automorphism of $U_{\mathbb{R}(q)}$ and  so that $\psi\kappa\psi^{-1}$ restricts to a conjugate linear antiautomorphism of $\mathcal{B}_{\theta}$.  Set $\kappa_{\theta} = \psi\kappa\psi^{-1}$ and note that $\kappa_{\theta}$ is a conjugate linear antiautomorphism of $U_q(\mathfrak{g})$.  Thus  Theorem \ref{theorem:unitary}, Corollary \ref{cor:unitary1} and Corollary \ref{theorem:unitary2} hold for $\kappa$ replaced by $\kappa_{\theta}$ and apply to $C=\mathcal{B}_{\theta}$.  
Therefore,  $\mathcal{B}_{\theta}$ acts semisimply on finite-dimensional unitary $\mathcal{B}_{\theta}$-modules and the set of finite-dimensional unitary $\mathcal{B}_{\theta}$-modules includes the set of finite-dimensional $U_q(\mathfrak{g})$-modules. Note that $\kappa_{\theta}(X)$ is a nonzero scalar multiple of $\kappa(X)$ for any nonzero weight vector $X$ in $ U^+$, $G^+$, $U^-$ or $G^-$.  Thus, if we show that $Y$ is a nonzero scalar multiple of $\kappa(X)$ for weight vectors $X\in U^+$ (resp. $X\in G^+$) and $Y\in G^-$ (resp. $Y\in U^-$), then the same assertion holds for $\kappa$ replaced by $\kappa_{\theta}$.

By Theorem 4.4 of \cite{Ko}, $\theta_q$ restricts to the identity on $\mathcal{M}$ and $\theta_q(K_{\beta}) = K_{\theta(\beta)}$ for all $\beta\in Q(\pi)$.
Also, in the case where $\theta(\alpha_i)  = -\alpha_{p(i)}$, we get $\theta_q(F_iK_i)= cE_{p(i)}$ for some nonzero scalar $c$.  More generally, by Theorem 4.4 of 
\cite{Ko},  for $\theta(\alpha_i)\neq \alpha_i$ there exists a nonnegative integer $r$ and 
elements 
\begin{align*}
Z_i^- = F_{i_1}\cdots F_{i_r} \quad {\rm and} \quad Z_i^+ = E_{j_1}\cdots E_{j_r}
\end{align*}
with $\alpha_{i_k}$ and $\alpha_{j_k}$ in $\pi_{\theta}$ for $1\leq k\leq r$ and nonzero scalars $u_i$ and $v_i$ such that 
\begin{align*} &\theta_q(E_i) = u_i\sigma((\ad Z^-_{p(i)})(F_{p(i)} K_{p(i)})) 
\end{align*}
and 
\begin{align}\label{theta_qfdefn}
\theta_q(F_iK_i) = v_i(\ad Z^+_{p(i)})(E_{p(i)})
\end{align}
where 
 $\sigma$ is the $\mathbb{C}(q)$ 
algebra antiautomorphism defined in (\ref{sigma}).  Moreover, 
$\theta_q(F_{i}K_i)$ is the highest weight vector of the $({\rm ad}\ \mathcal{M})$-module generated by $E_{p(i)}$ as described in Lemma \ref{lemma:lowest_weight}(ii).  An analogous assertion holds for $\theta_q(E_i)$.

Set $\check{\mathcal{T}}_{\theta} = \{K_{\mu}|\ \mu\in P(\pi)^{\theta}\}$ and note that $\check{\mathcal{T}}_{\theta}$ is a subgroup of $\check{\mathcal{T}}$ which is contained in the simply connected quantized enveloping algebra (see Section \ref{section:basic}).  It is sometimes convenient to extend $\mathcal{B}_{\theta}$ to a ``simply connected version" by including the elements $\check{\mathcal{T}}_{\theta}$.  We denote this extended algebra by $\check{\mathcal{B}}_{\theta}$. By definition, $\check{\mathcal{B}}_{\theta}=\mathcal{B}_{\theta}\check{\mathcal{T}}_{\theta}$.

\subsection{Biweight Space Expansions}\label{section:weight_space}
 Set $\mathcal{M}^+=\mathcal{M}\cap U^+$, the subalgebra of $\mathcal{M}$ generated by $E_i$ for all $\alpha_i\in \pi_{\theta}$. 
Note that we have only defined the generators $B_i$ for $\alpha_i\notin\pi_{\theta}$.   It is useful to set $B_i = F_i$ for  $\alpha_i\in \pi_{\theta}$.  Given an $m$-tuple $J=(j_1,\dots, j_m)$  with entries in $j_k\in \{0,1\}$, we set 
$F_J = F_{j_1}\cdots F_{j_m}$ and $B_J = B_{j_1}\cdots B_{j_m}$.

  Let $\mathcal{J}$ be the subset of tuples with entries in $\{0,1\}$ so that $\{F_J, J\in \mathcal{J}\}$ forms a basis for $U^-$. We have the following direct sum decomposition for $\mathcal{B}_{\theta}$ (\cite{L1}, (7.17)).
\begin{align}\label{Bthetadecomp}
\mathcal{B}_{\theta} = \bigoplus_{J\in \mathcal{J}}B_J\mathcal{M}^+\mathcal{T}_{\theta}.
\end{align}
Given an $m$-tuple $J$, write ${\rm wt}(J)  = \alpha_{j_1}+ \alpha_{j_2} + \cdots + \alpha_{j_m}$.
 The next result provides us with fine information concerning the biweight expansion of elements $B_J$ in $\mathcal{B}_{\theta}$.
\begin{lemma} \label{lemma:decomp}
 Let $J=(j_1,\dots, j_m)$ be an $m$-tuple of nonnegative integers such that  ${\rm wt}(J) =\beta$. Then  $\deg_{\mathcal{F}}(B_J) =m = {\rm ht}(\beta)$ and 
  \begin{align}\label{inclusion}
 B_J\in F_J +
  \sum_{(\lambda,\eta,\gamma,\gamma')\in \mathcal{N}}
G_{-\beta+\lambda+\eta + \gamma}^-U^+_{\theta(-\lambda)-\gamma}K_{-\beta + 2\gamma'}
 \end{align}
where $\mathcal{N}$ is the set of four-tuples  $(\lambda,\eta,\gamma,\gamma')$  in $Q^+(\pi\setminus \pi_{\theta})\times \mathbb{N}\mathcal{S}\times Q^+(\pi)\times Q^+(\pi)$ satisfying
\begin{itemize}
\item  $0<\lambda \leq\beta$
\item  $\eta\leq \lambda$
\item $\gamma\leq \beta-\lambda$ and $ \gamma\leq \theta(-\lambda)-\eta$
\item $ \gamma'\leq \gamma$
\end{itemize}
\end{lemma}
\begin{proof}
 The claim about degree follows immediately from (\ref{inclusion}) and the fact that elements in $G^-$ and $U^+$ have degree $0$. By assumption,  $K_{j_1}^{-1} \cdots K_{j_m}^{-1} = K_{-\beta}$. 
 Recall that for $\alpha_i\notin \pi_{\theta}$, $B_i$ is a linear combination of $F_i$, $\theta_q(F_iK_i)K_i^{-1}$ and $K_i^{-1}$ with the final term only appearing if $\alpha_i\in \mathcal{S}$. 
 Also, by\  the definition of $\theta_q(F_{j_r}K_{j_r})$ (see (\ref{theta_qfdefn}) and related discussion), we have
\begin{align*}
\theta_q(F_{j_r}K_{j_r}) \in [({\rm ad}\ \mathcal{M}^+)E_{p(j_r)}]
\subseteq U^+_{\theta(-\alpha_{j_r})}
\end{align*}
To prove (\ref{inclusion}), 
note that $B_J$ can be written as a linear combination
\begin{align}\label{Beqn}
B_J \in  F_J +\sum_{\mathcal{I}_2\cup \mathcal{I}_3\neq \emptyset} \mathbb{C}(q)A(\mathcal{I}_1, \mathcal{I}_2, \mathcal{I}_3) K_{-\beta}
\end{align} 
where  $\{\mathcal{I}_1,\mathcal{I}_2,\mathcal{I}_3\}$  forms a partition of $\{1, \dots, m\}$ such that $\{\alpha_i|i\in \mathcal{I}_2\}\cap\pi_{\theta}=\emptyset$ and  $\{\alpha_i|i\in\mathcal{I}_3\}\subseteq\mathcal{S}$,
and 
\begin{align*}
A(\mathcal{I}_1, \mathcal{I}_2, \mathcal{I}_3)  = A_{j_1}A_{j_2}\cdots A_{j_m}
\end{align*}
where  
\begin{itemize}
\item $A_{j_i} = F_{j_i}K_{j_i}$ if $i\in \mathcal{I}_1$
\item $A_{j_i} =  \theta_q(F_{j_i}K_{j_i})$
 if $i\in \mathcal{I}_2$
\item  $A_{j_i} = s_{j_i}$ if $i\in \mathcal{I}_3$.
\end{itemize}

Assume that $\mathcal{I}_2$ is not empty.  Set $\gamma_k = \sum_{i\in \mathcal{I}_k}\alpha_{j_i}$ for $k = 1, 2,$ and $3$ and note that $\gamma_1 + \gamma_2 + \gamma_3 = \beta$. The assumption on $\mathcal{I}_2$ further forces $\gamma_2>0$.
Consider first the case where $\mathcal{I}_3$ is empty and so $\gamma_1 = \beta-\gamma_2$.  If we simply  reorder the $A_{j_i}$ so that elements with $j_i\in \mathcal{I}_1$ are on the left and elements with $j_i\in \mathcal{I}_2$ are on the right we get a term in 
\begin{align*}
G^-_{-\beta+\gamma_2}U^+_{\theta(-\gamma_2)}K_{-\beta}.
\end{align*}
In order to understand what other spaces of the form $G^-_{\xi}U^+_{\xi'}K_{-\xi''}$ show up in the biweight expansion of $A(\mathcal{I}_1, \mathcal{I}_2,\mathcal{I}_3)K_{-\beta}$, we need to understand what happens when we commute an element $A_{j_r}$ in $U^+$ (i.e. $j_r\in \mathcal{I}_2$)  pass an element $A_{j_k}$ in $G^-$ (i.e. $j_k\in \mathcal{I}_1$).

The following identity is easily derived from one of the standard defining relations for $U_q(\mathfrak{g})$:
\begin{align}\label{commuting_relation} 
E_iF_jK_j-q^{-2} F_j K_jE_i = -\delta_{ij}(q-q)^{-1}(1-K_i^{2}).
\end{align}
Suppose that  $A_{j_r}= \theta_q(F_{j_i}K_{j_i})\in U^+_{\theta(-\alpha_{j_r})}$ and $A_{j_k}=F_iK_i$.  
Relation (\ref{commuting_relation}) implies that there exists a scalar $c$ such that 
\begin{align*}
A_{j_r}A_{j_k}\in c  A_{j_k}A_{j_r} + U^+_{\theta(-\alpha_{j_r})-\alpha_i}+ U^+_{\theta(-\alpha_{j_r})-\alpha_i}K_i^2.
\end{align*}
Thus,  moving the $A_{j_r}$ terms in $U^+$ to the right of the $A_{j_k}$ terms in $G^-$ yields
\begin{align}\label{Aeqn1}
A(\mathcal{I}_1,\mathcal{I}_2, \mathcal{I}_3)K_{-\beta} &\in G^-_{-\beta+\gamma_2}U^+_{\theta(-\gamma_2)}K_{-\beta}\cr &+ \sum_{\substack{\gamma\leq \beta-\gamma_2\\ \gamma\leq \theta(-\gamma_2)}}\sum_{\gamma'\leq \gamma} 
G^-_{-\beta +\gamma_2 + \gamma}U^+_{\theta(-\gamma_2)+\gamma}K_{-\beta+2\gamma'}
\end{align}
when $\mathcal{I}_2 \neq \emptyset$ and $\mathcal{I}_3 = \emptyset$.  Moreover, since $\mathcal{I}_2 \neq \emptyset$, we have $\gamma_2 >0$.

Now suppose that $\mathcal{I}_3$ is nonempty.  We can write 
\begin{align*}A(\mathcal{I}_1, \mathcal{I}_2, \mathcal{I}_3) = sA_{j'_1}\cdots A_{j'_r}
\end{align*}
where $s = \prod_{j_i\in \mathcal{I}_3}s_{j_i}$ and $j'_1, \dots, j'_r$ is the sequence obtained from $j_1, \dots, j_m$ by removing all $j_k$ with $k\in \mathcal{I}_3$.  In other words, we have consolidated the terms $A_{j_i}$ with $j_i\in \mathcal{I}_3$ into a single scalar $s$, and then renumbered the remaining $A_{j_i}$ accordingly.   
Note that ${\rm wt}(j'_1, \dots, j'_r) = \beta-\gamma_3$.  Arguing as above we get 
\begin{align}\label{Aeqn2}
A(\mathcal{I}_1, \mathcal{I}_2,\mathcal{I}_3)K_{-\beta} &= sA_{j'_1}\cdots A_{j'_r}K_{-\beta}\cr&\in G^-_{-\beta+\gamma_2+\gamma_3}U^+_{\theta(-\gamma_2)}K_{-\beta}\cr &+ \sum_{\substack{\gamma\leq \beta-\gamma_2-\gamma_3\\ \gamma\leq \theta(-\gamma_2)}}\sum_{\gamma'\leq \gamma} 
G^-_{-\beta +\gamma_2 +\gamma_3+ \gamma}U^+_{\theta(-\gamma_2)+\gamma}K_{-\beta+2\gamma'}
\end{align}
when  $\mathcal{I}_3\neq \emptyset$.  Note that in this case, $\gamma_3>0$.

The lemma now follows by expanding out the second term of the right hand side of  (\ref{Beqn}) using (\ref{Aeqn1}), (\ref{Aeqn2}) with $\gamma_2 = \lambda$ and $\gamma_3 = \eta$.
\end{proof}

\begin{remark}  \label{remark:decomp} It is sometimes useful to use  coarser versions of Lemma \ref{lemma:decomp}.  It particular, it follows from (\ref{inclusion}) that 
\begin{align*}
B_J\in F_J + \sum_{0< \lambda\leq\beta}\ \sum_{0\leq \gamma\leq \beta} G^-_{-\beta+\lambda}U^+K_{-\beta+\gamma}
\end{align*}
and 
\begin{align*}
B_J\in F_J+\sum_{0<\lambda\leq\beta}G^-_{-\beta+\lambda}U^0U^+
\end{align*}
for each $m$-tuple $J$ with ${\rm wt}(J) = \beta$.
It follow from the direct sum decomposition (\ref{Bthetadecomp}) and the definition of $l$-weight (Section \ref{section:basic}) that if $g$ is a minimal $l$-weight summand of an element $b\in \mathcal{B}_{\theta}$ of $l$-weight $-\lambda$ then 
$g\in U^-_{-\lambda}\mathcal{M}^+\mathcal{T}_{\theta}$.
\end{remark}

\subsection{Decompositions and Projections}\label{section:decomp_proj}

Set $N_{\theta}^+$ equal to the subalgebra  of $U^+$ generated by $(\ad \mathcal{M}^+)U^+_{\pi\setminus \pi_{\theta}}$. By \cite{L1}, Section 6, $N_{\theta}^+$ is an $({\rm ad}\ \mathcal{M})$-module.  Multiplication induces an isomorphism  (this is just \cite{L1}, (6.2) with $G^-$ replaced by $U^+$)
\begin{align*}
N_{\theta}^+\otimes \mathcal{M}^+\cong U^+.
\end{align*} 
Let $\mathcal{T}'$ be the subgroup of $\mathcal{T}$ generated by  $\{K_i|\ \alpha_i\notin\pi_{\theta} {\rm \ and \ }i\leq p(i)\}$.  Note that 
$\mathcal{T} = 
  \mathcal{T}_{\theta}\times \mathcal{T}' $ and hence the multiplication map defines an isomorphism 
  \begin{align*}
  \mathbb{C}(q)[\mathcal{T}_{\theta}]\otimes \mathbb{C}(q)[\mathcal{T}'] \cong U^0.
 \end{align*}
It follows from the triangular decomposition (\ref{triangle}) for $U_q(\mathfrak{g})$ and the above tensor product decompositions, that multiplication also induces an isomorphism
 \begin{align*}
 U_q(\mathfrak{g}) \cong U^-\otimes \mathbb{C}(q)[\mathcal{T}_{\theta}]\otimes \mathbb{C}(q)[\mathcal{T}' ]\otimes N_{\theta}^+\otimes \mathcal{M}^+.
 \end{align*}
 Using the fact that elements of $\mathcal{M}^+$ are eigenvectors with respect to the adjoint action of $\mathcal{T}$, we may reorder this tensor product decomposition as 
 \begin{align}\label{tensor_prod_decomp}
 U_q(\mathfrak{g}) \cong U^-\otimes\mathcal{M}^+\otimes \mathbb{C}(q)[\mathcal{T}_{\theta}]\otimes \mathbb{C}(q)[\mathcal{T}' ]\otimes N_{\theta}^+.
 \end{align}
 
 Let $(N_{\theta}^+)_+$ denote  the span of the weight vectors in   $N_{\theta}^+$ of positive weight in $Q^+(\pi)$. 
Set 
\begin{align*}(\mathcal{T}')_+=\sum_{\substack{K\in \mathcal{T}'\\ K\neq Id}}\mathbb{C}(q)K
\end{align*}
and let $(\mathcal{T}'N_{\theta}^+)_+$ denote the subspace of $\mathcal{T}'N_{\theta}$ defined by 
\begin{align*}
(\mathcal{T}'N_{\theta}^+)_+ = (\mathcal{T}')_+N_{\theta}^+ + \mathcal{T}'(N_{\theta}^+)_+.
\end{align*}
The tensor product decomposition (\ref{tensor_prod_decomp}) yields a direct sum decomposition
\begin{align}\label{Pdecomp}
U_q(\mathfrak{g}) = \left(U^- \mathcal{M}^+\mathcal{T}_{\theta}\right)\ \oplus \ \left(U ^-\mathcal{M}^+\mathcal{T}_{\theta}(\mathcal{T}'N_{\theta}^+)_+\right).
\end{align}
Let $\mathcal{P}$ denote the projection of $U_q(\mathfrak{g})$ onto  $U^- \mathcal{M}^+\mathcal{T}_{\theta}$ with respect to this decomposition.

\begin{proposition}\label{proposition:BUiso} Given  a positive weight $\beta$ and an element $b\in \mathcal{B}_{\theta}$  such that
\begin{align*}
b = \sum_{{\rm wt}(I) = \beta}B_Ia_I
\end{align*} where each $a_I\in \mathcal{M}^+\mathcal{T}_{\theta}$,
there exists an element $\tilde b \in \mathcal{B}_{\theta}$  such that 
\begin{itemize}
\item[(i)]$\tilde b =   b +
 \sum_{{\rm wt}(J) < \beta}B_Jc_J$ for $c_J \in \mathcal{M}^+\mathcal{T}_{\theta}$.
  \item[(ii)] $\mathcal{P}(\tilde b) = \sum_{ {\rm wt}(I) = \beta}F_Ia_I$.
\end{itemize}
\end{proposition}
\begin{proof}  
 If ${\rm wt}(I) = \alpha_i$ for some $i$,  then  $b=B_ia$ for  some $a\in \mathcal{M}^+\mathcal{T}_{\theta}$ and the lemma follows immediately from the definitions of the generator $B_i$ (both when $\alpha_i\in \pi_{\theta}$ and when $\alpha_i\notin \pi_{\theta}$) and the projection map $\mathcal{P}$.  We prove the lemma using induction on ${\rm ht}(\beta)$.  In particular, we  assume that (i) and (ii) hold for 
all $\beta'$ with $\beta' <\beta$.  

By Remark \ref{remark:decomp} and the definition of $G^-$,  
we can write 
\begin{align*} b= \sum_{{\rm wt}(I) = \beta} F_Ia_I+ Y
\end{align*}
with 
\begin{align*}
Y\in \sum_{0<\lambda\leq\beta} G^-_{-\beta+\lambda} U^+U^0
= \sum_{0<\lambda\leq\beta}U^-_{-\beta+\lambda} U^+U^0.
\end{align*}   
Using the direct sum decompositon (\ref{Pdecomp}), we see that 
\begin{align*}
\mathcal{P}(Y) \in \sum_{0<\lambda\leq\beta}U_{-\beta+\lambda}^-\mathcal{M}^+\mathcal{T}_{\theta}.
\end{align*}
Thus, there exist $c_J\in \mathcal{M}^+\mathcal{T}_{\theta}$ such that 
\begin{align*}
\mathcal{P}(Y)  = \sum_{\beta'<\beta} \sum_{{\rm wt}(J) =\beta'} F_J c_J.
\end{align*}
It follows that 
\begin{align}\label{property}
\mathcal{P}(b) = \sum_{ {\rm wt}(I) = \beta} F_Ia_I+ \mathcal{P}(Y) = \sum_{{\rm wt}(I) = \beta} F_Ia_I+ \sum_{\beta'<\beta}\sum_{{\rm wt}(J) =\beta'} F_J c_J.
\end{align}
Set 
\begin{align}\label{bdefn} b_{\beta'}  = \sum_{\beta'<\beta}\sum_{{\rm wt}(J) =\beta'} B_J c_J
\end{align} and note that $b_{\beta'}\in \mathcal{B}_{\theta}$.   For $\beta'<\beta$, we may apply the inductive hypothesis, yielding  elements $\tilde b_{\beta'} $ satisfying (i) and (ii) with respect to $b_{\beta'}$.  In particular, we have 
\begin{align}\label{finalsum}
\mathcal{P}(\tilde b_{\beta'})= \sum_{{\rm wt}(J) =\beta'}F_Jc_J.
\end{align} Set 
\begin{align*}
\tilde b = b- \sum_{\beta'<\beta} \tilde b_{\beta'}
\end{align*}  and note that $\tilde b\in \mathcal{B}_{\theta}$.  The definition (\ref{bdefn}) of $b_{\beta'}$ and the  fact that each $\tilde b_{\beta'}$ satisfies (i)  with $\beta$ replaced by $\beta'$  guarantees that $\tilde b$ also satisfies (i).  
Condition (ii) for $\tilde b$ follows from (\ref{property}) and (\ref{finalsum}).
\end{proof}

It follows from Proposition \ref{proposition:BUiso} that the map $b\mapsto \mathcal{P}(b)$ defines a vector space isomorphism from $\mathcal{B}_{\theta}$ onto $U^-\mathcal{M}^+\mathcal{T}_{\theta}$. This is the essence of the next result. 

\begin{corollary} \label{corollary:pmapzero}The projection map $\mathcal{P}$ restricts to a vector space isomorphism from $\mathcal{B}_{\theta}$ onto $U^-\mathcal{M}^+\mathcal{T}_{\theta}$.
  Thus, if $b\in \mathcal{B}_{\theta}$ then $\mathcal{P}(b) = 0$ if and only if $b=0$. \end{corollary}
  
  Let $\beta_1$ be a maximal weight such that $c_J\neq 0$ for some $J$ such that  ${\rm wt}(J) = \beta_1$  in the right hand side of (\ref{property}).  The right hand side of  (\ref{property}) can be written as 
\begin{align*}
\sum_{{\rm wt}(I) = \beta} F_Ia_I + \sum_{{\rm wt}(J) = \beta_1}F_Jc_J + \sum_{\substack{\beta'<\beta\\ \beta'\not\geq\beta_1}}\sum_{{\rm wt}(J) = \beta'}F_Jc_J.
\end{align*}
Set $d_1 = b - \sum_{{\rm wt}(J) = \beta_1}B_Jc_J$ and apply induction on the set of $\beta'\not\geq \beta_1$.  The end result is an element $d_2$ in $\mathcal{B}_{\theta}$ with the same image under $\mathcal{P}$ as $\tilde b$. Hence, by Corollary \ref{corollary:pmapzero}, we have   $d_2 = \tilde b$. 

Now assume that each of the $a_I$ terms in the above expression are scalars. We can expand $c_J$ as a sum of weight vectors in $\mathcal{M}^+\mathcal{T}_{\theta}$.  For each of these weight vectors, say $c_J'$, we can find $\lambda_0,  \gamma_0, \gamma'_0, \eta_0$ satisfying the assumptions of Lemma \ref{lemma:decomp} with respect to $\beta$ such that 
\begin{align*}
\beta_1 = \beta-\lambda_0 -\eta_0-\gamma_0
\end{align*}
and $c_J'\in  U^+_{\theta(-\lambda_0)-\gamma_0}K_{-\lambda_0 - \eta_0 -\gamma_0+ 2\gamma_0'}$.
Now suppose that $\lambda_1, \gamma_1, \gamma_1', \eta_1$ also satisfy the assumptions of this lemma with $\beta$ replaced by $\beta_1$.  Set $\lambda_2 = \lambda_0 +\lambda_1, $
$\gamma_2 = \gamma_0 +\gamma_1, \gamma'_2 = \gamma_0' +\gamma'_1$, and $\eta_2 = \eta_0+\eta_1$.  It is straightforward to check that $\lambda_2, \gamma_2, \gamma_2',\eta_2'$ satisfy the conditions of Lemma \ref{lemma:decomp} with respect to $\beta$.  Moreover, 
\begin{align*} (G^-_{-\beta_1+\lambda_1+\eta_1+\gamma_1}U^+_{\theta(-\lambda_1) - \gamma_1}K_{-\beta_1+2\gamma_1'})c_J' \subseteq G^-_{-\beta+\lambda_2+\eta_2+\gamma_2}U^+_{\theta(-\lambda_2)-\gamma_2}K_{-\beta+2\gamma'_2}.
\end{align*}
It follows that, just as  for $b$ of Proposition \ref{proposition:BUiso} , the term $d_1$ is contained in a sum of spaces of the form 
\begin{align*}
G^-_{-\beta+\lambda +\eta+\gamma}U^+_{\theta(-\lambda)-\gamma}K_{-\beta+2\gamma'}
\end{align*}
where $\beta, \lambda, \eta, \gamma, \gamma'$ satisfy the conditions of Lemma \ref{lemma:decomp}.  By induction, the same is true for $d_2=\tilde b$.

\section{Lifting to the Lower Triangular Part}\label{section:lower}
We begin the process of lifting Cartan elements of the form $e_{\beta} + f_{-\beta}$ to the quantum setting.  The main result of this section identifies a lift  with nice properties of $f_{-\beta}$ for nonsimple roots $\beta$ in the special strongly orthogonal $\theta$-systems of Theorem \ref{theorem:cases_take2}.
\subsection{Properties of Special Root Vectors}\label{section:siv}

 Consider a subset $\pi'$ of $\pi$. Given a weight $\lambda\in P^+(\pi)$, we write $\bar L(\lambda)$ for the finite-dimensional simple $U(\mathfrak{g})$-module with highest weight $\lambda$ and let $\bar L_{\pi'}(\lambda)$ denote the finite-dimensional simple $U(\mathfrak{g}_{\pi'}+\mathfrak{h})$-module generated by a vector of highest weight $\lambda$.  Recall that $U(\mathfrak{g}_{\pi'}+\mathfrak{h})$ acts semisimply on $U(\mathfrak{g})$-modules.  For all weights $\gamma$ that are in $P(\pi)$ and restrict to an integral dominant weight in $P^+(\pi')$, we write $[\bar L(\lambda): \bar L_{\pi'}(\gamma)]$ for the multiplicity of $\bar L_{\pi'}(\gamma)$ in $\bar L(\lambda)$ where both are viewed as $U(\mathfrak{g}_{\pi'}+\mathfrak{h})$-modules. Suppose that $\beta\in P^+(\pi)$ and  $(\beta,\alpha) = 0$ for all $\alpha\in \pi'$. It follows that the finite-dimensional simple $U(\mathfrak{g}_{\pi'}+\mathfrak{h})$-module $\bar L_{\pi'}(\beta)$ of highest weight $\beta$ is  a trivial one-dimensional 
$U(\mathfrak{g}_{\pi'}+\mathfrak{h})$-module. 

Set $\mathfrak{n}^-_{\pi'} = \mathfrak{g}_{\pi'} \cap \mathfrak{n}^-$.  Let $M$ be a finite-dimensional $U(\mathfrak{g}_{\pi'})$-module.  We can decompose $M$ into a direct sum of finite-dimensional simple $U(\mathfrak{g}_{\pi'})$-modules, each generated by a highest weight vector. Note that any nonzero highest weight vector with respect to the action of $U(\mathfrak{g}_{\pi'})$ inside $M$ is not an element of 
$\mathfrak{n}^-_{\pi'}M$.  In particular, 
\begin{align}\label{Moplus} M =M_0 \oplus \mathfrak{n}^-_{\pi'}M
\end{align} where $M_0$ is the space spanned by the $U(\mathfrak{g}_{\pi'})$ highest weight  vectors in $M$.

For each simple root $\alpha$, let $\nu_{\alpha}$ denote the corresponding fundamental weight.  

\begin{theorem}\label{theorem:invariant_elements}
Let $\beta\in \Delta^+$ and $\alpha_{\beta}\in{\rm Supp}(\beta)$ such that 
\begin{itemize}
\item[(i)] $\beta= 
\alpha_{\beta} + w_{\beta}\alpha_\beta$ 
\item[(ii)] 
$\pi'\subseteq {\rm StrOrth}(\beta)$.
\end{itemize}
where $\pi' = {\rm Supp}(\beta)\setminus\{\alpha_{\beta}\}$ and $w_{\beta} = w(\pi')_0$.
Then for all $m\geq 1$, $\bar L(m\nu_{\alpha_{\beta}})$  has a unique nonzero vector (up to scalar multiple) of weight $m\nu_{\alpha_{\beta}}-\beta$ that generates a trivial $U(\mathfrak{g}_{\pi'})$-module.

\end{theorem}
\begin{proof} Fix $m\geq 1$. Set $\nu = \nu_{\alpha_{\beta}}$.  Note that $w_{\beta}$ is a product of reflections defined by  roots in $\pi'$. Since $\alpha_{\beta}\notin\pi'$, we must have $w_{\beta}\alpha_{\beta}-\alpha_{\beta}\in Q(\pi')$. 
Thus, the coefficient of $\alpha_{\beta}$ in $\beta= \alpha_{\beta} + w_{\beta}\alpha_\beta$ written as a linear combination of simple roots is $2$.     By assumption (ii) and the fact that $\nu$ is the fundamental weight associated to $\alpha_{\beta}$, we have   $(m\nu-\beta, \alpha) = 0$ for all $\alpha\in \pi'$.  Hence any  weight vector in $\bar L(m\nu)$ of weight $m\nu-\beta$ that is a highest weight vector with respect to the action of $U(\mathfrak{g}_{\pi'})$ generates a trivial $U(\mathfrak{g}_{\pi'})$-module.  Thus we only need to prove
\begin{align}\label{multbar}[\bar L(m\nu):\bar L_{\pi'}(m\nu-\beta)] = 1.
\end{align}

 Let $v$ denote a highest weight generating vector of   $\bar L(m\nu)$. Since $(\nu, \alpha) = 0$ for all $\alpha\in \pi'$,  $U(\mathfrak{g}_{\pi'})$ acts trivially on $v$. 
Moreover, $\alpha_{\beta}\notin \pi'$ ensures that $f_{-\alpha_\beta} v$ is a highest weight vector with respect to the action of 
$U(\mathfrak{g}_{\pi'})$.  It follows that $U(\mathfrak{g}_{\pi'})f_{-\alpha_{\beta}}v$ and $\bar L_{\pi'}(m\nu-\alpha_{\beta})$ are isomorphic as  $U(\mathfrak{g}_{\pi'}+\mathfrak{h})$-modules. Thus $U(\mathfrak{g}_{\pi'})f_{-\alpha_{\beta}}v$ contains a unique nonzero lowest weight vector (up to scalar multiple) of weight $w_{\beta}(m\nu-\alpha_{\beta})=m\nu -w_{\beta}\alpha_{\beta}=m\nu-\beta+\alpha_{\beta}$
with respect to the action of $U(\mathfrak{g}_{\pi'}+\mathfrak{h})$.  
 
 By the 
 previous paragraph,  $yf_{-\alpha_{\beta}}v$ is a (possibly zero) scalar multiple of the lowest weight vector of $U(\mathfrak{g}_{\pi'})f_{-\alpha_{\beta}}v$ for 
any choice of $y\in U(\mathfrak{n}^-_{\pi'})$ of weight $-\beta+2\alpha_{\beta}$.
 Thus,  there is at most one nonzero vector (up to scalar multiple) 
in $\bar L(m\nu)$ of weight $m\nu-\beta$  of the form 
\begin{align*}
f_{-\alpha_{\beta}} y f_{-\alpha_{\beta}}v
\end{align*}
where $y\in U(  \mathfrak{n}^-_{\pi'})$.  
Hence the 
$m\nu-\beta$ weight space of $\bar L(m\nu)$ satisfies
\begin{align}\label{barLinclusion}
\bar L(m\nu)_{m\nu-\beta}
\subseteq \mathbb{C}f_{-\alpha_{\beta}}yf_{-\alpha_{\beta}}v + \mathfrak{n}^-_{\pi'}U(\mathfrak{n}^-)v.
\end{align}
Since  there are no highest weight $U(\mathfrak{g}_{\pi'})$ vectors inside $\mathfrak{n}^-_{\pi'}U(\mathfrak{n}^-)v$, (\ref{barLinclusion}) ensures that $\bar L(m\nu)$ contains at most one highest weight vector of weight $m\nu-\beta$.  In particular, 
\begin{align}\label{barLequality}[\bar L(m\nu):\bar L_{\pi'}(m\nu-\beta)] \leq 1.
\end{align}

   By assumption (ii), $(\alpha,\beta) = 0$ and  $\alpha+\beta$ is not a root for all $\alpha\in \pi'$.  It follows that  $\beta-\alpha$ is also not a root for all $\alpha\in \pi'$.  In particular, we have 
\begin{align*}
[e_{\alpha}, f_{-\beta}] = 0 = [f_{-\alpha}, f_{-\beta}]
\end{align*}
for all $\alpha\in \pi'$.   Thus $e_{\alpha}f_{-\beta}v= [e_{\alpha}, f_{-\beta}]v  =0$ for all $\alpha\in \pi'$.  Similarly, $f_{-\alpha}f_{-\beta}v=0$ for all $\alpha\in \pi'$. Therefore, $\mathbb{C}f_{-\beta}v$ is a trivial $U(\mathfrak{g}_{\pi'}+\mathfrak{h})$-submodule of $\bar L(m\nu)$ of the desired weight.  This shows that the inequality in (\ref{barLequality}) is actually an equality which proves (\ref{multbar}).
\end{proof}

Let $\lambda \in P^+(\pi)$ with $m_i = 2(\lambda, \alpha_i)/(\alpha_i, \alpha_i)$ for each $i$ and let $v$ denote the highest weight generating vector for $\bar L(\lambda)$.  By Section 21.4 of \cite{H}, we have 
\begin{align*}
fv=0 {\rm \ if\ and \ only \ if\ }f\in \sum_iU(\mathfrak{n}^-)f_i^{m_i+1}
\end{align*}
for all $f\in U(\mathfrak{n}^-)$.
Consider $\beta\in Q^+(\pi)$. Suppose that $\pi'$ is  a proper subset of ${\rm Supp}(\beta)$,   $m_i=0$ for all $\alpha_i\in \pi'$ and $m_j\geq {\rm mult}_{\alpha_j}\beta$ for all $\alpha_j\in {\rm Supp}(\beta)\setminus \pi'$.  Then 
\begin{align}\label{fvzero2}fv = 0 {\rm \ if \ and \ only \ if\ }f\in \sum_{\alpha_i\in \pi'}U(\mathfrak{n}^-)f_i
\end{align}
for all elements $f$ of weight $-\beta$ in $U(\mathfrak{n}^-)$.

\begin{corollary}\label{corollary:uniqueness}
Let $\beta\in \Delta^+$ and $\alpha_{\beta}\in {\rm Supp}(\beta)$ such that 
\begin{itemize}
\item[(i)] $\beta = \alpha_{\beta}+w_{\beta}\alpha_{\beta}$
\item[(ii)] $\pi'\subseteq {\rm StrOrth}(\beta)$
\end{itemize}
where $\pi'={\rm Supp}(\beta)\setminus\{\alpha_{\beta}\}$ and $w_{\beta} = w(\pi')_0$.
Then the root vector $f_{-\beta}$ is the unique nonzero element (up to  scalar multiple) of weight $-\beta$ in $U(\mathfrak{n}^-)$ such that $[a,f_{-\beta}]=0$ for all $a\in  \mathfrak{g}_{\pi'}$.
\end{corollary}
\begin{proof}
Suppose that $f$ is a nonzero element of weight $-\beta$ in $U(\mathfrak{n}^-)$ and assume that $[e_i,f]=[f_i,f] = [h_i,f]= 0$ for all $\alpha_i\in \pi'$.  
Choose $m\geq 2$.  By (i) we have  $m\geq {\rm mult}_{\alpha_{\beta}}\beta$.  Set $\nu = \nu_{\alpha_{\beta}}$ and let $v$ be the highest weight generating vector for $\bar L(m\nu)$.  Since $[e_i,f]=0$ for all $\alpha_i\in \pi'$, it follows that $fv$ is a highest weight vector of weight $m\nu-\beta$.  By assumptions on $\beta$, we also have $[e_i,f_{-\beta}]=0$ for all $\alpha_i\in \pi'$.  By Theorem \ref{theorem:invariant_elements} and its proof, any  highest weight vector in $\bar L(m\nu)$ of highest weight $m\nu-\beta$ with respect to the action of $U(\mathfrak{g}_{\pi'})$ is a scalar multiple of  $f_{-\beta}v$. Hence 
\begin{align*}
(f-cf_{-\beta})v = 0
\end{align*}
for some  scalar $c$.  Set $f' = f-cf_{-\beta}$.  By  (\ref{fvzero2}) and related discussion, it follows that 
\begin{align}
\label{fprime}
f'\in \sum_{\alpha_i\in \pi'}U(\mathfrak{n}^-)f_i.
\end{align}
Let $\iota$ be the Lie algebra antiautomorphism of $\mathfrak{g}_{\pi'}$ sending $e_i$ to $e_i$, $f_i$ to $f_i$ and $h_i$ to $-h_i$ for all $i =1, \dots, n$.   Since $[a,f']=0$ for all $a\in \mathfrak{g}_{\pi'}$, the same is true for $\iota(f')$.  By (\ref{fprime}), we have 
\begin{align}\label{fprime2}
\iota(f')\in \sum_{\alpha_i\in \pi'}f_iU(\mathfrak{n}^-).
\end{align}

Choose $\lambda\in P^+(\pi)$ so that $2(\lambda,\alpha_i)/(\alpha_i,\alpha_i)\geq {\rm mult}_{\alpha_i}\beta$ for all $\alpha_i\in \pi'$.  Let $w$ be a nonzero highest weight generating vector for $\bar L(\lambda)$.   Since $[a,\iota(f')]=0$ for all $a\in \mathfrak{g}_{\pi'}$, it follows that $\iota(f')w$ is a $U(\mathfrak{g}_{\pi'})$ highest weight vector with respect to the action of $U(\mathfrak{g}_{\pi'})$.    By (\ref{fprime2}), we see that  $\iota(f')w \in \mathfrak{n}_{\pi'}^-U(\mathfrak{n}^-)w$.   It follows from (\ref{Moplus}) and related discussion that $\iota(f')=0$.   Hence  $f'=0$ and $f$ is a scalar multiple of $f_{-\beta}$.
 \end{proof}

\subsection{Quantum Analogs of Special Root Vectors}\label{section:sie}
In analogy to the classical setting, we write $[L(\lambda):L_{\pi'}(\gamma)]$ for the multiplicity of $L_{\pi'}(\gamma)$ in $L(\lambda)$ where $\lambda\in P^+(\pi)$, $\gamma$ is in $P(\pi)$ and restricts to a dominant integral weight in $P^+(\pi')$, and both modules are viewed as $U_{\pi'}U^0$-modules. 
Note that the character formula for the finite-dimensional simple $U_q(\mathfrak{g})$-module $L(\lambda)$ of highest weight $\lambda$ is the same as its classical counterpart (\cite{Jo} Chapter 4 or \cite{JL}  Section 5.10).   Hence $$[L(\lambda):L_{\pi'}(\gamma)] = [\bar L(\lambda): \bar L_{\pi'}(\gamma)]$$ for all subsets $\pi'$ of $\pi$, all dominant integral weights $\lambda\in P^+(\pi),$
and all weights $\gamma\in P(\pi)$ so that $\tilde\gamma\in P^+(\pi')$.  Thus, we have the following quantum analog of Theorem \ref{theorem:invariant_elements}.

\begin{theorem}\label{theorem:invariant_elements_q}
Let $\beta\in \Delta^+$ and $\alpha_{\beta}\in{\rm Supp}(\beta)$ such that 
\begin{itemize}
\item[(i)] $\beta= 
\alpha_{\beta} + w_{\beta}\alpha_\beta$ 
\item[(ii)] 
$\pi'\subseteq {\rm StrOrth}(\beta)$.
\end{itemize}
where $\pi' = {\rm Supp}(\beta)\setminus\{\alpha_{\beta}\}$ and $w_{\beta} = w(\pi')_0$.
Then for all $m\geq 1$, $ L(m\nu_{\alpha_{\beta}})$  has a unique nonzero vector (up to scalar multiple) of weight $m\nu_{\alpha_{\beta}}-\beta$ that generates a trivial $U_{\pi'}$-module.
\end{theorem}

Consider a subset $\pi'$ of $\pi$ and a  weight $\lambda\in P^+(\pi)$ such that $(\lambda, \alpha_i) = 0$ for all $\alpha_i\in \pi'$. It follows from the description of the (ordinary) adjoint action given in Section \ref{section:basic} that $({\rm ad}\ F_i)K_{-2\lambda} = ({\rm ad}\ E_i) K_{-2\lambda} = 0$ and $({\rm ad}\ K_i) K_{-2\lambda} = K_{-2\lambda}$ for all $\alpha_i\in \pi'$.   Hence, for all $\alpha_i\in \pi'$,  the ordinary adjoint action of $E_i$, namely    $({\rm ad}\ E_i)$, on $({\rm ad}\ U^-)K_{-2\lambda}$ agrees with the action of $({\rm ad}_{\lambda}E_i)$  on  $({\rm ad}_{
\lambda}U^-)1$,  which is the $({\rm ad}\ U_q(\mathfrak{g}))$-submodule of  $G^-(\lambda)$ (from Section \ref{section:dual_vermas}) generated by $1$.  The analogous assertion holds with $E_i$ replaced by $K_i$.  In particular, as $U_{\pi'}$-modules, $({\rm ad}\ U^-)K_{-\lambda}$ is isomorphic to $({\rm ad}_{
\lambda}U^-)1$ where the action on the first module is using the ordinary adjoint action and the action on the latter is using the twisted adjoint action. (Note that this means that the ordinary adjoint action of $U_{\pi'}$ on $G^-K_{-\lambda}$ is the same as the graded action   of $U_{\pi'}$ on $G^-K_{-\lambda}$ as presented in Section \ref{section:dual_vermas}.)

\begin{corollary}\label{theorem:invariant_elements2}
Let $\beta\in \Delta^+$  and $\alpha_{\beta}\in {\rm Supp}(\beta)$ such that 
\begin{itemize}
\item[(i)] $\beta= \alpha_{\beta}+w_{\beta}\alpha_{\beta}$ 
\item[(ii)] $\pi'\subseteq {\rm StrOrth}(\beta)$ 
\end{itemize}
where $\pi' = {\rm Supp}(\beta)\setminus \{\alpha_{\beta}\}$ and $w_{\beta} = w(\pi')_0$.
Then there exists a unique nonzero element (up to  scalar multiple) $Y$ in $(U^-_{-\beta})^{U_{\pi'}}$.
Moreover, a scalar multiple of $Y$ specializes to  $f_{-\beta}$ as $q$ goes to $1$ and $YK_{\beta}K_{-2\nu_{\alpha_{\beta}}} \in ({\rm ad}\ U^-)K_{-2\nu_{\alpha_{\beta}}}$.
\end{corollary}
\begin{proof}
Set $\nu=\nu_{\alpha_{\beta}}$.  Let $G^-(2\nu)$ be the $({\rm ad}_{2\nu}U_q(\mathfrak{g}))$-module defined in Section \ref{section:dual_vermas}. Note that $1$ corresponds to a  highest weight vector of highest possible weight inside $G^-(2\nu)$.  Hence  the socle of $G^-(2\nu)$, which is isomorphic to $L(\nu)$, is equal to $({\rm ad}_{2\nu}U^-)1$.  By the discussion preceding the lemma, $({\rm ad}_{2\nu}U^-)1$ is isomorphic to $({\rm ad}\ U^-)K_{-2\nu}$ as $U_{\pi'}$-modules with respect to the ordinary adjoint action. Thus as $U_{\pi'}$-modules, $({\rm ad}\ U^-)K_{-2\nu}$  and  $L(\nu)$ are isomorphic.  Hence, by Theorem \ref{theorem:invariant_elements_q}, $({\rm ad}\ U^-)K_{-2\nu}$ contains a nonzero element $YK_{\beta-2\nu}$ of weight $-\beta$  that admits a trivial action with respect to $({\rm ad}\ U_{\pi'})$.   By the definition of the adjoint action, we have  $Y\in U^-_{-\beta}$.  Since $({\rm ad}\ U_{\pi'})$ acts trivially on $K_{\beta-2\nu}$, we also have $({\rm ad}\ m)Y=\epsilon(m)Y$ for all $m\in U_{\pi'}$.   Multiplying $Y$ by a (possibly negative) power  of $(q-1)$ if necessary, we may assume that 
$Y\in \hat U$ but $Y\not\in (q-1)\hat U$. Hence $Y$ specializes to a nonzero element of weight $-\beta$ in $U(\mathfrak{n}^-)$ that commutes with all elements of $\mathfrak{g}_{\pi'}$.  By Corollary \ref{corollary:uniqueness}, $Y$ specializes to a nonzero scalar multiple of $f_{-\beta}$. Now suppose that $Y'$ is another element in $U^-_{-\beta}$ satisfying $({\rm ad}\ m)Y'=\epsilon(m)Y'$ for all $m\in U_{\pi'}$.  Assume that $Y'$ is not a scalar multiple of $Y$.   Then some linear combination of $Y$ and $Y'$ must specialize to  an element in $U(\mathfrak{n}^-)$ different from the specialization of $Y$.  This contradicts Corollary \ref{corollary:uniqueness}, thus proving the desired uniqueness assertion.
\end{proof}

\begin{remark} In general, the lifts of root vectors $f_{-\beta}$ given in Corollary \ref{theorem:invariant_elements2} do not agree  with the corresponding weight $-\beta$ element of the quantum PBW basis defined by Lusztig (see \cite{Ja}, Chapter 8).  To get an idea of why this is true, assume that in addition $\beta = w\alpha_{\beta}=\alpha_{\beta}+w_{\beta}\alpha_{\beta}$ where $w = s_{\alpha_{\beta}}w_{{\beta}}$, $w_{{\beta}} = w(\pi')_0$,  and $\pi'={\rm Supp}(\beta)\setminus \{\alpha_{\beta}\}$ as happens for many of the weights appearing in the strongly orthogonal $\theta$-systems of Theorem \ref{theorem:cases_take2}.  By Lemma \ref{lemma:lowest_weight}, $F=T_{w_{\beta}}^{-1}(F_{-\alpha_{\beta}}K_{\alpha_{\beta}})$ is $({\rm ad}\ U^-_{\pi'})$ invariant.  Applying $T_{\alpha_{\beta}}^{-1}$ to $F$ yields an element that is trivial with respect to the action of $({\rm ad}\ F_i)$ for simple roots $\alpha_i\in\pi'$ that also satisfy $(\alpha_i, \alpha_{\beta})=0$.  This is because $T_{\alpha_{\beta}}(F_i) = F_i$ under these conditions. However, if $(\alpha_i, \alpha_{\beta})\neq 0$, then  $T_{\alpha_{\beta}}^{-1}(F_i) \neq  F_i$, and, as a result, we cannot expect $({\rm ad}\ F_i)(T_{\alpha_{\beta}}^{-1}(F))=0$.  Hence $F$ is generally not a trivial element with respect to the action of $({\rm ad}\ U_{\pi'})$.  If instead, we express $\beta = w'\alpha_j$ with $\alpha_j\in \pi'$ and $w'$ some element in the Weyl group $\mathcal{W}$, then we cannot expect $({\rm ad}\ F_j)( (T^{-1}_{w'}(F_jK_j))=0$.   In contrast to the results in this section, we see in the next section that Lusztig's automorphisms play a key role in defining quantum analogs of root vectors for a different family of positive roots. \end{remark}

\subsection{Quantum Analogs of Special Root Vectors, Type A }\label{section:lowtri}

We  establish  a result  similar to  Corollary \ref{theorem:invariant_elements2}  for positive roots in type A.

\begin{theorem}\label{theorem:invariant_elements3} Let $\beta\in \Delta^+$ and $\alpha_{\beta}\in {\rm Supp}(\beta)$ such that 
\begin{itemize}
\item[(i)] $\beta = p(\alpha_{\beta}) + w_{\beta}\alpha_{\beta} = w\alpha_{\beta}$
\item[(ii)] $\pi'\subseteq {\rm StrOrth}(\beta)$
\end{itemize}
where $\pi' = {\rm Supp}(\beta)\setminus\{\alpha_{\beta}, p(\alpha_{\beta})\}$, $w_{\beta} = w(\pi')_0$, and $w = w({\rm Supp}(\beta)\setminus \{\alpha_{\beta}\})_0$.  
Then there exists a unique nonzero element (up to scalar multiple)   $Y\in (U^-_{-\beta})^{U_{\pi'}}$ such that  $YK_{\beta}K_{-2\nu_{\alpha_{\beta}}} \in ({\rm ad}\ U^-)K_{-2\nu_{\alpha_{\beta}}}$. Moreover, a  scalar multiple of   $Y$ specializes to  $f_{-\beta}$ as $q$ goes to $1$. \end{theorem}

\begin{proof}
Set $\nu = \nu_{\alpha_{\beta}}$ and $\pi'' = {\rm Supp}(\beta)\setminus \{\alpha_{\beta}\}$.  Choose $k$ so that $\alpha_{\beta} = \alpha_k$. Set $Y= T^{-1}_{w}(F_{k})$.  By the definition of $T_{w}$ and its inverse (see Section \ref{section:lusztig} and Section 3.4 of \cite{Ko}), we see  that  
$Y$ is a nonzero scalar multiple of $T^{-1}_{w}(F_{k}K_k)K_{-\beta}$.  By Lemma \ref{lemma:lowest_weight}, $YK_{\beta}$ is in $({\rm ad}\ U_{\pi''}^-)(F_{k}K_k)$ and is a lowest weight vector of weight $-\beta$ in the $({\rm ad}\ U_{\pi''})$-module generated by $F_kK_k$.  Since $YK_{\beta}$ has weight $-\beta$ and $(\beta, \alpha_i) = 0$ for all $\alpha_i\in \pi'$, it follows that $YK_{\beta}$ is a trivial $({\rm ad}\ U_{\pi'})$ vector in this module.  Thus $({\rm ad}\ m)YK_{\beta} = ({\rm ad}\ m)Y=\epsilon(m)Y$ for all $m\in U_{\pi'}$.   The facts that  $F_kK_kK_{-2\nu}$ is a scalar multiple of $({\rm ad}\ F_k)K_{-2\nu}$ and $({\rm ad}\ U_{\pi''})$ acts trivially on $K_{-2\nu}$ ensures that  $Y\in [({\rm ad}\ U^-)K_{-2\nu}]K_{2\nu-\beta}$.   Moreover, by the definition of the Lusztig automorphisms, $Y$ specializes to the root vector $f_{-\beta}$ as $q$ goes to $1$. 

Note that $({\rm ad}\ U^-)K_{-2\nu} = ({\rm ad}\ U^-)(F_kK_kK_{-2\nu})+\mathbb{C}(q)K_{-2\nu}$ since $({\rm ad}\ F_i)K_{-2\nu} = 0$ for all $i\neq k$.  Hence, any element in $({\rm ad}\ U^-)K_{-2\nu}$ of weight $-\beta$ is contained in the $({\rm ad}\ U_{\pi''})$-module $({\rm ad}\ U_{\pi''})(F_kK_kK_{-2\nu})$.  As stated in Lemma \ref{lemma:lowest_weight}, this  $({\rm ad} \ U_{\pi''})$-module is isomorphic to the  finite-dimensional $({\rm ad}\ U_{\pi''})$-module with lowest weight $-\beta$.  Hence $({\rm ad}\ U_{\pi''})F_kK_kK_{-2\nu}$ contains a unique nonzero element (up to  scalar multiple) of weight $-\beta$.  This proves the uniqueness assertion. 
\end{proof}

Let $\varphi$ be the ${\mathbb{C}}$ algebra automorphism of $U_q(\mathfrak{g})$ which fixes all $E_i, F_i$, sends $K_i$ to $K_i^{-1}$, and sends $q$ to $q^{-1}$ and let $\varphi'$ be 
the $\mathbb{C}$ algebra automorphism of $U_q(\mathfrak{g})$ which fixes all $K_i^{-1}E_i, F_iK_i$, sends $K_i$ to $K_i^{-1}$, and sends $q$ to $q^{-1}$. Note that $\kappa\varphi = \varphi'\kappa$ where $\kappa$ is the quantum Chevalley antiautomorphism of Section \ref{section:chevalley}.  

The next lemma will be useful in the analysis of quantum root vectors and  Cartan elements associated to a weight $\beta$  satisfying the conditions of Theorem \ref{theorem:invariant_elements3}.  Note that in these cases, $\Delta({\rm Supp}(\beta))$ is a root system of type A. It turns out that this  lemma is also useful in later sections for the  analysis associated to weights $\beta$ that satisfy   condition (4) of Theorem \ref{theorem:cases_take2}. In this latter case, $\Delta({\rm Supp}(\beta))$ is of type B.  

\begin{lemma}\label{lemma:nu}  Let $\alpha_{i_1}, \dots, \alpha_{i_{m+1}}$ be a set of roots with $(\alpha_{i_j}, \alpha_{i_k}) = -1$ if $k = j\pm 1$ and $0$ if $k\neq j\pm 1$ and $i_j\neq i_k$. Then 
\begin{itemize}
\item[(i)]$(({\rm ad}\ E_{i_1}\cdots E_{i_m})E_{i_{m+1}} ) = (-q)^m\varphi(({\rm ad}\ E_{i_{m+1}}\cdots E_{i_2})E_{i_{1}} )$
\item[(ii)] $(({\rm ad}\  F_{i_1}\cdots F_{i_m})F_{i_{m+1}}K_{i_{m+1}} ) =q^{m} \varphi'(({\rm ad}\ F_{i_{m+1}}\cdots F_{i_2})F_{i_{1}}K_{i_{1}} )$
\end{itemize}
\end{lemma}
\begin{proof} 
By the definition of the adjoint action and the assumptions of the lemma, we see that 
\begin{align*}
({\rm ad}\ E_{j})E_{j+1} = E_jE_{j+1}  -q E_{j+1}E_j = -q(E_{j+1}E_j-q^{-1}E_jE_{j+1}) = -q\varphi(({\rm ad}\ E_{j+1})E_{j} ).
\end{align*}
The first assertion follows from this equality and  induction. The second assertion follows by applying the  quantum Chevalley antiautomorphism $\kappa$ to both sides and the similar equality given by (\ref{kappaeq}).
\end{proof}

We are interested in $\beta\in \Delta^+$ that satisfy  (i) and (ii) of Theorem \ref{theorem:invariant_elements3} and appear in one of the  strongly orthogonal $\theta$-systems of Theorem \ref{theorem:cases_take2} and its proof.  In particular, $\beta$ fits into case (3) of Theorem \ref{theorem:cases_take2} and  by Remark   \ref{remark:cases}, ${\rm Supp}(\beta)$ is of type A and $\theta$ restricts to an involution of type AIII/AIV on the Lie algebra $\mathfrak{g}_{{\rm Supp}(\beta)}$.  One checks that this situation only arises when $\mathfrak{g}$ is of type A, D, or E.   Hence  $\Delta$, and thus $\Delta({\rm Supp}(\beta))$ are both simply-laced. Note that the simple root $\alpha_{\beta}$ corresponds to the first simple root of ${\rm Supp}(\beta)$ and $p(\alpha_{\beta})$ corresponds to the last when the simple roots of ${\rm Supp}(\beta)$ are ordered in the standard way corresponding to the Dynkin diagram of type A. Thus, $(\alpha_{\beta}, \beta)\neq 0$ and $(p(\alpha_{\beta}), \beta)\neq 0$.   Moreover, since $\Delta$ is of simply-laced type, we have  
\begin{align}\label{negone}
(\alpha_{\beta}, \beta) = -1 = (p(\alpha_{\beta}), \beta).
\end{align} 

The next corollary establishes connections between a lift of the root vector $e_{\beta}$ to $G^+$ to a lift of the root vector $f_{-\beta}$ to $U^-$.

\begin{corollary}\label{corollary:XtoY}  Let $\beta\in \Delta^+$  and $\alpha_{\beta}\in {\rm Supp}(\beta)$ such that 
\begin{itemize}
\item[(i)] $\beta = p(\alpha_{\beta}) +w_{\beta}\alpha_{\beta} = w\alpha_{\beta}$
\item[(ii)] $\pi'\subseteq {\rm StrOrth}(\beta)$
\end{itemize}
where $\pi' = {\rm Supp}(\beta)\setminus \{\alpha_{\beta}, p(\alpha_{\beta})\}$, $w_{\beta} = w(\pi')_0$, $w =w({\rm Supp}(\beta)\setminus \{\alpha_{\beta}\})_0$. Assume further that  
$\Delta({\rm Supp}(\beta))$ is a root system of type A, the symmetric pair $\mathfrak{g}_{{\rm Supp}(\beta)}, (\mathfrak{g}_{{\rm Supp}(\beta)})^{\theta}$ is of type AIII/AIV, and neither $\alpha_{\beta}$ nor $p(\alpha_{\beta})$ is in ${\rm Orth}(\beta)$. Let $Y$ be the unique  nonzero  element  (up to scalar multiple) in $(U^-_{-\beta})^{U_{\pi'}}$   such that 
 $YK_{\beta}K_{-2\nu}\in ({\rm ad}\ U^-)K_{-2\nu_{\alpha_{\beta}}}$.  Then 
 \begin{itemize}
 \item[(a)] $Y$ is  the unique  nonzero  element  (up to scalar multiple) in $(U^-_{-\beta})^{U_{\pi'}}$   such that 
$F_{-p(\alpha_{\beta})}Y - q^{-1}YF_{-p(\alpha_{\beta})}=0$
\item[(b)]  $Y$ is  the unique   nonzero  element (up to scalar multiple) in $(U^-_{-\beta})^{U_{\pi'}}$  such that   $ F_{-\alpha_{\beta}}Y - qYF_{-\alpha_{\beta}}=0$.
\end{itemize}
Moreover, if $X\in (G^+_{\beta})^{U_{\pi'}}$ and 
\begin{align}\label{Xeqn}
\left(\theta_q(F_{-\alpha_{\beta}}K_{\alpha_{\beta}})K_{\alpha_{\beta}}^{-1}\right)X- qX\left(\theta_q(F_{-\alpha_{\beta}}K_{\alpha_{\beta}})K_{\alpha_{\beta}}^{-1}\right) = 0
\end{align}
 then $\kappa(X)$ is a nonzero scalar multiple of $Y$.
\end{corollary}
\begin{proof}  Set $\nu = \nu_{\alpha_{\beta}}$. By Theorem \ref{theorem:invariant_elements3}, there exists a unique nonzero element (up to  scalar multiple)  $YK_{\beta-2\nu}$ in $ (U^-_{-\beta})^{U_{\pi'}}$ that is also in  $({\rm ad}\ U^-)K_{-2\nu}$.  
Recall (Section \ref{section:dual_vermas}) that $G^-K_{-2\nu}$ has a unique nonzero (up to scalar multiple)  $({\rm ad}\ U_q(\mathfrak{g}))$ lowest weight vector $Y'$ of weight $-\nu+w\nu = -w\alpha_{\beta} = -\beta$ and this lowest weight vector is in $  ({\rm ad}\ U^-)K_{-2\nu}$.  Since $(\beta,\alpha_i) = 0$ for all $\alpha_i\in \pi'$, $Y'\in (U^-_{-\beta})^{U_{\pi'}}$.  The uniqueness property ensures that 
$Y' = YK_{\beta-2\nu}$ and so $YK_{\beta-2\nu}$ is the unique  (up to  scalar  multiple) $({\rm ad}\ U_q(\mathfrak{g}))$ lowest weight vector of $G^-K_{\beta-2\nu}$. 
By (\ref{negone}) and the definition of the adjoint action in Section \ref{section:basic}, we have
\begin{align}\label{Yeqnearly}
(F_{-p(\alpha_{\beta})} Y-q^{-1} YF_{-p(\alpha_{\beta})}) 
=[({\rm ad}\ F_{-p(\alpha_{\beta})})YK_{\beta-2\nu}]K_{-\beta+2\nu} =0.
\end{align}
In particular, $Y$ is the unique nonzero element of $(U^-_{-\beta})^{U_{\pi'}}$ (up to scalar multiple) that satisfies (\ref{Yeqnearly}) which proves (a).

By Lemma \ref{lemma:nu},  $\varphi(Y)$ is the unique nonzero element of $(U^-_{-\beta})^{U_{\pi'}}$ (up to  scalar multiple) that satisfies
\begin{align}\label{Yeqn2early}
(F_{-\alpha_{\beta}} \varphi(Y)-q^{-1} \varphi(Y)F_{-\alpha_{\beta}}) 
=[({\rm ad}\ F_{-\alpha_{\beta}})\varphi(Y)K_{\beta-2\nu}]K_{-\beta+2\nu} =0.
\end{align}
Applying $\varphi$ to each side of (\ref{Yeqn2early}) yields
\begin{align}\label{Yeqn3early}
(F_{-\alpha_{\beta}}Y-q YF_{-\alpha_{\beta}})  =0.
\end{align}
Hence, we can also view $Y$ as the unique nonzero element of $(U^-_{-\beta})^{U_{\pi'}}$ (up to  scalar multiple), that satisfies (\ref{Yeqn3early}). Thus (b) also holds.

Now suppose that $X\in (G^+_{\beta})^{U_{\pi'}}$ and $X$ satisfies (\ref{Xeqn}).  Note that (\ref{Xeqn}) is equivalent to 
\begin{align}\label{Xeqn2}
q^{-1}\left(\theta_q(F_{-\alpha_{\beta}}K_{\alpha_{\beta}})\right)X- qX\theta_q(F_{-\alpha_{\beta}}K_{\alpha_{\beta}})= 0.
\end{align}
Since $\mathfrak{g}_{{\rm Supp}(\beta)}, (\mathfrak{g}_{{\rm Supp}(\beta)})^{\theta}$ is of type AIII/AIV, there are two possibilities for the intersection $\pi_{\theta}$ and ${\rm Supp}(\beta)$.  The first is 
\begin{align*}\pi_{\theta}\cap {\rm Supp}(\beta) \subsetneq {\rm Supp}(\beta)\setminus\{\alpha_{\beta}, p(\alpha_{\beta})\}
\end{align*} and the second is 
\begin{align*}
\pi_{\theta}\cap {\rm Supp}(\beta)= {\rm Supp}(\beta)\setminus\{\alpha_{\beta}, p(\alpha_{\beta})\}.
\end{align*}  In the first case, we have $\theta_q(F_{-\alpha_{\beta}}K_{\alpha_{\beta}}) = E_{p(\alpha_{\beta})}$ and so (\ref{Xeqn2}) becomes 
\begin{align}\label{Xeqn3}
q^{-1}E_{p(\alpha_{\beta})}X- qXE_{p(\alpha_{\beta})}= 0
\end{align} while in the second case, it follows from (\ref{theta_qfdefn}) that
 \begin{align*}
 \theta_q(F_{-\alpha_{\beta}}K_{\alpha_{\beta}}) = c[({\rm ad}\ E_{i_1}E_{i_2}\cdots E_{i_m}) E_{p(\alpha_{\beta})}]
 \end{align*}
 for some nonzero scalar $c$ and choice of $i_1, \dots,i_m$ in $\{1, \dots, n\}$ with $\alpha_{i_j}\in \pi'$ for $j=1, \dots, m$.    Since $\mathfrak{g}_{{\rm Supp}(\beta)}, (\mathfrak{g}_{{\rm Supp}(\beta)})^{\theta}$ is of type AIII/AIV, the sequence of roots 
 \begin{align*}
\alpha_{i_0} = \alpha_{p(\alpha_{\beta})},  \alpha_{i_1},\dots, \alpha_{i_m}, \alpha_{i_{m+1}}=\alpha_{\beta}
 \end{align*} form a set of simple roots for a root system of type A$_{m+2}$.  Note that   $\{\alpha_{i_1}, \dots, \alpha_{i_m}\}=\pi'$.  Moreover, \begin{align*}
 ({\rm ad}\ F_{i_m}\cdots F_{i_1}E_{i_1}E_{i_2}\cdots E_{i_m}) E_{p(\alpha_{\beta})}
 \end{align*} is a nonzero scalar multiple of $E_{p(\alpha_{\beta})}$.  Since $X$ commutes with all elements of $U_{\pi'}$, applying $({\rm ad}\ F_{i_m}\cdots F_{i_1})$ to both sides of (\ref{Xeqn2}) yields (\ref{Xeqn3}) in this case as well. 
 
 Now (\ref{Xeqn3}) is equivalent to 
 \begin{align}\label{Xeqn4}
E_{p(\alpha_{\beta})}K_{p(\alpha_{\beta})}^{-1}X- qXE_{p(\alpha_{\beta})}K_{p(\alpha_{\beta})}^{-1}= 0.
\end{align} Applying the quantum Chevalley antiautomorphism  $\kappa$ (see Section \ref{section:chevalley})  to  $X$ and to both sides of (\ref{Xeqn4}) yields
an element $\kappa(X)\in U^-_{-\beta}$ that commutes with all elements of $U_{\pi'}$ and satisfies
 \begin{align*}
F_{-p(\alpha_{\beta})}\kappa(X)- q^{-1}\kappa(X)F_{-p(\alpha_{\beta})}= 0.
\end{align*}
By (a), $\kappa(X)$ is a (nonzero) scalar multiple of $Y$.  
\end{proof}

\subsection{The Lower Triangular Part }\label{section:lowtri2}

Recall that $\mathcal{T}_{\theta}$ is equal to the group consisting of all elements $K_{\beta}$, $\beta\in Q(\pi)^{\theta}$.
Write $\mathbb{C}(q)[\mathcal{T}_{\theta}]$ for the group algebra generated by $\mathcal{T}_{\theta}$.   Note that   $\mathbb{C}(q)[\mathcal{T}_{\theta}]$ is the Laurent polynomial ring with generators 
\begin{itemize}
\item $K_i, \alpha_i\in \pi_{\theta}$
\item $K_iK_{\theta(\alpha_i)}, \alpha_i\notin \pi_{\theta}$ and $i<p(i)$.
\end{itemize}

\begin{theorem} \label{theorem:lift} Let $\theta$ be a  maximally split involution and let $\Gamma_{\theta}=\{\beta_1,\dots, \beta_m\}$  be  a maximum strongly orthogonal $\theta$-system   as in Theorem \ref{theorem:cases_take2}.  For each $j$, there exists a unique nonzero element $Y_j$  (up to scalar multiple) in $U^-_{-\beta_j}$ such that 
\begin{itemize}
\item [(i)]$Y_j\in [(\ad U^-)K_{-2\nu_{j}}]K_{2\nu_j-\beta_j}$ where $\nu_j$ is the fundamental weight corresponding to  $\alpha_{\beta_j}$.
\item[(ii)] 
For all $\alpha_s\in {\rm StrOrth}(\beta_j)$, $Y_j$ commutes with $E_s,F_s, $ and $K_s^{\pm1}$.
\item[(iii)]  A scalar multiple of $Y_j$ specializes to $f_{-\beta_j}$ as $q$ goes to $1$.
\item[(iv)] If $\beta_j$ satisfies Theorem \ref{theorem:cases_take2} (5) then $Y_j \in [({\rm ad}\ F_{-\alpha'_{\beta}})({\rm ad}\ U^-)K_{-2\nu_{j}}]K_{2\nu_j-\beta_j}$.
\end{itemize}Moreover,  $\mathbb{C}(q)[\mathcal{T}_{\theta}][Y_1,\dots, Y_m]$ is a commutative polynomial ring over $\mathbb{C}(q)[\mathcal{T}_{\theta}]$ in $m$ generators that specializes to 
$U(\mathfrak{h}^{\theta}\oplus \mathfrak{k})$ as $q$ goes to $1$ where $\mathfrak{k}$ is the commutative Lie algebra generated by $f_{-\beta_1}, \dots, f_{-\beta_m}$.
\end{theorem}
\begin{proof}  
We use the notation of Theorem \ref{theorem:cases_take2}. Fix $j$ with $1\leq j\leq m$ and  set $\beta= \beta_j$ and $\nu=\nu_j$. 
Set $\pi' = {\rm Supp}(\beta)\setminus\{\alpha'_{\beta}, \alpha_{\beta}\}$ and $w_{\beta}=w(\pi')_0$. Let $k$ be the integer in $\{1,\dots, n\}$
such that $\alpha_{\beta_j} = \alpha_{k}$.   Note that $({\rm ad}\  E_{k})K_{-2\nu}$ is a scalar multiple of  $E_{k}K_{-2\nu}$ while $({\rm ad}\ E_s)K_{-2\nu} = 0$ for $s \neq k$. 

Suppose that $Y$ is a nonzero element of $U^-_{-\beta}$ satisfying conditions (i).  Since $Y$ has weight $-\beta$, it follows from (i) that 
\begin{align*} Y \in [({\rm ad}\ U^-_{{\rm Supp}(\beta)})K_{-2\nu}]K_{2\nu-\beta}
\end{align*}
and hence $Y$ is an element of $ G^-$ of weight $-\beta$.  It follows that $Y\in G^-_{{\rm Supp}(\beta)}$.  Consider $\alpha_s\in {\rm StrOrth}(\beta)$ and assume that $\alpha_s\notin {\rm Supp}(\beta)$.  Recall that $(\alpha_s, \alpha)\leq 0$ for all $\alpha\in \pi\setminus \{\alpha_s\}$.  Since $\beta$ is a positive root and $(\beta, \alpha_s)=0$, we must have  $(\alpha_s, \alpha) = 0$ for all $\alpha \in {\rm Supp}(\beta)$.  Moreover, since $\alpha_s$ is a simple  root not equal to $\alpha$, we  have that $\alpha_s-\alpha$ is not a root.  This ensures that  $\alpha_s+\alpha$ is also not a root and so $\alpha_s, \alpha$  are strongly orthogonal for all $\alpha\in {\rm Supp}(\beta)$. Using the defining relations of $U_q(\mathfrak{g})$, it is straightforward to see that  $E_s, F_s, K_s^{\pm 1}$ commute with any element in $U_{{\rm Supp}(\beta)}$ and hence commutes with $Y$.  Thus, given $Y\in U^-_{-\beta}$ satisfying (i), we need only show (ii) holds for $\alpha_s\in {\rm StrOrth}(\beta)\cap {\rm Supp}(\beta)$.

 Using Theorem \ref{theorem:cases_take2},  we break the argument up into  five possibilities for $\beta$. 
 Each of these five cases is handled separately.  In the discussion below, we drop the subscript $j$ from $Y_j$ and simply write $Y$.

\medskip
\noindent
{\bf Case 1: } $\beta=\alpha_{\beta}=\alpha'_{\beta}$.  In this case, $Y = F_k$ and it is easy to see that $Y$ is the unique nonzero element (up to  scalar  multiple) of weight $-\beta$ satisfying the three conditions (i), (ii), and (iii).  

\medskip
\noindent
{\bf Case 2:} $\beta  =\alpha_{\beta}+w_{\beta}\alpha_{\beta}$ and $\alpha_{\beta} = \alpha'_{\beta}$. 
By Theorem \ref{theorem:cases_take2} (ii), we have $\pi'\subseteq {\rm StrOrth}(\beta)$.  By Corollary \ref{theorem:invariant_elements2},
there exists a unique  nonzero element (up to  scalar multiple) $Y\in (U^-_{-\beta})^{U_{\pi'}}$.  Moreover $Y\in  [({\rm ad}\ U^-)K_{-2\nu}]K_{2\nu-\beta}$ and multiplying by a scalar if necessary, $Y$ specializes to $f_{-\beta}$ as $q$ goes to $1$. Thus $Y$ satisfies (i), (ii), and (iii).  

\medskip
\noindent
{\bf Case 3: }$\beta = p(\alpha_{\beta} )+ w_{\beta}\alpha_\beta=w\alpha_{\beta}$  where $w = w({\rm Supp}(\beta)\setminus \{\alpha_{\beta}\})_0$.   By Theorem  \ref{theorem:invariant_elements3} and its proof, 
$Y=T_w^{-1}(F_k)$  satisfies the conclusions of the theorem.

\medskip
\noindent
{\bf Case 4:} $\beta = \alpha_{\beta}' + w_{\beta}\alpha_{\beta} = w\alpha'_{\beta}$ where $w = w({\rm Supp}(\beta)\setminus \{\alpha_{\beta}'\})_0$. Furthermore, we may assume that ${\rm Supp}(\beta) = \{\gamma_1, \dots, \gamma_r\}$ generates a root system of type B$_r$ where we are using the standard ordering, $\alpha'_{\beta} = \gamma_r$ the unique short root and $\alpha_{\beta}= \gamma_1$.  
Note that $w=s_1\cdots s_r$ is a reduced expression for $w$ where $s_i$ is the reflection associated to the simple root $\gamma_i$ for each $i$.  

Set $Y' = T_w^{-1}(F_{-\gamma_s})$.  Arguing as in the proof of Theorem \ref{theorem:invariant_elements3}, we see that $Y'$ commutes with $E_s, F_s, K_s^{\pm 1}$ for all $\alpha_s\in {\rm StrOrth}(\beta)$ and $Y'$ specializes to $f_{-\beta}$ as $q$ goes to $1$.  Hence $Y'$ satisfies (ii) and (iii).  Set $\pi''={\rm Supp}(\beta)\setminus \{\alpha_{\beta}'\}$.
By Lemma \ref{lemma:lowest_weight},   $Y'K_{\beta}$ is a lowest weight vector in the $({\rm ad}\ U_{\pi''})$- module generated by 
$F_{-\gamma_r}K_{\gamma_r}.$  Arguing as in the proof of Theorem \ref{theorem:invariant_elements3}, we observe  that  $Y'$ is the unique nonzero element (up to  scalar multiple) element in $[({\rm ad}\ U^-_{\pi''}))F_{-\gamma_r}K_{\gamma_r}]K_{-\beta}$  of weight $-\beta$ satisfying (ii) and (iii).  However, $\gamma_r = \alpha'_{\beta}$ and we want this result to hold for $\gamma_1= \alpha_{\beta}$.  To achieve this, we apply the $\mathbb{C}$ algebra automorphism $\varphi'$ defined in 
the previous section.  Indeed, set $Y = \varphi'(Y'K_{\beta})K_{-\beta}$.   It follows from the definition of  $\varphi'$  that $Y$
also satisfies (ii) and (iii).  

To see that $Y$ satisfies (i), we apply Lemma \ref{lemma:nu}.  In particular, using the reduced expression for $w$ and the fact that $Y'K_{\beta}$ is a lowest weight vector in the $({\rm ad}\ U_{\pi'})$-module 
generated by $F_{-\gamma_r}K_{\gamma_r}$ yields
\begin{align*}
Y' K_{\beta}= c({\rm ad}\ F_{-\gamma_1}F_{-\gamma_2}\cdots F_{-\gamma_{s-1}})(F_{-\gamma_{s}}K_{\gamma_s})
\end{align*} for some  nonzero scalar $c$.   By Lemma \ref{lemma:nu}, we see that 
\begin{align*}
Y K_{\beta}= q^{s-1}c({\rm ad}\ F_{-\gamma_s}F_{-\gamma_{s-1}}\cdots F_{-\gamma_{2}})(F_{-\gamma_1}K_{\gamma_1})
\end{align*}
and so $Y\in [({\rm ad}\ U^-)K_{-2\nu}]K_{2\nu-\beta}$ as desired.

\medskip
\noindent
{\bf Case 5:} $\beta = \alpha_{\beta}' + \alpha_{\beta} +w_{\beta}\alpha_{\beta}$ and $\alpha_{\beta}'\neq \alpha_{\beta}$. Furthermore, $\alpha_{\beta}'\in \pi_{\theta}$, $\alpha_{\beta}'$ is strongly orthogonal to all simple roots in ${\rm Supp}(\beta)\setminus \{\alpha_{\beta}', \alpha_{\beta}\}$, and $\alpha_{\beta}'$ and $\beta$ are orthogonal but not strongly orthogonal to each other.   In particular, $\beta' = \beta-\alpha_{\beta'}$ is a positive root.  Note also that  $\beta' = \alpha_{\beta} + w\alpha_{\beta}$ where $w= w({\rm Supp}(\beta')\setminus \{\alpha_{\beta}\})_0$. Moreover, properties of $\beta$ ensure that $({\rm Supp}(\beta)\setminus \{\alpha_{\beta}\})\subseteq {\rm StrOrth}(\beta)$.  Arguing as in Case 2  yields an element $Y'\in U^-_{-\beta'}$ satisfying (i), (ii), and (iii) with $\beta$ replaced by $\beta'$.
Set $Y = [({\rm ad}\ F_{-\alpha'_{\beta}})Y'K_{\beta'}]K_{-\beta}$ and note that $Y$ satisfies (i) - (iv).  

Now suppose that $Y''$ is another element in $U^-_{-\beta}$ satisfying (i), (ii),  (iii) and (iv).   Since  $Y''\in [({\rm ad}\ U^-)K_{-2\nu}]K_{2\nu-\beta}$.   the multiplicity of $\alpha'_{\beta}$ in $\beta$ is $1$ and $\alpha_{\beta}'$ is strongly orthogonal to all $\alpha_i\in \pi'$, we have 
\begin{align*}[({\rm ad} E_{\alpha'_{\beta}})(Y''K_{\beta-2\nu})]K_{2\nu-\beta}\in (U^-_{-\beta'})^{U_{\pi'}}\cap ({\rm ad}\ U^-)K_{-2\nu} .
\end{align*}
By the uniqueness part of Corollary \ref{theorem:invariant_elements2}, $[({\rm ad} E_{\alpha'_{\beta}})(Y''K_{\beta-2\nu})]K_{2\nu-\beta}$ is a scalar multiple of $Y'$.  On the other hand, if
\begin{align*}
Y'' = [({\rm ad}\ F_{-\alpha_{\beta}'}) fK_{\beta'}]K_{-\beta'}
\end{align*}
for some $f\in U^-_{-\beta'}$, then $[({\rm ad} E_{\alpha'_{\beta}})(Y''K_{\beta-2\nu})]K_{2\nu-\beta}$  is a scalar multiple of $f$ and so $f = Y'$ after suitable scalar adjustment. 
This completes the proof of the uniqueness assertion in this case.

\bigskip
We have shown that for each $\beta_j\in \Gamma_{\theta}$, there exists a unique nonzero element (up to scalar multiple)  $Y_j\in U^-_{-\beta}$ satisfying  (i)-(iv). Recall that each $\beta\in\Gamma_{\theta}$ satisfies $\theta(\beta)=-\beta$ and so $(\alpha,\beta)=0$ for all $\alpha\in \pi_{\theta}$. Hence, the fact that $Y_j$ has weight $-\beta_j$  ensures that $Y_j$ commutes with all elements in $\mathcal{T}_{\theta}$.  Now consider $Y_j$ and $Y_s$ with $s > j$.   
By construction, $Y_s$ is in $U^-_{{\rm Supp}(\beta_s)}$.  By Theorem \ref{theorem:cases_take2}, ${\rm Supp}(\beta_s) \subseteq {\rm StrOrth}(\beta_j)$.  Hence assertion (ii) ensures that each $F_r$ with $\alpha_r\in {\rm Supp}(\beta_s)$  commutes with $Y_j$.  Thus $Y_s$ commutes with $Y_j$ for all $s>k$.  It follows that  $\mathbb{C}(q)[\mathcal{T}_{\theta}][Y_1,\dots, Y_m]$ is a commutative ring.

By the discussion above, each $Y_j$ specializes to the root vector $f_{-\beta_j}$ as $q$ goes to $1$.  Recall that $\mathcal{B}_{\theta}$ specializes to $U(\mathfrak{g}^{\theta})$ and, in particular, $\mathbb{C}(q)[\mathcal{T}_{\theta}]$ specializes to $U(\mathfrak{h}^{\theta})$ as $q$ goes to $1$.  Therefore, $ \mathbb{C}(q)[\mathcal{T}_{\theta}][Y_1,\dots, Y_m]$ specializes to 
$U(\mathfrak{h}^{\theta}\oplus \mathfrak{t})$ as $q$ goes to $1$.  Note that $\mathfrak{h}^{\theta}\oplus \mathfrak{t}$ is a commutative Lie algebra.  Thus, 
$U(\mathfrak{h}^{\theta}\oplus \mathfrak{t})$ is a polynomial ring in $\dim \mathfrak{h}^{\theta} + m$ variables.  It follows that there are no additional relations among the $Y_j$ other than the fact that they commute since such relations would specialize to extra relations for elements of $U(\mathfrak{h}^{\theta}\oplus \mathfrak{t})$.  Hence $\mathbb{C}(q)[\mathcal{T}_{\theta}][Y_1,\dots, Y_m]$ is a polynomial ring in $m$ variables over $\mathbb{C}(q)[\mathcal{T}_{\theta}]$.
\end{proof}

\section{Conditions for Commutativity}\label{section:cond_comm}
  Given a weight $\beta\in Q^+(\pi)$ and a subset $\pi'$ of ${\rm Supp}(\beta)$, we establish conditions for special elements in $U_q(\mathfrak{g})$ to commute with all elements in $U_{\pi'}$.    These results will be used in Section \ref{section:cart_cons} to show that the quantum Cartan subalgebra elements commute with each other. 
 
\subsection{Generalized Normalizers}\label{section:gen}

One of the key ideas in establishing commutativity  is to use the simple fact that 
 if $a\in \mathcal{B}_{\theta}$, then so is the commutator $ba-ab$ for any $b\in \mathcal{B}_{\theta}$. Many of the arguments we use break down $a$ into sums of terms that still retain this commutator property for certain choices of $b$.  In particular, these summands are elements of a generalized normalizer. By generalized normalizer, we mean the following: given three  algebras $C,D,U$  with $C\subseteq D\subseteq U$, the generalized normalizer of $D$ inside $U$ with respect to $C$, denoted $\mathcal{N}_{U}(D:C)$, is defined by
\begin{align*}
\mathcal{N}_{U}(D:C) = \{X\in U|\ bX-Xb\in D{\rm \ for \ all \ }b\in C\}.
\end{align*}
Since we will only be considering $U=U_q(\mathfrak{g})$, we drop the $U$ subscript and just write $\mathcal{N}(D:C)$. 

We use standard commutator notation in the discussion below. In other words, $[b,a] = ba-ab$ for all $a,b\in U_q(\mathfrak{g})$. It follows from the definition of the adjoint action (see Section \ref{section:basic}) that $[F_i, a] = \left(({\rm ad}\ F_i)a\right)K_i^{-1}$.  Hence in many cases it is straightforward to translate between the commutator and the adjoint action.

Recall the notion of $l$-weight defined in Section \ref{section:basic}. 

\begin{lemma}\label{lemma:adF} Let  $\pi'$ be a subset of $\pi$, let $\zeta\in Q^+(\pi)$ and $\xi\in Q(\pi)$ and let $g,u$ be nonzero elements of $G^-_{-\zeta},  U^+$ respectively.  Either there exists $\alpha_i\in \pi'$ such that  
\begin{align}\label{figtu}
({\rm ad}\ F_i)gK_{\xi}u\in  g'K_{\xi}u + G_{-\zeta} U^0 U^+
\end{align}
where $g'$ is a nonzero term in $G^-_{-\zeta-\alpha_i}$ or $gK_{\xi}$ is a nonzero $({\rm ad}\ U_{\pi'})$ lowest weight vector.  Hence if 
\begin{align}\label{Xexp}
X\in gK_{\xi}u+\sum_{\lambda\not\geq\zeta}G^-_{-\lambda}U^0U^+
\end{align} where $g\in G_{-\zeta}$ and $X\in \mathcal{N}(\mathcal{B}_{\theta}: \mathcal{B}_{\theta}\cap U_{\pi'})$ then $gK_{\xi}$ is an $({\rm ad}\ U_{\pi'})$ lowest weight vector or $uK_{\xi+\zeta}\in \mathcal{M}^+\mathcal{T}_{\theta}$.
\end{lemma}

\begin{proof} For each  $\alpha_i\in \pi'$ we have
\begin{align*}({\rm ad}\ F_i) gK_{\xi}u = \left(({\rm ad}\ F_i)gK_{\xi}\right) K_i^{-1}uK_i + gK_{\xi}\left(({\rm ad}\ F_i)u\right).\end{align*}
Note that 
\begin{align*}\left(({\rm ad}\ F_i)gK_{\xi}\right)K_i^{-1}uK_i \in G_{-\zeta-\alpha_i }^-K_{\xi}U^+{\rm\quad  and \quad }gK_{\xi}\left(({\rm ad}\ F_i)u\right) \in G_{-\zeta} U^0 U^+.
\end{align*}
Hence  if $({\rm ad}\ F_i)gK_{\xi}\neq 0$ for some choice of $i$ with $\alpha_i\in \pi'$ then (\ref{figtu}) holds with $g'= \left(({\rm ad}\ F_i)gK_{\xi}\right)K_{-\xi}\neq 0$. On the other hand, if 
$({\rm ad}\ F_i)gK_{\xi}=0$ for all $i$ with $\alpha_i\in \pi'$, then  $gK_{\xi}$ is a nonzero $({\rm ad}\ U_{\pi'})$ lowest weight vector. 

Now assume that $X$  is an element of $\mathcal{N}(\mathcal{B}_{\theta}:\mathcal{B}_{\theta}\cap U_{\pi'})$ that satisfies (\ref{Xexp}). Let $\alpha_i\in \pi'$ and assume that $({\rm ad}\ F_i)g\neq 0$.  It follows that $g'K_{\xi}uK_i^{-1}$  is a minimal $l$-weight summand of $[B_i,gK_{\xi}u] $ where $g'\in G_{-\zeta-\alpha_i}$ and $[F_i,gK_{\xi}u ]K_i= [({\rm ad}\ F_i)gK_{\xi}u]$ as in  (\ref{figtu}).  Note that $g'K_{\xi}uK_i^{-1}\in U^-_{-\zeta-\alpha_i}uK_{\xi+\zeta}$. By Remark \ref{remark:decomp},  we see that $uK_{\xi+\zeta}\in \mathcal{M}^+\mathcal{T}_{\theta}$. \end{proof}

\subsection{Lowest Weight Terms}\label{section:weight_cons}
Recall the notion of height defined in Section \ref{section:basic_classical}.  Here, we use a version of height with respect to a subset $\tau$ of $\pi$.  In particular, given a weight $\beta = \sum_{\alpha_i\in \pi}m_i\alpha_i$, set ${\rm ht}_{\tau}(\beta) = \sum_{\alpha_i\in \tau}m_i$.

Consider a weight $\beta\in Q^+(\pi)$ and let $b=\sum_{{\rm wt}(I)=\beta}B_Ia_I$ be an element of $\mathcal{B}_{\theta}$ where each $a_I\in \mathcal{M}^+\mathcal{T}_{\theta}$.  Set $\tilde b$ equal to the element of $\mathcal{B}_{\theta}$ so that 
$\mathcal{P}(\tilde b) = \sum_{{\rm wt}(I)=\beta}F_Ia_I$ as constructed in Proposition \ref{proposition:BUiso}.  Assume that, in addition, there exists
a subset $\pi'$ of ${\rm Supp}(\beta)$ so that ${\rm ht}_{\tau}(\beta)= 2$ where $\tau = {\rm Supp}(\beta)\setminus \pi'$. 
By the discussion at the end of Section  \ref{section:decomp_proj}, 
$\tilde b$ is contained in a sum of spaces of the form 
\begin{align*}G^-_{-\beta+\lambda+\eta+\gamma}U^+_{\theta(-\lambda)-\gamma} K_{-\beta+2\gamma'}
\end{align*}
where $\beta, \lambda, \eta, \gamma, \gamma'$ satisfy the conditions of Lemma \ref{lemma:decomp}.
We can express $\tilde b$ as a sum of four terms $a_0$, $a_1$, $a_2$, and $a_3$ where 
\begin{itemize}
\item $a_3 =  \sum_{{\rm wt}(I)=\beta}F_Ia_I$
\item $a_2\in \sum_{\delta\in Q^+(\pi')\setminus \{0\}}G^-_{-\beta+\delta}U^0U^+$
\item $a_1 \in\sum_{{\rm ht}_{\tau}(\delta) = 1}G^-_{-\delta}U^0U^+$
\item $a_0 \in \sum_{\delta\in Q^+(\pi')}G^-_{-\delta} U^0U^+$ 
\end{itemize}
We will be using terms of the form $\tilde b$  associated to weights $\beta$ that live in special strongly orthogonal $\theta$-systems in order to construct the quantum Cartan element at $\beta$. A crucial part of the construction relies on
understanding when a summand of minimal $l$-weight of carefully chosen elements in $U_q(\mathfrak{g})$ is an $({\rm ad}\ U_{\pi'})$ lowest weight  vector.  We present three lemmas  that analyze properties of such  $({\rm ad}\ U_{\pi'})$ lowest weight  vectors. These results will eventually be applied to $a_0, a_1,$ and $a_2$ with respect to  particular choices of $\tilde b$.

The first lemma addresses properties of $({\rm ad}\ U_{\pi'})$ lowest weight vectors contained in  $G^-_{-\beta+\delta}U^0U^+$  where $\delta\in Q^+(\pi')$ and  $0\leq \delta<\beta$. Note that  in the context of this lemma, $\delta = \lambda  +\gamma$. Moreover,   the weights $\lambda, \gamma,$ and $\gamma'$ in this lemma satisfy similar conditions to those of Lemma \ref{lemma:decomp}.

\begin{lemma}\label{lemma:weight_cons1}  Let $\beta\in Q^+(\pi)$  and let $\pi'$ be a subset of ${\rm Orth}(\beta)\cap {\rm Supp}(\beta)$.  
Let $\lambda, \gamma,\gamma'$ be three elements in $Q^+(\pi)$ such that  
\begin{itemize}
\item[(i)] $\lambda \leq \beta$
\item[(ii)]  $\gamma\leq \beta-\lambda$, $\gamma \leq \theta(-\lambda)$, and $\lambda +\gamma\in Q^+(\pi')$.
\item[(iii)] $0\leq \gamma'\leq \gamma$. 
\end{itemize}  If $G^-_{-\beta+\lambda+\gamma}K_{-\beta+2\gamma'}$ contains a nonzero $({\rm ad}\ U_{\pi'})$ lowest weight vector, then $\lambda = \gamma=\gamma'=0$ and $G^-_{-\beta+\lambda+\gamma}K_{-\beta+2\gamma'} = U^-_{-\beta}.$
\end{lemma}

\begin{proof} Let $g$ be a nonzero element of $G^-_{-\beta+\lambda+\gamma}$ so that $gK_{-\beta+2\gamma'}$ is an $({\rm ad}\ U_{\pi'})$ lowest weight vector. Note that the assumption $\lambda +\gamma\in Q^+(\pi')$ combined with the fact that both $\lambda$ and $\gamma$ are in $Q^+(\pi)$, forces both $\lambda, \gamma$ to be in $Q^+(\pi')$.  Since $\gamma'\leq \gamma$, we also have $\gamma'\in Q^+(\pi')$.

Set $\xi = (\beta-2\gamma')/2$. Recall that  we can make $G^-$ into a $U_q(\mathfrak{g})$-module, denoted $G^-{(2\xi)}$, using the twisted adjoint action ${\rm ad}_{2\xi}$ (Section \ref{section:dual_vermas}).  Moreover, $G^-{(2\xi)}$ and $G^-K_{-2\xi}$ are isomorphic as $(U^-)$-modules where the action on the latter space is via the ordinary adjoint action.  Hence $g$ is an $({\rm ad}_{2\xi}U_{\pi'})$ lowest weight vector, and, by the discussion in Section \ref{section:dual_vermas}, $g$ generates a finite-dimensional $({\rm ad}_{2\xi}U_{\pi'})$-submodule of $G^-(2\xi)$.  
The fact that this is a twisted action implies that the weight of $g$ viewed as a vector in the $({\rm ad}_{2\xi}U_q(\mathfrak{g}))$-module $G^-{(2\xi)}$ is $$\xi-\beta+\lambda +\gamma = -\beta/2 + \lambda + \gamma - \gamma'.$$    Upon restricting to the action of $U_{\pi'}$, the weight of $g$ is $\lambda + \gamma-\gamma'$ since $\lambda, \gamma, \gamma'\in Q^+(\pi')$ and $\tilde \beta = 0$.   Note that  $\gamma-\gamma'\geq 0$ and $\lambda\geq 0$.  Hence $\lambda +\gamma-\gamma'\geq 0$. On the other hand, $g$ is a lowest weight vector, hence its weight must be nonpositive.  This forces $\lambda + \gamma-\gamma'=0$ and so $\lambda = \gamma-\gamma' = 0$.
Since $\gamma'\leq \gamma\leq \theta(-\lambda)=0$, we also have $\gamma = \gamma'=0$.  The final equality of weight spaces of $U_q(\mathfrak{g})$ follows upon replacing $\lambda,\gamma,$ and $\gamma'$ with $0$.\end{proof}
  
The second lemma focuses on understanding $({\rm ad}\ U_{\pi'})$ lowest weight vectors contained in  $G^-_{-\delta}U^0$ where $\delta \in Q^+(\pi')$. In the notation of this lemma, $\delta$ corresponds to $\beta-\lambda -\gamma$.

\begin{lemma}\label{lemma:weight_cons2} Let $\beta\in Q^+(\pi)$  and let $\pi'$ be a subset of ${\rm Orth}(\beta)\cap {\rm Supp}(\beta)$.
Let $\lambda,\gamma, \gamma'$ be three elements in $Q^+(\pi)$ such that 
\begin{itemize}
\item[(i)] $0< \lambda \leq \beta$
\item[(ii)]  $\gamma\leq \beta-\lambda$, $\gamma \leq \theta(-\lambda)$ and $\beta-\lambda -\gamma\in Q^+(\pi')$
\item[(iii)] $\gamma'\leq \gamma$ and $\gamma'\in Q^+(\pi')$
\end{itemize}  If $G^-_{-\beta+\lambda+\gamma}K_{-\beta+2\gamma'}$ contains a nonzero $({\rm ad}\ U_{\pi'})$ lowest weight vector then $\gamma'=0$, $\beta=\lambda+\gamma,$ and 
 \begin{align*}
 G^-_{-\beta+\lambda+\gamma}K_{-\beta+2\gamma'}= \mathbb{C}(q)K_{-\beta}.
 \end{align*} 
  \end{lemma}
 \begin{proof}
 Let $g$ be a nonzero element of $G^-_{-\beta+\lambda+\gamma}$ so that $gK_{-\beta+2\gamma'}$ is an $({\rm ad}\ U_{\pi'})$ lowest weight vector. Assumption (ii) ensures that $G^-_{-\beta+\lambda+\gamma}\subset G^-_{\pi'}$.  
  
Set $\xi = (\beta-2\gamma')/2$.  It follows from    Section \ref{section:dual_vermas}, that $\tilde\xi\in P^+(\pi')$ and the weight of $g$ is  $-\tilde\xi+w_{\pi'}\tilde\xi = -\xi+w_{\pi'}\xi$ where $w_{\pi'} = w(\pi')_0$.  Hence, the assumptions on  $gK_{-2\xi}$ ensure that  
\begin{align}\label{beta_lambda} \beta-\lambda - \gamma = \xi-w_{\pi'}\xi = {{1}\over{2}}\left((\beta-2\gamma')- w_{\pi'}(\beta-2\gamma')\right). 
\end{align} 
 
Since $\tilde\beta=0$, we must have $\tilde\xi-w_{\pi'}\tilde \xi =  -2 \tilde\gamma'+w_{\pi'}2\tilde\gamma'.$  But $\tilde\xi\in P^+(\pi')$ and $\gamma'\in Q^+(\pi')$ yields a contradiction unless $\xi=\tilde\xi = -2\tilde\gamma'=-2\gamma' = 0$.  Hence $g$ has weight $\tilde\xi-w_{\pi'}\tilde\xi$ which equals zero.   This in turn implies that $\beta-\lambda -\gamma = 0$, and $G^-_{-\beta+\lambda+\gamma}K_{-\beta+2\gamma'} = G^-_{0}K_{-\beta} = \mathbb{C}(q)K_{-\beta}$.  
\end{proof}

The third lemma in the series also analyzes properties of $({\rm ad}\ U_{\pi'})$ lowest weight vectors contained in  $G^-_{-\delta}U^0$ with $\delta\in Q^+(\pi')$ where again $\delta = \beta-\lambda-\gamma$.  However, in contrast to Lemma \ref{lemma:weight_cons1}  and  Lemma \ref{lemma:weight_cons2}, we assume that $\gamma'\notin Q^+(\pi')$.  In addition, we add assumptions to $\beta$ that correspond to some of the properties of roots in the set $\Gamma_{\theta}$ of Theorem \ref{theorem:cases_take2}.

\begin{lemma}\label{lemma:weight_cons3}  Let $\beta\in Q^+(\pi)$ and set $\pi'={\rm Orth}(\beta)\cap {\rm Supp}(\beta)$. Assume that ${\rm ht}_{\pi\setminus \pi'}\beta = 2$ and  $\theta$ restricts to an involution on $\Delta(\pi')$. 
Let $\lambda,\gamma, \gamma'$ be three weights in $Q^+(\pi)$  satisfying
\begin{itemize}
\item[(i)] $0< \lambda \leq \beta$ and $\lambda \in Q^+(\pi\setminus \pi_{\theta})$.
\item[(ii)]  $\gamma\leq \beta-\lambda$, $\gamma \leq \theta(-\lambda)$, and $\beta-\lambda -\gamma\in Q^+(\pi')$.
\item[(iii)] $\gamma'\leq \gamma$ and $\gamma'\notin Q^+(\pi')$
\item[(iv)] $-w_{\pi'}=-w(\pi')_0$ restricts to a permutation on $\pi_{\theta}\cap \pi'$.
\end{itemize}  If $G^-_{-\beta+\lambda+\gamma}K_{-\beta+2\gamma'}$ contains a nonzero $({\rm ad}\ U_{\pi'})$ lowest weight vector then $ \theta(-\lambda)+\gamma-2\gamma'\in Q^+(\pi_{\theta})$, $\lambda +\gamma-2\gamma'\in Q(\pi)^{\theta}$ and 
\begin{align}\label{secondinclusion}G^-_{-\beta+\lambda+\gamma}U^+_{\theta(-\lambda)-\gamma} K_{-\beta+2\gamma'}\subseteq U^-\mathcal{M}^+\mathcal{T}_{\theta}.
\end{align}
\end{lemma}

\begin{proof}  Let $g$ be a nonzero element of $G^-_{-\beta+\lambda+\gamma}$ so that $gK_{-\beta+2\gamma'}$ is an $({\rm ad}\ U_{\pi'})$  lowest weight vector. We start the argument as in the proof of Lemma \ref{lemma:weight_cons2}.  
Using the assumption $\beta-\lambda-\gamma\in Q^+(\pi')$, we see that  $G^-_{-\beta+\lambda+\gamma}\subset G^-_{\pi'}$.   The weight $\xi = (\beta-2\gamma')/2$ satisfies $\tilde\xi\in P^+(\pi')$.  Also,  the weight of $g$ is  $-\tilde\xi+w_{\pi'}\tilde\xi = -\xi+w_{\pi'}\xi$.  In particular, just as in Lemma \ref{lemma:weight_cons2}, the assumptions on  $gK_{-2\xi}$ ensure that (\ref{beta_lambda}) is true.

We are assuming $\gamma'\notin Q^+(\pi')$ (see (iii)).  This combined with the assumptions  that both $\gamma'$ and $\gamma$ are in $Q^+(\pi)$ and $\gamma'\leq \gamma$, ensure that $\gamma$ is also not in $Q^+(\pi')$.  Suppose that $\lambda \in Q^+(\pi')$. The fact that $\lambda \in Q^+(\pi\setminus \pi_{\theta})$ implies that $\theta(-\lambda)\in Q^+(\pi)$.  Since  $\theta$ restricts to an involution on $\Delta(\pi')$, it follows that 
$\theta(-\lambda)$ is  in $Q^+(\pi')$.  By (ii) and (iii), we have $\gamma'\leq \theta(-\lambda)$ which forces $\gamma'\in Q^+(\pi')$, a contradiction of (iii).   Hence $\lambda \notin Q^+(\pi')$.
Therefore  the assumption on the height of $\beta$ combined with $\beta-\lambda-\gamma\in Q^+(\pi')$ (assumption (ii)) yields
\begin{align*}
{\rm ht}_{{\rm Supp}(\beta)\setminus \pi'}(\lambda) = {\rm ht}_{{\rm Supp}(\beta)\setminus \pi'}(\gamma)={\rm ht}_{{\rm Supp}(\beta)\setminus \pi'}(\gamma')=1.
\end{align*}
Thus ${\rm ht}_{{\rm Supp}(\beta)\setminus\pi'}(\beta-2\gamma')=0$.  Since $\gamma'\leq \gamma$ and $\gamma'\leq \beta-\lambda$, it follows that $\beta-2\gamma'\geq 0$ and hence $\beta-2\gamma'\in Q^+(\pi')$. 
Hence, by assumption (iv) and (\ref{beta_lambda}), we have 
\begin{align*}
{\rm ht}_{\pi\setminus \pi_{\theta}}(\beta-\lambda-\gamma) = {\rm ht}_{\pi\setminus \pi_{\theta}}\left((\beta-2\gamma')/2 - w_{\pi'}(\beta-2\gamma')/2\right) = {\rm ht}_{\pi\setminus \pi_{\theta}} (\beta-2\gamma').
\end{align*}
Thus
\begin{align*}
{\rm ht}_{\pi\setminus\pi_{\theta}}(\lambda+\gamma-2\gamma') = 0.
\end{align*}

 It follows from assumptions  (ii)  and (iii) that   $\gamma'\leq \theta(-\lambda)$ and so $2\gamma'\leq \theta(-\lambda) + \gamma$.
Therefore, 
 \begin{align*}
\theta(-\lambda)+\gamma-2\gamma' \geq  0.
\end{align*}
Since $\lambda\in Q^+(\pi\setminus \pi_{\theta})$, we have that $\theta(-\lambda)- \lambda \in  (p(\lambda )- \lambda) +Q^+(\pi_{\theta})$. Recall that $p$ permutes the roots of $\pi\setminus \pi_{\theta}$.    Hence
${\rm ht}_{\pi\setminus \pi_{\theta}}(\theta(-\lambda)- \lambda ) = 0$ and so 
\begin{align*}0 = {\rm ht}_{\pi\setminus\pi_{\theta}} (\lambda+\gamma-2\gamma')  = {\rm ht}_{\pi\setminus \pi_{\theta}}(\theta(-\lambda) +\gamma-2\gamma').
\end{align*}
This equality combined with the positivity of the weight $\theta(-\lambda)+\gamma-2\gamma'$ guarantees that  $\alpha_i\notin {\rm Supp}(\theta(-\lambda)+\gamma-2\gamma')$ for each $\alpha_i\in \pi\setminus \pi_{\theta}$.
Thus 
\begin{align}\label{intheta}\theta(-\lambda)+\gamma-2\gamma'\in Q^+(\pi_{\theta})
\end{align}
and so 
\begin{align*}
\lambda +\gamma-2\gamma'= (\theta(\lambda)+\lambda) + \theta(-\lambda)+\gamma-2\gamma'\in Q(\pi)^{\theta}.
\end{align*}
Therefore 
\begin{align}\label{firstequality}
G^-_{-\beta+\lambda +\gamma}K_{-\beta+2\gamma'} = U^-_{-\beta+\lambda +\gamma}K_{-\lambda-\gamma+2\gamma'} \subseteq U^-_{-\beta+\lambda+\gamma}\mathcal{T}_{\theta}.
\end{align}

Now $\gamma'\leq \gamma$ and $\gamma\leq \theta(-\lambda)$  ensure that 
$0\leq \theta(-\lambda)-\gamma \leq \theta(-\lambda)+\gamma-2\gamma'$.   Thus,   (\ref{intheta}) implies that $\theta(-\lambda)-\gamma$ is also in $Q^+(\pi_{\theta})$.      Hence $U^+_{\theta(-\lambda)-\gamma}\subseteq \mathcal{M}^+$.  This result combined with (\ref{firstequality}) yields (\ref{secondinclusion}).

\end{proof}

\subsection{Highest Weight Terms }\label{section:com_crit}

In this section, we analyze certain elements in subspaces of the form $U^+K_{\zeta}$ that also belong to $  \mathcal{N}(\mathcal{B}_{\theta}:\mathcal{B}_{\theta}\cap U_{\pi'})$.  These results will be applied to the term $a_2\in \sum_{\delta\in Q^+(\pi')}G_{-\delta}^-U^0U^+$ of Section \ref{section:weight_cons} for special choices of $\tilde b$.

\begin{lemma} \label{lemma:beta} Let $\beta\in Q^+(\pi)$ such that $\beta- \alpha_i \notin Q^+(\pi_{\theta})$ and  $\beta+ \alpha_i \notin Q^+(\pi_{\theta})$ for all $\alpha_i\in {\rm Supp}(\beta)$. Set $\pi'={\rm Orth}(\beta)\cap{\rm Supp}(\beta)$ and assume that $\theta$ restricts to an involution on $\Delta(\pi')$.  Let $X\in U^+K_{-\beta}$.
If $X\in \mathcal{N}(\mathcal{B}_{\theta}:\mathcal{B}_{\theta}\cap U_{\pi'})$ then 
$X\in U^{\mathcal{B}_{\theta}\cap U_{\pi'}}$.
\end{lemma}
\begin{proof}
It follows from the defining relations  (in particular, see (\ref{commuting_relation})) of $U_q(\mathfrak{g})$ that 
\begin{align*}
[F_i, X] \in U^+K_{-\beta+\alpha_i}+U^+K_{-\beta-\alpha_i}.
\end{align*}
The assumptions on $\beta\pm \alpha_i$ ensure that  neither $K_{-\beta+\alpha_i}$ nor $K_{-\beta-\alpha_i}$ is in $\mathcal{T}_{\theta}$ for all $\alpha_i\in \pi'$.  It follows that $\mathcal{P}([F_i,X]) = 0$  for all $\alpha_i\in \pi'$. A similar argument shows that $\mathcal{P}([K_i^{-1}, X])=0$.
Note that 
\begin{align*}
[\theta_q(F_iK_i)K_i^{-1}, X] \in U^+K_{-\beta-\alpha_i}.
\end{align*}
for all $\alpha_i\in \pi'\setminus \pi_{\theta}$.
Hence we also have $\mathcal{P}([\theta_q(F_iK_i)K_i^{-1}, X]) = 0$ for these choices of $\alpha_i$.  It follows that $\mathcal{P}([B_i,X])=0$ for all $\alpha_i\in \pi'$.   

If $\alpha_i\in \pi'$ then  $B_i\in U_{\pi'}$  since $\theta$ restricts to an involution on $\Delta(\pi')$.  Hence, by assumption, 
$[B_i,X] \in \mathcal{B}_{\theta}$ for all $i$ such that $\alpha_i\in \pi'\setminus \pi_{\theta}$.  It follows from Corollary  \ref{corollary:pmapzero} that 
\begin{align*}
[B_i,X] = 0 {\rm \ for \ all \ }\alpha_i\in \pi'\setminus \pi_{\theta}.
\end{align*}
Similar arguments shows that $\mathcal{P}([E_i, X]) = \mathcal{P}([F_i,X]) = \mathcal{P}([K_i^{\pm 1}, X]) = 0$ for all $\alpha_i\in \pi_{\theta}\cap \pi'$. Hence $[E_i, X] = [F_i, X]=[K_i^{\pm 1}, X] = 0$ for all $\alpha_i\in \pi_{\theta}\cap \pi'$. It follows that $[b,X] = 0$ for all $b\in \mathcal{B}_{\theta}\cap U_{\pi'}$.
\end{proof}

 Recall that in the fifth case of Theorem \ref{theorem:cases_take2}, $\beta = \alpha_{\beta}'+\alpha_{\beta} + w_{\beta}\alpha_{\beta}$ where $\alpha_{\beta}'\in \pi_{\theta}$, $\alpha_{\beta}'\neq \alpha_{\beta}$, $w_{\beta} = w({\rm Supp}(\beta)\setminus \{\alpha_{\beta}, \alpha_{\beta}'\})_0$ and $\alpha_{\beta}'$ is strongly orthogonal to all simple roots in ${\rm Supp}(\beta)\setminus \{\alpha_{\beta}, \alpha_{\beta}'\}$.  It follows that $\alpha_{\beta}'\notin {\rm Supp}(w_{\beta}\alpha_{\beta})$ and so $\alpha_{\beta}'$ is also not in 
 ${\rm Supp}(\beta-\alpha_{\beta}')$.  Set $\beta' = \beta -\alpha_{\beta}'$. Since $\theta(\beta) = -\beta$ and $\theta(\alpha_{\beta}') = \alpha_{\beta}'$, we have $\theta(\beta-\alpha_{\beta}') = -\beta-\alpha_{\beta}'$ and so $\theta(\beta') = -\beta'-2\alpha_{\beta}'$.  Thus $\beta'$ satisfies the following conditions:
\begin{itemize}
\item $\theta(\beta') \in -\beta'-Q^+(\pi_{\theta})$
\item ${\rm Supp}(\theta(-\beta')-\beta')\cap {\rm Supp}(\beta') = \emptyset$.
\end{itemize}
Note that if $\theta(\beta) = -\beta$ then $\beta$ also satisfies the above two properties (with $\beta'$ replaced by $\beta$).  We use these two properties as assumptions on $\beta$ in the next lemma so that we can eventually apply it to both roots $\beta$ satisfying $\theta(\beta) = -\beta$ and to roots of the form $\beta'$ derived from Case 5 of Theorem \ref{theorem:cases_take2}.

\begin{lemma}\label{lemma:beta2} Let $\beta\in Q^+(\pi)$ such that $\theta(-\beta) - \beta \in Q^+(\pi_{\theta})$ and ${\rm Supp}(\theta(-\beta)-\beta)\cap {\rm Supp}(\beta) = \emptyset$.  Set $\pi'={\rm Orth}(\beta)\cap {\rm Supp}(\beta)$. Assume that 
\begin{itemize}
\item $\pi' \subseteq {\rm Orth}(\theta(-\beta))$
\item $\theta$ restricts to an involution on   $\Delta(\pi')$
\item ${\rm ht}_{{\rm Supp}(\beta)\setminus \pi'}(\beta) = 2$
\item $(\pi_{\theta}\cup \mathcal{S})\cap {\rm Supp}(\beta) \subseteq \pi'$. 
\end{itemize}
Let 
\begin{align*}
X\in   \sum_{\eta\in  Q^+(\pi')}\sum_{0\leq\gamma\leq \beta}U^+_{\theta(-\beta+\gamma)-\gamma-\eta}K_{-\beta}.
\end{align*}
If $X\in \mathcal{N}(\mathcal{B}_{\theta}:\mathcal{B}_{\theta}\cap U_{\pi'})$
then $X\in U^+_{\theta(-\beta)}K_{-\beta}+\mathbb{C}(q)K_{-\beta}$  and  $X\in (U_q(\mathfrak{g}))^{ U_{\pi'}}$.
\end{lemma}
\begin{proof}
Assume $X\in \mathcal{N}(\mathcal{B}_{\theta}:\mathcal{B}_{\theta}\cap U_{\pi'})$.  Since ${\rm ht}_{{\rm Supp}(\beta)\setminus \pi'}\beta=2$ and $\pi_{\theta}\cap {\rm Supp}(\beta)\subseteq \pi'$, it follows that $\beta\pm \alpha_i\notin \pi_{\theta}$ for all $\alpha_i\in \pi$. Hence by 
Lemma \ref{lemma:beta}, $X\in (U_q(\mathfrak{g}))^{ \mathcal{B}_{\theta}\cap U_{\pi'}}$.

Let $\gamma\in Q^+(\pi)$ and $\eta\in Q^+(\pi')$ such that  in its expression as a sum of weight vectors, $X$ has a nonzero contribution from $U^+_{\theta(-\beta+\gamma)-\gamma-\eta}K_{-\beta}$. Note that this ensures that 
\begin{align*}
\theta(-\beta+\gamma)-\gamma-\eta\in Q^+(\pi).
\end{align*}
Write 
$\gamma = \gamma_1'+\gamma_2'$ with $\gamma_2'\in Q^+(\pi')$ and $\gamma_1'\in Q^+({\rm Supp}(\beta)\setminus \pi')$. Since $\theta$ restricts to an involution on $\Delta(\pi')$, we have $\theta(\gamma_2') \in Q(\pi')$.  The fact that $\pi_{\theta}\cap {\rm Supp}(\beta) \subseteq \pi'$ forces $\theta(-\gamma) \in p(\gamma_1') + Q^+(\pi_{\theta})\subseteq p(\gamma_1') + Q^+(\pi')$.  Hence 
\begin{align*}
{\rm ht}_{{\rm Supp}(\beta)\setminus\pi'}(\gamma) = {\rm ht}_{{\rm Supp}(\beta)\setminus\pi'}(\theta(-\gamma)).
\end{align*}  

By the hypothesis of the lemma, $\mu = \theta(-\beta) - \beta$ satisfies ${\rm Supp}(\mu)\cap {\rm Supp}(\beta) = \emptyset$.  The fact that $\theta$ restricts to an automorphism of  $\Delta(\pi')$ combined with the fact that $\pi_{\theta}\cap {\rm Supp}(\beta) \subset \pi'$ guarantees that 
\begin{align*}{\rm ht}_{{\rm Supp}(\beta)\setminus\pi'}(\beta) = 
{\rm ht}_{{\rm Supp}(\beta)\setminus\pi'}(\theta(-\beta)).
\end{align*}  Thus by our assumptions on the height of $\beta$ and the results of the previous paragraph, we oberve that 
\begin{align*}
 {\rm ht}_{{\rm Supp}(\beta)\setminus\pi'}(\gamma)\in \{0,1\}
\end{align*} 
and so 
\begin{align*}
 {\rm ht}_{{\rm Supp}(\beta)\setminus\pi'}(\theta(-\gamma)+\gamma)\in \{0,2\}.
\end{align*}  It follows that either $\theta(-\beta + \gamma)-\gamma -\eta\in Q^+(\pi')$ or $\theta(-\gamma)+\gamma +\eta\in Q^+(\pi')$.  
  Thus 
\begin{align*}
X\in \sum_{\zeta\in Q^+(\pi'), \zeta<\beta} U^+_{\beta+\mu-\zeta}K_{-\beta} + \sum_{\zeta\in Q^+(\pi'), \zeta<\beta}U^+_{\zeta}K_{-\beta}.
\end{align*}
Note that $[b, X]=0$ for $b\in \mathcal{M}_{\pi'}$ ensures that each weight  summand of $X$ also commutes with all $b\in \mathcal{M}_{\pi'}$. Thus each weight summand of $X$ generates a trivial $({\rm ad}\ \mathcal{M}_{\pi'})$-module.

Recall that $\mathcal{B}_{\theta}$ is invariant under the action of $\kappa_{\theta}$ (see Section \ref{section:qsp}), which is  conjugate to the Chevalley antiautomorphism $\kappa$, defined in Section \ref{section:chevalley}.  Applying $\kappa_{\theta}$ to $X$ yields an element $Y\in (U_q(\mathfrak{g}))^{ \mathcal{B}_{\theta}\cap U_{\pi'}}$ such that 
\begin{align*}
Y\in \sum_{\zeta\in Q^+(\pi'), \zeta<\beta} G^-_{-\beta-\mu + \zeta}K_{-\beta} + \sum_{\zeta\in Q^+(\pi'), \zeta<\beta}G^-_{-\zeta}K_{-\beta}.
\end{align*}
Since $X\in \mathcal{N}(\mathcal{B}_{\theta}:\mathcal{B}_{\theta}\cap U_{\pi'})$ so is $Y$ and hence   $[B_i, Y] = 0$ for all $\alpha_i\in \pi'$.   By the previous paragraph, we also have $Y\in (U_q(\mathfrak{g}))^{\mathcal{M}_{\pi'}}$.

Suppose that $Y_{-\beta-\mu+\zeta}K_{-\beta}\in G^-_{-\beta-\mu+\zeta}K_{-\beta}$ is a nonzero $({\rm ad}\ U_{\pi'})$ lowest weight vector  for some $\zeta\in Q^+(\pi')$ with $\zeta<\beta$. Since $(\mu, \alpha) = 0$ for all $\alpha\in \pi'$, it follows that $Y_{-\beta-\mu+\zeta}K_{-\beta-\mu}$ is also a nonzero $({\rm ad}\ U_{\pi'})$ lowest weight vector.  Hence, we may apply Lemma \ref{lemma:weight_cons1} with $\theta(-\beta) = \beta+\mu$ playing the role of $\beta$, $\zeta$ playing the role of $\lambda$ and $\gamma = \gamma' = 0$.   This forces $\zeta = 0$ and so $Y_{-\beta-\mu+\zeta}K_{-\beta} \in G^-_{-\beta-\mu}K_{-\beta}$.  Now suppose that $Y_{-\zeta}K_{-\beta}\in G^-_{-\zeta}K_{-\beta}$ is a nonzero $({\rm ad}\ U_{\pi'})$ lowest weight vector for some 
$\zeta\in Q^+(\pi')$ with $\zeta<\beta$. Applying Lemma \ref{lemma:weight_cons2} with $\gamma=\gamma'=0$ and $\lambda = \beta -\zeta$ yields $\zeta = 0$ and so $Y_{-\zeta}K_{-\beta} \in \mathbb{C}(q)K_{-\beta}$.

Now choose $\delta$ maximal so that 
\begin{align}\label{Ydeltaexp}
Y \in  Y_{-\delta}K_{-\beta}  + \sum_{\xi \not\geq \delta} G^-_{-\xi}K_{-\beta}.
\end{align}
 By Lemma \ref{lemma:adF}, $Y_{-\delta}K_{-\beta}$ is an $({\rm ad}\ U_{\pi'})$ lowest weight vector.  Hence by the previous paragraph, either 
$\delta = \beta+\mu$ or $\delta =0$.  If $\delta = 0$,  then $Y\in \mathbb{C}(q)K_{-\beta}$ and thus so is $X$ and the lemma follows.  Thus, we assume that $\delta = \beta+ \mu = \theta(-\beta)$.

Since $\delta = \beta+\mu$, (\ref{Ydeltaexp}) can be rewritten as 
\begin{align*}
Y \in  Y_{-\beta-\mu}K_{-\beta} + Y_{-\delta'}K_{-\beta}  + \sum_{\xi \not\geq \delta'} G^-_{-\xi}K_{-\beta}
\end{align*} where $Y_{-\delta} = Y_{-\beta-\mu}\in G^-_{-\beta-\mu}$ is an $({\rm ad}\ U_{\pi'})$ lowest weight vector and $Y_{-\delta'}\in G^-_{-\delta'}$  with $\delta'\not\geq\beta+\mu$. By Section \ref{section:dual_vermas},  $Y_{-\beta-\mu}K_{-\beta}$ generates a finite-dimensional $({\rm ad}_{\beta}U_{\pi'})$-module. 
Since $\beta+\mu$ restricts to $0$ with respect to ${\pi'}$,  this module is trivial  and so 
\begin{align}\label{adEexp} ({\rm ad}_{\beta}E_i)(Y_{-\beta-\mu}K_{-\beta}) = 0 {\rm \  for\ each\ } \alpha_i\in \pi'. 
\end{align} We also have $({\rm ad}_{\beta} E_i) K_{-\beta} = ({\rm ad}\ E_i)K_{-\beta} = 0$ and, thus
\begin{align*}
[E_iK_i^{-1},Y_{-\beta-\mu}K_{-\beta} ]=(({\rm ad}\ E_i) Y_{-\beta-\mu}K_{-\beta})K_i^{-1}
\end{align*}
for all $\alpha_i\in \pi'$.  

  It follows from (\ref{adEexp}), Section \ref{section:dual_vermas} (in particular, the inequality (\ref{lower_degree})  and related discussion) and the definition of the adjoint action in Section \ref{section:basic} that 
\begin{align*}
\deg_{\mathcal{F}}\left([E_i,Y_{-\beta-\mu}K_{-\beta} ]\right)=\deg_{\mathcal{F}}\left((({\rm ad}\ E_i) Y_{-\beta-\mu}K_{-\beta})\right)< \deg_{\mathcal{F}}K_{-\beta}
\end{align*}
for all $\alpha_i\in \pi'\setminus\pi_{\theta}$.
 Since  $\theta$ restricts to an involution on $\Delta(\pi')$, it follows that 
$\theta_q(F_iK_i) \in ({\rm ad}\ \mathcal{M}_{\pi'}^+)E_{p(i)}
$ for all $\alpha_i\in \pi'\setminus \pi_{\theta}$.  Recall that $Y_{-\beta-\mu}K_{-\beta}$ is $({\rm ad}\ \mathcal{M}_{\pi'})$ invariant.  Hence, we also have  \begin{align*}
\deg_{\mathcal{F}}[\theta_q(F_iK_i), Y_{-\beta-\mu}K_{-\beta}]<\deg_{\mathcal{F}}K_{-\beta}
\end{align*}
for all $\alpha_i\in \pi'\setminus \pi_{\theta}$.
On the other hand, 
$[F_iK_i, G^-K_{-\beta})]\subseteq G^-K_{-\beta}$ and so 
\begin{align*}
 \deg_{\mathcal{F}}[F_iK_i, G^-K_{-\beta})] = \deg_{\mathcal{F}}K_{-\beta}.
 \end{align*}
 Hence, if $[F_i, Y_{-\delta'}K_{-\beta}]\neq 0$ then 
 \begin{align*}
 [B_i, Y] = [F_i, Y_{-\delta'}K_{-\beta}] + &{\rm \ terms \ of \ lower \ degree \ with \  respect \ to \  }\mathcal{F}.
\end{align*}  
But $[B_i, Y] =0$, and so, we must have $[F_i, 
Y_{-\delta'}K_{-\beta}]=0$ for all $\alpha_i\in \pi'\setminus \pi_{\theta}$.

 Recall that each weight summand of $Y$ generates a trivial $({\rm ad}\ \mathcal{M}_{\pi'})$-module.  This fact combined with the above paragraph yields   $[F_i, 
 Y_{-\delta'}K_{-\beta}]=0$ for all $\alpha_i\in \pi'$.  Thus, $Y_{-\delta'}K_{-\beta}$ must  be an $({\rm ad} \ U_{\pi'})$ lowest weight vector. Since $\delta'<\beta+\mu$, it follows from the arguments above that  $\delta' = 0$.
Therefore \begin{align*}
Y \in Y_{-\beta-\mu} K_{-\beta} + \mathbb{C}(q)K_{-\beta}.
\end{align*}
Applying $\kappa_{\theta}$ to every term in the above expression yields the desired analogous assertion for $X$.  

Note that $X\in (U_q(\mathfrak{g}))^{U_{\pi'}}$ if and only if $Y\in (U_q(\mathfrak{g}))^{U_{\pi'}}$.  By assumptions on $\pi'$ and $\beta$, it follows that $K_{-\beta}\in (U_q(\mathfrak{g}))^{U_{\pi'}}$ and $Y_{-\beta-\mu}K_{-\beta}$ commutes with all elements of $\mathcal{T}\cap U_{\pi'}$.  Hence it is sufficient to show that $Y_{-\beta-\mu}K_{-\beta}\in  (U_q(\mathfrak{g}))^{U_{\pi'}}$.   We have already shown that $[a, Y_{-\beta-\mu}K_{-\beta}] = 0$ for all $a\in \mathcal{M}_{\pi'}$ and all $a=F_i$ with $\alpha_i\in \pi'$. Consider $\alpha_i\in \pi'\setminus \pi_{\theta}$.  Now $[B_i, Y_{-\beta-\mu}K_{-\beta}]$ can be written as a sum of terms of weights \begin{align*}-\beta-\mu-\alpha_i,\  -\beta-\mu, {\rm \ and \ } -\beta-\mu-\theta(\alpha_i).
\end{align*} 
But  $[B_i,Y_{-\beta-\mu}] = [B_i, Y]=0$ and so each of these weight terms must be zero.  In particular, the term of weight $-\beta-\mu-\theta(\alpha_i)$, which is $[\theta_q(F_iK_i)K_i^{-1}, Y_{-\beta-\mu}K_{-\beta}]$, equals zero for all $\alpha_i\in \pi'\setminus \pi_{\theta}$.  The lemma now follows from the fact that 
 $U_{\pi'}$ is generated by $\mathcal{M}_{\pi'}$,  $F_i,\alpha_i\in \pi'$, $U^0\cap U_{\pi'}$, and $\theta_q(F_iK_i)K_i^{-1}, \alpha_i\in \pi'\setminus \pi_{\theta}$.\end{proof}

\subsection{Determining Commutativity}\label{section:det_comm}
Let $\tau$ be a subset of $\pi$. For each $m\geq 0$, set
\begin{align*}
G^-_{(\tau,m)} = 
\sum_{{\rm ht}_{\tau}(\delta) = m}
G^-_{-\delta}
{\rm \quad and \quad }U^+_{(\tau,m)} = 
\sum_{ {\rm ht}_{\tau}(\delta) = m}
U^+_{\delta}.
\end{align*}
 Note that 
 \begin{align*}
 G^-_{(\tau,m)}\ \cap\  G^-_{(\tau, j)} = 0
 \end{align*}
 whenever $j\neq m$.  We also have \begin{align}\label{commFandE}[K_i^{\pm 1}, G^-_{(\tau,j)}]\subseteq G^-_{(\tau,j)},\  [F_i, G^-_{(\tau,j)}]\subseteq G^-_{(\tau,j)}{\rm \ and \ }[E_i, G^-_{(\tau,j)}]\subseteq G^-_{(\tau,j)}U^0U^+
\end{align} for all  $j\geq 0$ and  $\alpha_i\in \pi\setminus \tau$.   It follows that  for each $j$, the space $G^-_{(\tau,j)}U^0U^+$ is preserved by the action of $({\rm ad}\ U_{\pi\setminus \tau})$.   Thus, if $\theta$ restricts to an involution on $\Delta(\pi')$ where $\pi'$ is a subset   of $\pi\setminus \tau$ then 
\begin{align*}
[B_i, G^-_{(\tau,j)}]\subseteq G^-_{(\tau,j)}U^0U^+
\end{align*}
for all $j\geq 0$ and  $\alpha_i\in \pi'$.  Analogous statements hold for $G^-_{(\tau,j)}$ replaced by $U^+_{(\tau,j)}$.  Moreover, we have 
\begin{align*}
[B_i, G^-_{(\tau,j)}U^0U^+_{(\tau,m)}]\subseteq G^-_{(\tau,j)}U^0U^+_{(\tau,m)}
\end{align*}
for all nonnegative integers $j,m$ and all $\alpha_i\in \pi'$.

Now suppose that $\beta$ is a positive weight, $\pi'$ a subset of ${\rm Supp}(\beta)$,  and $\tau = {\rm Supp}(\beta)\setminus \pi'$.  Assume further that 
${\rm ht}_{\tau}(\beta) = 2$. Note that 
\begin{align*}
\sum_{\delta\in Q(\pi')} G^-_{-\beta+\delta} = G^-_{(\tau,2)} {\rm \ and \ }\sum_{\delta\in Q(\pi')} G^-_{-\delta} = G^-_{(\tau,0)}.
\end{align*}

\begin{proposition}\label{proposition:liftbeta} Let $\beta\in Q^+(\pi)$ such that $\theta(-\beta)-\beta\in Q^+(\pi_{\theta})$ and 
${\rm Supp}(\theta(-\beta)-\beta)\cap {\rm Supp}(\beta) = \emptyset$. Set $\pi'={\rm Orth}(\beta)\cap {\rm Supp}(\beta)$ and assume that $\pi'\subseteq {\rm Orth}(\theta(-\beta))$. Assume that  \begin{itemize}
\item  $\theta$ restricts to an involution on   $\Delta(\pi')$
\item  ${\rm ht}_{{\rm Supp}(\beta)\setminus \pi'}(\beta) = 2$
\item  $(\pi_{\theta}\cup \mathcal{S})\cap {\rm Supp}(\beta) \subseteq \pi'$
 \item   $-w_{\pi'}=-w(\pi')_0$ restricts to a permutation on $\pi_{\theta}\cap \pi'$.
 \end{itemize}
Let $a = \sum_{I,{\rm wt}(I)\leq \beta}B_Ia_I$ be an element of $\mathcal{B}_{\theta}$ where each $a_I\in \mathcal{M}^+_{\theta}$ and $a_I$ is a scalar whenever ${\rm wt}(I) = \beta$, such that the following two conditions hold:
\begin{itemize}
\item[(i)]   $\mathcal{P}(a) = \sum_{{\rm wt}(I)= \beta}F_Ia_I$.
\item[(ii)] $\mathcal{P}(a) \in (U_q(\mathfrak{g}))^{U_{\pi'}}$.
\end{itemize}
Then  \begin{align*}a\in \mathcal{P}(a) + G^-_{(\tau.1)}U^0U^+_{(\tau,1)} + \mathbb{C}(q)K_{-\beta}+ X_{\theta(-\beta)} K_{\theta(-\beta)-\beta}
\end{align*} where $\tau= {\rm Supp}(\beta) \setminus \pi'$ and  $X_{\theta(-\beta)} \in G^+_{\theta(-\beta)}$.  Moreover, 
$$a \in (U_q(\mathfrak{g}))^{\mathcal{B}_{\theta}\cap U_{\pi'}}, \quad 
X_{\theta(-\beta)} K_{\theta(-\beta)-\beta}\in (U_q(\mathfrak{g}))^{U_{\pi'}},$$
the lowest weight summand  of $a$  is $\mathcal{P}(a)$ (which is also the lowest $l$-weight summand of $a$) and the highest weight summand of $a$ is the 
element $X_{\theta(-\beta)} K_{\theta(-\beta)-\beta}$ in $G^+_{\theta(-\beta)} K_{\theta(-\beta)-\beta}
$.   
\end{proposition}

\begin{proof} 
The fact that $\theta$ restricts to an involution on $\Delta(\pi')$ ensures that $B_i\in U_{\pi'}$ for all $i$ with $\alpha_i\in \pi'$.  Hence $[B_i, \mathcal{P}(a)] = 0$ for all $\alpha_i\in \pi'$.

We may assume that $a=\tilde b$  where $b = \sum_{I, {\rm wt}(I) =\beta}B_Ia_I$ as in Proposition \ref{proposition:BUiso}.  By the discussion at the end of Section  \ref{section:decomp_proj}, 
we may write 
\begin{align}\label{cplusd1}
a = \sum_{ {\rm wt}(I)=\beta}F_Ia_I + \sum_{\lambda', \lambda, \gamma, \gamma'}b_{-\beta+\lambda' }u_{\lambda, \gamma}K_{-\beta+2\gamma'}
\end{align}
where each  index of four weights   $\lambda, \eta, \gamma, \gamma'$ with $\eta = \lambda'-\lambda -\gamma$ satisfy the conditions of Lemma \ref{lemma:decomp} with respect to $\beta$ and 
$u_{\lambda, \gamma}\in U^+_{\theta(-\lambda)-\gamma}$.  Set $\tau= {\rm Supp}(\beta)\setminus \pi'$. As in Section \ref{section:weight_cons}, $a$ can be expressed as  a sum $a= a_0 + a_1 + a_2 + a_3$ where 
\begin{align*}
a_3 &= \sum_{{\rm wt}(I) = \beta}F_Ia_I\cr\cr
a_2 &= \sum_{\lambda' \in Q^+(\pi')\setminus\{0\}}\sum_{\lambda,\gamma,\gamma'}b_{-\beta+\lambda'}u_{\lambda, \gamma}K_{-\beta+2\gamma'}\ \in\ G^-_{(\tau, 2)}U^0U^+\cr\cr
a_1 &=\sum_{{\rm ht}_{\tau}(\beta-\lambda') = 1}\sum_{\lambda,\gamma,\gamma'}b_{-\beta+\lambda'}u_{\lambda, \gamma}K_{-\beta+2\gamma'}\ \in\ G^-_{(\tau,1)}U^0U^+\cr\cr
a_0 &= \sum_{\beta-\lambda' \in Q^+(\pi')}\sum_{\lambda,\gamma,\gamma'}b_{-\beta+\lambda'}u_{\lambda,\gamma}K_{-\beta+2\gamma'}\ \in \ G^-_{(\tau,0)}U^0U^+.
\end{align*}
Given $\alpha_j\in \pi'$, 
we have $[B_j, a] = [B_j, a_2 + a_1 + a_0]\in \mathcal{B}_{\theta}$ and so, by (\ref{commFandE}), $[B_j, a]=[B_j, a_2 + a_1 + a_0]$ is an element of $G^-_{(\tau,2)} U^0U^+ + G^-_{(\tau, 1)} U^0U^+ + G^-_{(\tau, 0)}U^0U^+$.  Since $a\in \mathcal{B}_{\theta}$, we also have $a_2+a_1+a_0\in \mathcal{N}(\mathcal{B}_{\theta}:\mathcal{B}_{\theta}\cap U_{\pi'})$. 

The fact that $\mathcal{P}(a) = \mathcal{P}(a_3) = a_3$ (i.e. assumption (i)) ensures that  $u_{\lambda, \gamma}K_{\lambda'+2\gamma'}\notin \mathcal{M}^+\mathcal{T}_{\theta}$ for any choice of $(\lambda', \lambda, \gamma, \gamma')$ with $b_{-\beta+\lambda'}u_{\lambda, \gamma}K_{-\beta+2\gamma'}\neq 0$.  Thus 
\begin{align*}
\mathcal{P}(a_2) = \mathcal{P}(a_1) = \mathcal{P}(a_0) = 0.
\end{align*} 
Let   $(\lambda'_0, \lambda_0, \gamma_0, \gamma'_0)$ be chosen so that  $b_{-\beta+\lambda'_0}u_{\lambda_0, \gamma_0}K_{-\beta+2\gamma_0'}\neq 0$ and $\beta-\lambda'_0$ is maximal. It follows that 
\begin{align}\label{loweraexp}
a_2+a_1+a_0 \in b_{-\beta+\lambda'_0}u_{\lambda_0, \gamma_0}K_{-\beta+2\gamma_0'}+\sum_{\delta\not\geq\beta-\lambda'_0}G^-_{-\delta}U^0U^+.
\end{align}
By Lemma \ref{lemma:adF}, $b_{-\beta+\lambda'_0}K_{-\beta+2\gamma'_0}$ is an $({\rm ad}\ U_{\pi'})$ lowest weight vector.

Suppose that $\lambda'$ and $\bar\lambda'$ are two weights such that 
${\rm ht}_{\tau}(\beta-\lambda') > {\rm ht}_{\tau}(\beta-\bar\lambda')$.  It follows that $\beta-\bar\lambda'\not\geq\beta-\lambda'$.  Hence if $a_2\neq 0$, 
we can always choose $(\lambda'_0, \lambda_0, \gamma_0, \gamma'_0)$ as in (\ref{loweraexp}) so that $b_{-\beta+\lambda'_0}u_{\lambda_0, \gamma_0}K_{-\beta+2\gamma'_0}$ is a nonzero biweight summand of $a_2$. 
In other words, $\lambda'_0>0$ and 
$\lambda'_0\in Q^+(\pi')
\setminus\{0\}$. Recall that the four-tuple $(\lambda_0, \eta_0, \gamma_0, \gamma_0')$ satisfies the conditions of Lemma \ref{lemma:decomp} where $\eta_0 = \lambda_0'-\lambda_0-\gamma_0$.    Applying Lemma \ref{lemma:weight_cons1} to $\lambda_0 +\eta_0$ (instead of $\lambda_0$), $\gamma_0$, $\gamma'_0$, and $\beta$, we see that $\lambda_0 + \eta_0= \gamma_0=\gamma'_0=0$.  But then $\lambda'_0=0$, a contradiction.  Therefore $a_2 = 0$.   It follows that 
\begin{align*}
a= a_3 + a_1 + a_0\in \mathcal{P}(a) + G^-_{(\tau,1)}U^0U^+ + G^-_{(\tau,0)} U^0U^+.
\end{align*}

We next argue that $a_1 \in G^-_{(\tau, 1)}U^0U^+_{(\tau,1)}$.
Consider a  biweight summand $d$ of $a$ where $d=b_{-\beta+\lambda' }u_{\lambda, \gamma}K_{-\beta+2\gamma'}$.  By the definition of $a_1$, the assumption on the height of $\beta$ and the fact that $\eta\in \mathcal{S}\cap {\rm Supp}(\beta)\subseteq \pi'$, we see that 
\begin{align}\label{lgexp}
{\rm ht}_{{\rm Supp}(\beta)\setminus \pi'}(\lambda +\gamma)=1.
\end{align}
We can write $\lambda = \lambda_1 + \lambda_2$ where $\lambda_1\in Q^+({\rm Supp}(\beta)\setminus \pi')$ and $\lambda_2\in Q^+(\pi')$. Since $\theta(-\beta)-\beta \in Q^+(\pi_{\theta})$, ${\rm Supp}(\theta(-\beta)-\beta)\cap {\rm Supp}(\beta)= \emptyset$ and $\theta$ restricts to an involution on $\Delta(\pi')$, it follows that 
\begin{align*}\theta(-\lambda) \in p(\lambda_1) + Q^+(\pi') + Q^+({\rm Supp}(\theta(-\beta)-\beta).
\end{align*}
Hence \begin{align*}{\rm ht}_{{\rm Supp}(\beta)\setminus \pi'}(\lambda ) = {\rm ht}_{{\rm Supp}(\beta)\setminus \pi'}(\theta(-\lambda)).
\end{align*}
Now $\gamma\leq \theta(-\lambda)$.  Since $\gamma, \lambda,$ and $\theta(-\lambda)$ are all positive weights, we see from (\ref{lgexp}) that
\begin{align*}{\rm ht}_{{\rm Supp}(\beta)\setminus \pi'}(\lambda ) = {\rm ht}_{{\rm Supp}(\beta)\setminus \pi'}(\theta(-\lambda)) = {\rm ht}_{{\rm Supp}(\beta)\setminus \pi'}(\theta(-\lambda)-\gamma)=1.
\end{align*}
Thus $U^+_{\theta(-\lambda)-\gamma}\subseteq U^+_{(\tau, 1)}$. Hence    $u_{\lambda, \gamma}\in U^+_{(\tau,1)}$ and so $a_1\in G^-_{(\tau,1)}U^0U^+_{(\tau,1)}$
as claimed.  

Since $\pi_{\theta}\cap {\rm Supp}(\beta) \subseteq \pi'$, we have 
$U^+_{(\tau, 1)}\cap \mathcal{M}^+=0$. 
Consider $b\in \mathcal{B}_{\theta}\cap U_{\pi'}$ and recall that $[b,a_3]=0$.  It follows that $[b, a_1+a_0]\in \mathcal{B}_{\theta}$. By the discussion preceding the proposition, $[b, a_1]$ is also in $ G^-_{(\tau, 1)}U^0U^+_{(\tau,1)}$. On the other hand, 
 $[b, a_0]\in G^-_{\pi'}U^0U^+$. Hence, if $[b,a_1]\neq 0$, then there exists $\zeta$ and $\xi$ with ${\rm ht}_{\tau}(\zeta) = {\rm ht}_{\tau}(\xi)=1$ such that 
\begin{align*}[b,a_1+a_0]=  b' +\sum_{\delta\not\geq\zeta}G^-_{-\delta}U^0U^+
\end{align*}
where $b'$ is a nonzero element of $G^-_{-\zeta}U^0U^+_{\xi} $.
This contradicts Remark \ref{remark:decomp} since $[b, a_1 + a_0] \in \mathcal{B}_{\theta}$. Hence $[b, a_1] = 0$.

We have shown that $a_2 = 0$ and $a_1\in (U_q(\mathfrak{g}))^{\mathcal{B}_{\theta}\cap U_{\pi'}}$.  Hence $a_0\in \mathcal{N}(\mathcal{B}_{\theta}:\mathcal{B}_{\theta}\cap U_{\pi'})$. 
In analogy to (\ref{loweraexp}), $a_0$ satisfies 
\begin{align*}
a_0 \in b_{-\beta+\lambda'_1}u_{\lambda_1, \gamma_1}K_{-\beta+2\gamma_1'}+\sum_{\delta\not\geq\beta-\lambda'_1}G^-_{-\delta}U^0U^+
\end{align*}
for an appropriate choice of $\lambda'_1, \lambda_1, \gamma_1, \gamma_1'$.  Using Lemma \ref{lemma:decomp}, we obtain $b_{-\beta+\lambda'_1}K_{-\beta+2\gamma_1'}$ is an $({\rm ad}\ U_{\pi'})$ lowest weight vector.  By definition of $a_0$, $\beta-\lambda_1'\in Q^+(\pi')$.    If $\gamma'\notin Q^+(\pi')$, then by Lemma \ref{lemma:weight_cons3}, 
$b_{-\beta+\lambda'_1}u_{\lambda_1, \gamma_1}K_{-\beta+2\gamma_1'}\in U^-\mathcal{M}^+\mathcal{T}_{\theta}$ contradicting the fact that $\mathcal{P}(a_0) = 0$.  Therefore 
$\gamma'\in Q^+(\pi')$.
Applying Lemma \ref{lemma:weight_cons2} (with $\lambda_1 + \eta_1$ in the role of $\lambda$ of the lemma) yields 
\begin{align*}
b_{-\beta+\lambda'_1} K_{-\beta+2\gamma'_1} \in \mathbb{C}(q)K_{-\beta}.
\end{align*}
 Hence by the maximality of $\beta-\lambda_1'$,  we see that 
$a_0\in U^+K_{-\beta}.$
By Lemma \ref{lemma:beta2}, 
\begin{align*}a_0\in U^+_{\theta(-\beta)}K_{-\beta} + \mathbb{C}(q)K_{-\beta}
{\rm \ \ and \ \  }a_0\in (U_q(\mathfrak{g}))^{U_{\pi'}}.
\end{align*}
Note that $K_{-\beta}$ and $K_{\theta(-\beta)}$ are both elements of   $(U_q(\mathfrak{g}))^{U_{\pi'}}.$ Thus there exists an element $X_{\theta(-\beta)}$ in $ G^+_{\theta(-\beta)}\cap (U_q(\mathfrak{g}))^{U_{\pi'}}$ such that 
\begin{align*}
a_0 \in X_{\theta(-\beta)}K_{\theta(-\beta)-\beta} +  \mathbb{C}(q)K_{-\beta}.
\end{align*}
Since $a_3$ and $a_0$ are in $(U_q(\mathfrak{g}))^{U_{\pi'}}$ and $a_1 \in (U_q(\mathfrak{g}))^{\mathcal{B}_{\theta}\cap U_{\pi'}}$, it follows that 
\begin{align*}a=a_3+a_1+a_0\in (U_q(\mathfrak{g}))^{\mathcal{B}_{\theta}\cap U_{\pi'}}.
\end{align*}

From (\ref{cplusd1}) we see that $\mathcal{P}(a)$ is  both the lowest weight term  in the expansion of $a$ as a sum of weight vectors and the lowest $l$-weight summand of $a$.  Applying $\kappa_{\theta}$, the version of the Chevalley antiautomorphism that preserves $\mathcal{B}_{\theta}$, to $a$ yields a nonzero element element $\kappa_{\theta}(a)$ in $\mathcal{B}_{\theta}$ with 
\begin{align*}
\mathcal{P}(\kappa_{\theta}(a)) =  \mathcal{P}(\kappa_{\theta}(X_{\theta(-\beta)}))
\end{align*}
where $X_{\theta(-\beta)}K_{\theta(-\beta)-\beta}$ is the contribution to $a$ from $U^+_{\theta(-\beta)}K_{-\beta}$.     By Corollary \ref{corollary:pmapzero}, $\mathcal{P}(\kappa_{\theta}(a))$ is nonzero.   Hence $X_{\theta(-\beta)}K_{\theta(-\beta)-\beta}\neq 0$ and is the term of highest weight in the expansion of $a$ as a sum of weight vectors.  
 \end{proof}

\section{Constructing Cartan Subalgebras}\label{section:cart_cons}

We present the main result of the paper here, namely, the construction of quantum analogs of Cartan subalgebras (or, more precisely, their enveloping algebras) of $\mathfrak{g}^{\theta}$ where $\mathfrak{g}, \mathfrak{g}^{\theta}$ is a maximally split symmetric pair. 

\subsection{Specialization}\label{section:spec}
 Recall that $\mathcal{B}_{\theta}$ specializes to $U(\mathfrak{g}^{\theta})$ as $q$ goes to $1$ (Section \ref{section:Definitions}).
The next lemma sets up the basic tools needed to verify quantum Cartan subalgebras  specialize to their classical counterparts as $q$ goes to $1$.
 
 \begin{lemma}\label{lemma:specialization}
 Let $b=b_1+C$ be an element of $\mathcal{B}_{\theta}$ such that
 \begin{itemize}
 \item  $b_1\in \mathcal{B}_{\theta}\cap \hat U$ and $b_1$ specializes to  a nonzero element of $U(\mathfrak{g}^{\theta})$ as $q$ goes to $1$ 
 \item  $C\in \sum_{\zeta,\xi}G^-_{-\zeta}U^0U^+_{\xi}$ where each $\xi\notin Q^+(\pi_{\theta})$ and $\zeta\notin Q^+(\pi_{\theta})$. 
 \end{itemize} Then $C\in (q-1)\hat U$ and $b$ specializes to $\bar b_1$ as $q$ goes to $1$.  Moreover, if $\kappa_{\theta}(b_1) = ab_1 + b_2$ where $b_2 \in \sum_{\xi\notin Q^+(\pi_{\theta})}G^-U^0U^+_{\xi}$ and $a$ is a scalar, then $a$ evaluates to a nonzero scalar in $\mathbb{C}$  at $q=1$ and $\kappa_{\theta}(b) = a b$.
 \end{lemma}
 \begin{proof}   The fact that   $\xi\notin Q^+(\pi_{\theta})$ ensures that  $U^+_{\xi}\cap \mathcal{M}^+ = 0$.  Repeated applications of Lemma \ref{lemma:decomp} shows that
 $C \in \sum_{\xi\notin Q^+(\pi_{\theta})}\mathcal{B}_{\theta}U^0U^+_{\xi}$.   Thus, the specialization of  $(q-1)^sC$ cannot be a nonzero element of $U(\mathfrak{g}^{\theta})$ for any $s$. 

Choose $s$ minimal  such that  $(q-1)^s b \in \mathcal{B}_{\theta}\cap \hat U$. The assumptions on $b_1$ force $s\geq 0$. Moreover, the minimality of $s$ ensures that $(q-1)^sb$ specializes to a nonzero element of $U(\mathfrak{g})$.  If $s>0$, then $(q-1)^sb_1\in (q-1)\hat U$ and so the specialization of $(q-1)^sb$ as $q$ goes to $1$ is the same as the specialization of $(q-1)^sC$.  This contradicts the fact that $b\in \mathcal{B}_{\theta}$ and $\mathcal{B}_{\theta}$ specializes to $U(\mathfrak{g}^{\theta})$.  Hence $s= 0$,  $b\in \hat U$, $b$ specializes to $b_1$  and so $C \in (q-1)\hat U$.       

Note that  $\kappa_{\theta}(b_1)$ specializes to the image of $\bar b_1$ under the classical Chevalley antiautomorphism of $\mathfrak{g}$. Hence $ab_1 + b_2\in \hat U$, forces $a$ to evaluate to a nonzero element of $\mathbb{C}$ at $q=1$.   Since $\kappa_{\theta}$ preserves $\mathcal{B}_{\theta}$, $\kappa_{\theta}(b) - ab  = b_2 -a\kappa_{\theta}(C)$ is an element of $\mathcal{B}_{\theta}$.  Now 
$\kappa_{\theta}(C) \in \sum_{\zeta\notin Q^+(\pi_{\theta})}G^-U^0U^+_{\zeta}$, hence by the same arguments as above, $(q-1)^s\kappa_{\theta}(C)$ cannot specialize to a nonzero element of $\mathcal{B}_{\theta}$ for any $s$.  The assumptions of the lemma ensure the  same is true for $b_2$.   Hence $b_2 - a\kappa_{\theta}(C) = 0$ and $\kappa_{\theta}(b) = ab$.  
\end{proof}

\subsection{Quantum Symmetric Pair Cartan Subalgebras}\label{section:qcs}

In the next result, we put together the upper and lower triangular parts associated to positive roots $\beta$ using Section \ref{section:det_comm} in order to identify a quantum Cartan Subalgebra of $\mathcal{B}_{\theta}$.  We also show that this Cartan subalgebra satisfies a uniqueness property with respect to the action of the quantum Chevalley antiautomorphism.

Note that in general, we cannot expect to lift a term of the form $e_{\beta} + f_{-\beta}$ with $\theta(\beta)=-\beta$ to an element in $U^+_{\beta} + G_{-\beta}^-$ that is also in $\mathcal{B}_{\theta}$.  We see in the next theorem that for most cases, we get an intermediate term (referred to as $C_j$ or just $C$) that specializes to $0$ as $q$ goes to $1$. 
 
\begin{theorem}\label{theorem:main}Let $\mathfrak{g}, \mathfrak{g}^{\theta}$ be a  maximally split pair. Let $\Gamma_{\theta}=\{\beta_1,\dots, \beta_m\}$  be a maximum strongly orthogonal $\theta$-system and $\{\alpha_{\beta_1}, \dots, \alpha_{\beta_m}\}$ a set of simple roots satisfying the conditions of  Theorem \ref{theorem:cases_take2} and Remark \ref{remark:cases}.  For each $j$, there exists a unique nonzero element $H_j$  (up to scalar multiple) in $\mathcal{B}_{\theta}$ such that 
$H_j = X_{\beta_j} + C_j + s_{\beta_j}(K_{-\beta_j} -1) + Y_{-\beta_j}$
where $s_{\beta_j}$ is a (possibly zero) scalar and 
\begin{itemize}
\item[(i)] $X_{\beta_j}\in G^+_{\beta_j}$, $Y_{-\beta_j}\in U^-_{-\beta_j}$, and $
C_j\in U_{{\rm Supp}(\beta_j)}\ \cap\  G^-_{(\tau, 1)}U^0U^+_{(\tau, 1)}$
where $\tau = \{\alpha_{\beta_j}, p(\alpha_{\beta_j})\}$.
\item[(ii)] $X_{\beta_j}\in [({\rm ad}\ U^+) K_{-2\nu_j}]K_{-\beta_j+2\nu_j}$ and $Y_{-\beta_j}\in [({\rm ad}\ U^-) K_{-2\nu_j}]K_{-\beta_j+2\nu_j}$ where $\nu_j$ is the fundamental weight associated to the simple root $\alpha_{\beta_j}$.
 \item[(iii)] $\{X_{\beta_j},Y_{-\beta_j}\} \subset (U_q(\mathfrak{g}))^{U_{\pi_j}\mathcal{T}_{\theta}}$ and $C_j\in (U_q(\mathfrak{g}))^{(\mathcal{B}_{\theta}\cap U_{\pi_j})\mathcal{T}_{\theta}}$ where $\pi_j = {\rm StrOrth}(\beta_j)$.
  \item[(iv)] If $\beta_j$ satisfies Theorem \ref{theorem:cases_take2} (5) then $Y_{\beta_j }\in [({\rm ad}\ F_{-\alpha'_{\beta}})({\rm ad}\ U^-)K_{-2\nu_{j}}]K_{2\nu_j-\beta_j}$.
 \item[(v)] $\kappa_{\theta}(H_j) =  H_j$ for each $j$. 
\end{itemize}Moreover,  $\mathbb{C}(q)[\mathcal{T}_{\theta}][H_1,\dots, H_m]$ is a commutative polynomial ring over $\mathbb{C}(q)[\mathcal{T}_{\theta}]$ in $m$ generators  that specializes to the enveloping algebra of a  Cartan subalgebra of $\mathfrak{g}^{\theta}$ as $q$ goes to $1$. 
\end{theorem}

\begin{proof}  Consider $\gamma \in Q^+(\pi)$ and   $\alpha_i\in \pi\setminus {\rm Supp}(\gamma)$.  It follows that $(\alpha_i, \gamma) \leq 0$ with equality if and only if $(\alpha_i,\alpha) = 0$ for all simple roots $\alpha\in {\rm Supp}(\gamma)$. Therefore $E_i, F_i, K_i^{-1}$ commute with all elements of $G^+_{\gamma}$ and all elements of  $U^-_{-\gamma}$ for all $\alpha_i\in {\rm Orth}(\gamma)\setminus {\rm Supp}(\beta)$.  Hence, in proving (iii), we need only consider those $\alpha_i\in {\rm StrOrth}(\beta_j)\cap{\rm Supp}(\beta_j)$.

Suppose  that for each $\beta_j\in \Gamma_{\theta}$, we have identified an  element $H_j= X_{\beta_j} + C_j + s_{\beta_j}(K_{-\beta_j} -1) + Y_{-\beta_j}$ so that $X_{\beta_j}, Y_{-\beta_j},$ and $C_j$ each satisfy the relevant parts of  (i), (ii),  (iii), (iv), and, moreover, the following two properties hold:
\begin{itemize}
\item[(a)] $Y_{-\beta_j}$ specializes to $f_{-\beta_j}$ as $q$ goes to $1$.
\item[(b)] $\kappa_{\theta}(Y_{-\beta_j})$ is a nonzero scalar multiple of $X_{\beta_j}$.  
\end{itemize}   Note that these two properties  ensure $X_{\beta_j}$ specializes to $e_{\beta_j}$. Hence, by Lemma \ref{lemma:specialization},  $H_j$ specializes to $e_{\beta_j}+ f_{-\beta_j}$ and $\kappa_{\theta}(H_j) = H_j$ after a suitable scalar adjustment.  Thus the assertion concerning specialization follows from  these conditions.  

Note that  by construction of $\Gamma_{\theta}$, we have 
$(\gamma, \beta_j) = 0$ for all $\gamma\in Q(\pi)^{\theta}$ and all $j=1, \dots, m$.   Hence for each $j=1, \dots, m$ and each $K_{\gamma}\in \mathcal{T}_{\theta}$, $[K_{\gamma}, H_j] $, which is an element of $\mathcal{B}_{\theta}$, equals $[K_{\gamma}, C_j]$.   By assumption on $C_j$, we also have  $[K_{\gamma}, C_j]\in G^-_{(\tau,1)}U^0U^+_{(\tau, 1)}$.   Since $\tau\cap \pi_{\theta}=\emptyset$, it follows that $\mathcal{M}^+\cap U^+_{(\tau, 1)} =0$ and so  $\mathcal{P}([K_{\gamma}, C_j])=0$.  By Corollary \ref{corollary:pmapzero},  $[K_{\gamma}, C_j] =0$.  Thus each $H_j$ commutes with every element of $\mathcal{T}_{\theta}$.  By Theorem \ref{theorem:cases_take2} (i), ${\rm Supp}(\beta_i)\subseteq {\rm StrOrth}(\beta_j)$ for all $i>j$.   Thus (i) and (iii) ensure that  $H_j$ commutes with $H_i$ for $i>j$ and so the elements $\{H_1, \dots, H_m\}$ pairwise commute. 
Hence $\mathbb{C}(q)[\mathcal{T}_{\theta}][H_1, \dots, H_m]$ is a  commutative ring.   The property that each $H_j$ specializes to $e_{\beta_j} + f_{-\beta_j}$ implies that $\mathbb{C}(q)[\mathcal{T}_{\theta}][H_1, \dots, H_m]$ specializes to the enveloping algebra of the Cartan subalgebra of $\mathfrak{g}^{\theta}$, which is a polynomial ring over $\mathbb{C}$.   This guarantees that there are no additional 
 relations among the $H_j$ and so $\mathbb{C}(q)[\mathcal{T}_{\theta}][H_1, \dots, H_m]$ is  a  polynomial ring in $m$ generators over $\mathbb{C}(q)[\mathcal{T}_{\theta}]$.

For the uniquess assertion, suppose that $H_j'$ also satisfies the conclusions of the theorem.  
Note that conditions (i), (ii), (iii), (iv) and the specialization assertion are enough to ensure that each $Y_{\beta_j}$ corresponds to $Y_j$ of Theorem \ref{theorem:lift}.  Hence, rescaling if necessary, we get that 
$H_j - H_j' \in G^-_{(\tau,1)}U^0U^+_{(\tau,1)} + G^+_{\beta_j} + \mathbb{C}(q)(K_{-\beta_j}-1)$.  Hence $\mathcal{P}(H_j-H_j')$ is a scalar, say $s$ in $\mathbb{C}(q)$.  Since $H_j-H_j'-s\in \mathcal{B}_{\theta}$, it follows from Corollary \ref{corollary:pmapzero} that $H_j=H_j'+s$.
The fact that both $H_j$ and $H_j'$ are in $U^-_{-\beta_j}+ G^-_{(\tau,1)}U^0U^+_{(\tau,1)} + G^+_{\beta_j} + \mathbb{C}(q)(K_{-\beta_j}-1)$ forces $s=0$ and hence $H_j=H_j'$.

  We complete the proof by identifying the $H_j$ whose summands  satisfy (i), (ii), (iii), (iv) and the two properties (a) and (b) highlighted above. The proof follows the same breakdown into cases  of possible roots in $\Gamma_{\theta}$ as in the proof of Theorem \ref{theorem:lift}. Note that (iv) only comes into play in  Case 5. We drop the subscript $j$, writing $\beta$ for $\beta_j$, $H$ for $H_j$, etc.  Set $w_{\beta} = w({\rm Supp}(\beta)\setminus \{\alpha_{\beta}, \alpha'_{\beta}\})_0$.

\bigskip
\noindent
{\bf Case 1:} $\beta = \alpha_{\beta}=\alpha_{\beta}'$.  Choose $k$ so that $\alpha_k = \alpha_{\beta} = \beta$.  Since $\theta(\alpha_k) = \theta(\beta) = -\beta = -\alpha_k$, we see that 
$H = B_k-s_k = F_k + E_kK_k^{-1} + s_k(K_k^{-1}-1)$ is in $\mathcal{B}_{\theta}$.  Moreover, this element specializes to $f_k+e_k$ as $q$ goes to $1$.    It is straightforward to check that $H$ satisfies (i) - (iii),  (a), and (b).

\bigskip
\noindent
{\bf Case 2:}  $\beta = \alpha_{\beta}+w_{\beta}\alpha_{\beta}$ and $\alpha_{\beta} = \alpha'_{\beta}$. 
By Theorem \ref{theorem:lift}, there exists $Y_{-\beta}$ that satisfies (i), (ii) and (a).  Write 
\begin{align*} Y_{-\beta} = \sum_{{\rm wt}(I)=\beta}  F_Ia_I
\end{align*}
for appropriate scalars $a_I$.  By Proposition \ref{proposition:liftbeta}, there exists $a\in \mathcal{B}_{\theta}$ such that 
\begin{align*} 
\mathcal{P}(a)  = \sum_{{\rm wt}(I) = \beta} F_Ia_I,
\end{align*}
 the lowest weight summand of $a$ is $Y_{-\beta}$, the highest weight summand of $a$ is a term $X_{\beta}\in G^+_{\beta}$ (since $\theta(-\beta) = \beta$), and $X_{\beta}$ commutes with all elements of $U_{\pi'}$.  Hence $\kappa_{\theta}(X_{\beta})$ is an element of $U^-_{-\beta}$ that commutes with all elements of $U_{\pi'}$. 
  It  
follows from  Corollary \ref{theorem:invariant_elements2},  that  $\kappa_{\theta}(X_{\beta})$ is a scalar multiple of $Y_{-\beta}$ and so (b) holds.  By (\ref{kappaeq}) and the second half of assertion (ii), $X_{\beta}$ satisfies  (ii).  Hence $X_{\beta}$ and $Y_{-\beta}$ satisfy the assertions in (i) - (iii), (a), and (b).

By Proposition \ref{proposition:liftbeta}, $a  =  X_{\beta} + C + sK_{-\beta}+ Y_{-\beta}$ for some nonzero scalar $s$ and element $C$ in $U_q(\mathfrak{g})$ satisfying (i) and (iii).   Thus, the desired Cartan element  is $H = a -s = X_{\beta} + C + s(K_{-\beta}-1)+Y_{-\beta}$ in this case.

\bigskip
\noindent
{\bf Case 3:} $\beta = p(\alpha_{\beta}) + w_{\beta}\alpha_{\beta}=w\alpha_{\beta}$ where $w=w({\rm Supp}(\beta)\setminus\{\alpha_{\beta})_0$.  Arguing as in Case 2 yields an element $H = Y_{-\beta} + C + X_{\beta} + s(K_{-\beta}-1)$ so that $Y_{-\beta}$ 
satisfies the relevant parts of (i), (ii), (iii), (a) and $C$ satisfies the relevant parts of (i) and (iii). 

Choose $i$ so that $\alpha_{\beta} = \alpha_i$.  By Theorem \ref{theorem:cases_take2} (3) and Remark \ref{remark:cases}, we may assume that $\Delta({\rm Supp}(\beta))$ is a  root system of type A and the symmetric pair $\mathfrak{g}_{{\rm Supp}(\beta)}, 
 (\mathfrak{g}_{{\rm Supp}(\beta)})^{\theta}$ is of type AIII/AIV.  Corollary \ref{corollary:XtoY} (b) further ensures that 
\begin{align*}
F_{i}Y_{-\beta}-qY_{-\beta}F_{i} = 0.
\end{align*}
On the other hand, it follows from properties of $C$ that 
\begin{align}\label{Feqn}F_i(H-Y_{-\beta}) -q(H-Y_{-\beta})F_i \in \sum_{m\leq 2}G^-_{(\tau,m)}U^0U^+.
\end{align}
It follows from the discussion concerning $\mathcal{S}$ in Section \ref{section:Definitions} that $\alpha_i\notin \mathcal{S}$ and so 
$B_i = F_i +c_i\theta_q(F_iK_i)K_i^{-1} $ for a suitable scalar $c_i$ (i.e. $s_i=0$ and there is no $K_i^{-1}$ term in $B_i$.)  Since $\theta(-\beta) = \beta$, we see from  (\ref{Feqn}) and Remark \ref{remark:decomp} that 
\begin{align*}
B_iH - qHB_i \in \sum_{m\leq 2} \sum_{r\leq 2} G^-_{(\tau,m)}U^0U^+_{(\tau,r)}.
\end{align*}
On the other hand, 
\begin{align*}
\left(\theta_q(F_iK_i)K_i^{-1} \right)X_{\beta} - q X_{\beta}\left(\theta_q(F_iK_i)K_i^{-1} \right)
\end{align*}
either equals $0$ or has weight $\beta+\theta(-\alpha_i)$. Since ${\rm ht}_{\tau}(\beta)+\theta(-\alpha_i)) = 3$,  this term must be $0$. By Corollary
\ref{corollary:XtoY},  $\kappa_{\theta}(X_{\beta})$ is a nonzero scalar multiple of $Y_{-\beta}$, thus (b) holds.  Hence,  $H$ satisfies the conclusions of the lemma.

\medskip
\noindent
{\bf Case 4:} $\beta = \alpha'_{\beta} + w_{\beta}\alpha_{\beta} = w\alpha'_{\beta}$ where  $w = w({\rm Supp}(\beta)\setminus \{\alpha'_{\beta}\})_0$.  More precisely, we assume that $\beta$ satisfies the conditions of Theorem \ref{theorem:cases_take2}  (4) as further clarified by Remark \ref{remark:cases}.   Hence, $\beta =w\alpha_{\beta}'$.  Moreover, we may assume that  ${\rm Supp}(\beta)=\{\gamma_1, \dots, \gamma_s\}$ generates a root system of type $B_s$ and is ordered in the standard way so that $\alpha'_{\beta}.=\gamma_s$ is the unique shortest simple root and $\alpha_{\beta} = \gamma_1$.  

Let $m$ be chosen so that the $m^{th}$ simple root of $\pi$, namely $\alpha_m$, equals the first root $\gamma_1$ of ${\rm Supp}(\beta)$. It follows from the description of the generators for $\mathcal{B}_{\theta}$ in Section \ref{section:Definitions} that \begin{align*}B_m = F_{-\gamma_1} + c[({\rm ad\ }E_{\gamma_2}E_{\gamma_3}\cdots E_{\gamma_{s-1}} E_{\gamma_s}^2 E_{\gamma_{s-1}}\cdots E_{\gamma_2})E_{\gamma_1}]K_{\gamma_1}^{-1}
\end{align*} for some nonzero scalar $c$.  Note that if $\alpha_i\in \pi_{\theta}$ and $b\in \mathcal{B}_{\theta}$ then
\begin{align*} [({\rm ad} F_i)bK_j] K_j^{-1} = [F_ibK_j K_i - bK_j F_iK_i]K_j^{-1} = F_ibK_i -q^{-(\alpha_i, \alpha_j)}bF_iK_i
\end{align*}
is also an element of $\mathcal{B}_{\theta}$ for all choices of $j$.
Hence, 
since $\gamma_j\in \pi_{\theta}$ for  $j=2,\dots, s$, the element
\begin{align*}H' = &[({\rm ad\ }F_{-\gamma_s}F_{-\gamma_{s-1}}\cdots F_{-\gamma_{2}} ) B_mK_{\gamma_1}] K_{\gamma_1}^{-1}=\cr
& [ ({\rm ad\ }F_{-\gamma_s}F_{-\gamma_{s-1}}\cdots F_{-\gamma_{2}} )F_{-\gamma_1}K_{\gamma_1}] K_{\gamma_1}^{-1}+ c'[({\rm ad\ }E_{\gamma_s}E_{\gamma_{s-1}}\cdots E_{\gamma_2})E_{\gamma_1}]K_{\gamma_1}^{-1}
\end{align*}
is also in $\mathcal{B}_{\theta}$ where $c'$ is a suitably chosen nonzero scalar.  
Note that $$K_{-\beta}= K_{\gamma_s}^{-1} K_{\gamma_{s-1}}^{-1}\cdots K_{\gamma_2}^{-1} K_{\gamma_1}^{-1} \in \mathcal{T}_{\theta}K_{\gamma_1}^{-1}.$$
Hence $K_{\gamma_1}K_{-\beta}$ is in $\mathcal{B}_{\theta}$. Set 
\begin{align}\label{XYdefn}
X_{\beta} &= c'[({\rm ad\ }E_{\gamma_s}E_{\gamma_{s-1}}\cdots E_{\gamma_2})E_{\gamma_1}]K_{-\beta}{\rm \ and \ }\cr
Y_{-\beta} &= [ ({\rm ad\ }F_{-\gamma_s}F_{-\gamma_{s-1}}\cdots F_{-\gamma_{2}} )F_{-\gamma_1}K_{\gamma_1}] K_{-\beta}.
\end{align}  and note that $H =H'K_{-\beta+\gamma_1} = X_{\beta} + Y_{-\beta}$ is in $\mathcal{B}_{\theta}$.
We show that $H$ satisfied (i), (ii), (iii), (a) and (b) with $C=0$.

Note that $E_{\gamma_1} = d[({\rm ad}\ E_{\gamma_1})K_{-2\nu_1}]K_{-2\nu_1}$ and $F_{-\gamma_1} K_{\gamma_1}=d' [({\rm ad}\ F_{-\gamma_1})K_{-2\nu_1}]K_{-2\nu_1}$ for nonzero 
scalars $d$ and $d'$. Hence by (\ref{XYdefn}),  $X_{\beta}$ and $Y_{-\beta}$  both satisfy assertion (ii). 
Assertions (i)  and (a) follows from the construction of $X_{\beta}$ and $Y_{-\beta}$ and the definition of the adjoint action (Section \ref{section:basic}). Assertion (b) follows from (\ref{kappaeq}) of Section \ref{section:chevalley}.
 Assertion (iii) for $X_{\beta}$ follows from the fact that $X_{\beta}$ is a highest weight vector with respect to the action of $({\rm ad}\ U_{{\rm Supp}(\beta)\setminus\{\gamma_1,\gamma_2\}})$ and has weight $\beta$ with $(\beta, \gamma_j) = 0$ for $j = 2,\dots, s-1$.  Assertion (iii) for $Y_{-\beta}$ now follows from (b).  

\medskip
\noindent
{\bf Case 5:} $\beta = \alpha_{\beta}' + \alpha_{\beta} + w_{\beta}\alpha_{\beta}$ and $\alpha_{\beta}\neq \alpha_{\beta}'$.    Set $\beta' = \beta-\alpha'_{\beta}$. Recall from Theorem \ref{theorem:cases_take2}  that $\alpha_{\beta}'\in \pi_{\theta}$, $\alpha_{\beta}'$ is strongly orthogonal to all roots in ${\rm Supp}(\beta')\setminus \{\alpha_{\beta}\}$, and $(\alpha_{\beta}',\beta)=0$. Note also that $(\alpha'_{\beta}, \beta') =- (\alpha'_{\beta}, \alpha'_{\beta})$ and 
$\theta(-\beta') = \beta' + 2\alpha_{\beta}'$. 

By the proof of Theorem \ref{theorem:lift} Case 5, there  exists $Y_{-\beta'}$ that satisfies (i), (ii), and (a) with respect to $\beta'$ instead of $\beta$.  Write 
\begin{align*} Y_{-\beta'} = \sum_{{\rm wt}(I)=\beta'}  F_Ia_I
\end{align*}
for appropriate scalars $a_I$. By Proposition \ref{proposition:liftbeta}, there exists 
\begin{align*}H' = Y_{-\beta'} + C' + X_{\theta(-\beta')}K_{2\alpha_{\beta}'} + sK_{-\beta'}
\end{align*} in $\mathcal{B}_{\theta}$ such that $Y_{-\beta'}\in U^-_{-\beta'}$, $X_{\theta(-\beta')}\in G^+_{\theta(-\beta')}$,  $C' \in U_{{\rm Supp}(\beta')}\cap G^-_{(\tau,1)}U^+_{(\tau,1)}U^0$, and $s$ is a scalar.   Moreover, each of these terms  commute with all elements of $U_{\pi'}$.  

Let $k$ be the index so that $\alpha_{\beta}' = \alpha_k$ and recall that $(\alpha_k, \nu) = 0$ where $\nu$ is the fundamental weight associated to $\alpha_{\beta}$.  It follows 
that $({\rm ad}\ F_k)(uK_{-2\nu})K_{2\nu} =  ({\rm ad}\ F_k)u$ for all $u$.  Thus in much of  the discussion below, we ignore the $K_{-2\nu}$ term which can always be added in later.   Note that $F_k\in \mathcal{B}_{\theta}$ since $\alpha_k\in \pi_{\theta}$.  
Note that 
\begin{align*}
 [F_kK_{-\beta'}-q^{-(\beta',\alpha_k)}K_{-\beta'}F_k]=0.
\end{align*} 
Set \begin{align*}
H = [F_kH'-q^{-(\beta',\alpha_k)}H'F_k]=Y_{-\beta}+ C + X_{\beta}
\end{align*} where $C = [F_kC' - q^{-(\beta',\alpha_k)}C'F_k]K_k$,
\begin{align}\label{ybeta}
Y_{-\beta} =  
(F_kY_{-\beta'} - q^{-(\beta',\alpha_k)}Y_{-\beta'}F_k)=
 [({\rm ad}\ F_k)(Y_{-\beta'}K_{\beta'})]K_{-\beta} 
\end{align}
and 
\begin{align*}
X_{\beta} &= 
F_k(X_{\theta(-\beta')}K_{\theta(\beta')}K_{-\beta'} -q^{-(\beta', \alpha_k)} (X_{\theta(-\beta')}K_{\theta(\beta')}K_{-\beta'})F_k \cr &= 
[ ({\rm ad}\ F_k) (X_{\theta(-\beta')}K_{\theta(\beta')})] K_{-\beta}.
\end{align*}
We argue that $H$ and its summands satisfy (i) - (iv), (a) and (b).  We see from (\ref{ybeta}) (as in  the proof of Case 5 of Theorem \ref{theorem:lift}) that $Y_{-\beta}$ satisfies the relevant properties of (i), (ii) and (a) since $Y_{-\beta'}$ does.   Note that (\ref{ybeta}) also ensures that $Y_{-\beta}$ satisfies (iv).  By Proposition \ref{proposition:liftbeta}, 
\begin{align*}
\{X_{\theta(-\beta')}, Y_{-\beta'}\}\subseteq (U_q(\mathfrak{g}))^{U_{\pi_j}}{\rm \ and \ }C' \in (U_q(\mathfrak{g})^{\mathcal{B}_{\theta}\cap U_{\pi'}}
\end{align*}
where $\pi' = {\rm Supp}(\beta')\setminus \{\alpha_{\beta}\}$.  Hence, since $F_k$ commutes with all elements in $U_{\pi'}$, it follows from  these inclusions and  the discussion preceding the case work for this proof that $X_{\beta}, Y_{-\beta},$ and $C$ satisfy (iii).   So we only need to show that $X_{\beta}$ satisfies (i) and (ii) and $X_{\beta}, Y_{-\beta}$ are related as in (b).

Since $C'$ is an element of $ G^-_{(\tau,1)}U^0U^+_{(\tau,1)}$, so is $C$.   Thus $\mathcal{P}(C) = 0$.   Note that we also have $X_{\beta} \in G^-U^0U^+_{(\tau,2)}$ and so $\mathcal{P}(X_{\beta}) = 0$.   
Set 
\begin{align*} b = F_k\left(\sum_{{\rm wt}(I)=\beta'}B_Ia_I\right) - q^{-(\beta', \alpha_k)}\left(\sum_{{\rm wt}(I)=\beta'}B_Ia_I\right)F_k.
\end{align*} 
and define $\tilde b$ as in   Proposition \ref{proposition:BUiso}.  Note that $\mathcal{P}(\tilde b) = \mathcal{P}(H) = Y_{-\beta}$.   Hence by Corollary \ref{corollary:pmapzero}, $H = \tilde b$.   
Note that in addition to $\beta'$, $\beta$ also satisfies the conditions of Proposition \ref{proposition:liftbeta}. This is because $\alpha_k\in {\rm Orth}(\beta)$ even if it is not strongly orthogonal to $\beta$ and so $\beta$ satisfies the required assumption on height.  Note further that we may choose $a=H$ since $H$ satisfies Proposition \ref{proposition:liftbeta} (i) and (ii).  Hence by  Proposition \ref{proposition:liftbeta} we get
\begin{align*}
X_{\beta} \in U^+_{\beta}K_{-\beta} =G^+_{\beta}
\end{align*}
and so $X_{\beta}$ satisfies assertion (i) of this theorem.  Hence $X_{\beta}$ satisfies both (i) and (iii). 

We next show that  
\begin{align}\label{EkFkX}  ({\rm ad}\ E_kF_k^2)(X_{\theta(-\beta')}K_{\theta(-\beta')}) = 
({\rm ad}\ F_k)(X_{\theta(-\beta')}K_{\theta(-\beta')}) \end{align}
up to a nonzero scalar. 
The fact that $\alpha_k\in \pi_{\theta}$ ensures that $E_k, F_k, K_k^{\pm 1}\in \mathcal{B}_{\theta}$.  Hence 
\begin{align*}[({\rm ad}\ E_k)(H'K_{\beta'})]K_{-\beta'} &\cr
=[({\rm ad}\ E_k)&\left(Y_{-\beta'}K_{\beta'}+C'K_{\beta'}+ X_{\theta(-\beta')}K_{\theta(-\beta')}+s\right)]K_{-\beta'} \in \mathcal{B}_{\theta}.
\end{align*}
 Note that  $({\rm ad}\ E_k)(Y_{-\beta'}K_{\beta'})=0$ since $\alpha_k\notin {\rm Supp}(\beta')$.  It is straightforward to check that $({\rm ad}\ E_k)s=0$ and   $({\rm ad}\ E_k)(X_{\theta(-\beta')}K_{\theta(-\beta')})\in U^+_{\theta(-\beta')+\alpha_k}$. Also, $[({\rm ad}\ E_k) (C'K_{\beta'})]K_{-\beta'}\in G_{(\tau,1)}^-U^0U^+_{(\tau,1)}$ since the same is true for $C'$.    It follows that 
 \begin{align*}[({\rm ad}\ E_k)(H'K_{\beta'}] K_{-\beta'}\in G_{(\tau,1)}^-U^0U^+_{(\tau,1)}+ 
U^+_{\theta(-\beta')+\alpha_k}\subseteq \sum_{m\geq 1}G^-U^0U^+_{(\tau,m)}.
\end{align*}  By Remark \ref{remark:decomp}, we get  $[({\rm ad}\ E_k)(H' K_{\beta'})]K_{-\beta'}= 0$.  Since $(G_{(\tau,1)}^-U^0U^+_{(\tau,1)})\cap (U^0G^+_{\theta(-\beta')+\alpha_k}) = 0$, we must have $({\rm ad}\ E_k)(X_{\theta(-\beta')}K_{\theta(-\beta')})=0$.  Equality (\ref{EkFkX}) now follows from rewriting 
$E_kF_k^2$ as a linear combination of $F_k^2E_k$ and $F_ku$ for some $u\in U^0$ using  the defining relation (\ref{commuting_relation}) for $U_q(\mathfrak{g})$.  

Equality (\ref{EkFkX}) and the definition of $X_{\beta}$ implies that 
\begin{align} \label{anotherEkFk}({\rm  ad}\ E_kF_k)(X_{\beta}K_{\beta}) = X_{\beta}K_{\beta}
\end{align} 
up to a  nonzero scalar. Using (\ref{commuting_relation}) we see that  
\begin{align}\label{yetanotherreln}({\rm ad}\ F_k)(X_{\beta}K_{\beta})\subseteq  U^+ + U^+K_k^2 +U^+F_kK_k.
\end{align}  Note that $({\rm ad}\ E_k)U^+K_k^2\subseteq U^+K_k^2$ while $({\rm ad}\ E_k)uF_kK_k \in U^+F_kK_k+U^+(K_k^2-1)$ for any choice of $u\in U^+$.    Since $X_{\beta}K_{\beta}\in U^+$, it follows from (\ref{anotherEkFk}),   (\ref{yetanotherreln}) and the fact that $X_{\beta}$ satisfies (iii) that 
\begin{align*}({\rm ad}\ F_k)(X_{\beta}K_{\beta})\in (U^+_{\beta-\alpha_k})^{U_{\pi'}}=(U^+_{\beta'})^{U_{\pi'}}.\end{align*}

Applying $\kappa$ to $(({\rm ad}\ F_k)(X_{\beta}K_{\beta}))K_{-\beta'}$ yields an element $Y_{-\beta'}'\in (U^-_{-\beta'})^{U_{\pi'}}$.   By the uniqueness assertion of Corollary \ref{theorem:invariant_elements2}, it follows that $Y_{-\beta'}'= Y_{-\beta'}$ up to nonzero scalar. By identity (\ref{kappaeq}), (\ref{anotherEkFk}) and (\ref{ybeta}),  we see that 
\begin{align*}\kappa(Y_{-\beta}K_{\beta}) =\kappa(({\rm ad}\ F_k)Y_{-\beta'}K_{\beta'})= ({\rm ad}\ E_k)\kappa(Y_{-\beta'}K_{\beta'})= ({\rm ad}\ E_kF_k)(X_{\beta}K_{\beta})=X_{\beta}K_{\beta}
\end{align*} up to  nonzero scalars and so (b) holds.  Since $Y_{-\beta'}$ satisfies (ii) with $\beta$ replaced by $\beta'$, it now follows from  the identity (\ref{kappaeq}) that $X_{\beta}K_{\beta} \in [({\rm ad}\ U^+)K_{-2\nu}]K_{-\beta+2\nu}$ where recall that $\nu$ is the fundamental weight associated to $\alpha_{\beta}$. Hence,   we see that $X_{\beta}$ satisfies (ii). 
\end{proof}

We refer to the algebra $\mathcal{H} = \mathbb{C}(q)[\mathcal{T}_{\theta}][H_1, \dots, H_m]$ of Theorem \ref{theorem:main} as a quantum Cartan subalgebra of $\mathcal{B}_{\theta}$.   
The next result is the first step in understanding finite-dimensional $\mathcal{B}_{\theta}$-modules with respect to the action of $\mathcal{H}$.  It is an immediate consequence of Corollary \ref{cor:unitary1} and Corollary \ref{Haction2}.

\begin{corollary} \label{corollary:main} Let  $\mathcal{H}$ be a quantum Cartan subalgebra of the coideal subalgebra $\mathcal{B}_{\theta}$  associated to the symmetric pair $\mathfrak{g},\mathfrak{g}^{\theta}$.  Any finite-dimensional unitary $\mathcal{B}_{\theta}$-module can be written as a direct sum of eigenspaces with respect to the action of $\mathcal{H}$.  Moreover, any finite-dimensional $U_q(\mathfrak{g})$-module can be written as a direct sum of eigenspaces with respect to the action of $\mathcal{H}$.
\end{corollary}

\begin{remark} \label{nonstandard}
Assume that  $\mathfrak{g},\mathfrak{g}^{ \theta}$ is a symmetric pair of type AI.   
The algebra $\mathcal{B}_{\theta}$ in this case  is referred to in the literature as the (nonstandard) $q$-deformed algebra $U'_q(\mathfrak{so}_n)$ and its  finite-dimensional modules have been analyzed and classified by Klimyk, Iorgov and Gavrilik (see for example \cite{GK} and \cite{IK}) using quantum versions of Gel'fand-Tsetlin basis.  Rowell and Wenzyl (\cite{RW}) revisit the representation theory of this coideal subalgebra.  They take  a different approach that involves developing a highest weight module theory using   a Cartan subalgebra $\mathcal{H}$.  This Cartan subalgebra $\mathcal{H}$ agrees with the one  produced by the above theorem using Theorem \ref{theorem:cases_take2}.  In particular, in this case, $\mathfrak{h}_{\theta} = 0$,   $\Gamma_{\theta} = \{\alpha_1, \alpha_3, \cdots, \alpha_{\lfloor{(n+1)/2}\rfloor}\}$, and    $\mathcal{H} =\mathbb{C}(q)[ B_1, B_3,\dots, B_{\lfloor{(n+1)/2}\rfloor}]$. For more details on the representation theory of $\mathcal{B}_{\theta}$ developed using $\mathcal{H}$, the reader is referred to \cite{RW}.
\end{remark}

\begin{remark}
Suppose that $\beta= p(\alpha_{\beta})+ w_{\beta}\alpha_{\beta}$ with $\alpha_{\beta}\neq p(\alpha_{\beta})$ as in Case 3 of the proof of the above theorem. Then $\theta$ restricts to an involution of type AIII/AIV.  Moreover,  the  element  $\alpha_{\beta}$ can be switched with $p(\alpha_{\beta})$ in the construction of the corresponding Cartan element associated to $\beta\in \Gamma_{\theta}$. The end result is still a commutative polynomial ring specializing to the enveloping algebra of the same Cartan subalgebra of $\mathfrak{g}^{\theta}$.  However, the Cartan element corresponding to $\beta$ is different for the two constructions.    We take a closer look at such roots in the analysis of a large family of symmetric pairs of type AIII/AIV in Section \ref{section:examples} below.
\end{remark}

\section{ A Family of Examples: Type AIII/AIV}\label{section:examples}

In this section, we consider  the following family of examples:   symmetric pairs of type AIII/AIV with $\pi_{\theta} = \emptyset$. Note that the discussion presented here has applications to other symmetric pairs as well.  In particular, consider for a moment the case where $\mathfrak{g},\mathfrak{g}^{ \theta}$ is an arbitrary symmetric pair with corresponding maximum strongly orthogonal $\theta$-system $\Gamma_{\theta}$.  Suppose that there is a root  $\beta\in \Gamma_{\theta}$ such that 
$\beta= p(\alpha_{\beta})+ w_{\beta}\alpha_{\beta}$ with $\alpha_{\beta}\neq p(\alpha_{\beta})$ and $\Delta({\rm Supp}(\beta))$ is a root system of type A as in Case 3 of the proof of the Theorem \ref{theorem:cases_take2}.   By Remark \ref{remark:cases}, the involution $\theta$ restricts to an involution of type AIII/AIV.  If in addition, $\pi_{\theta} =\emptyset$, then the discussion in this section applies to the construction of the Cartan element associated to $\beta$.  

The presentation in this section was in part inspired by conversations with Stroppel (\cite{C}) who has analyzed the finite-dimensional simple modules of the quantum symmetric pair coideal subalgebras in type AIII/AIV.  In \cite{Wa1}, Watanabe also looks at symmetric pairs of type AIII with $\pi_{\theta}=0$ and restricts to the case when $\mathfrak{g}$ is of type $A_{n}$ where $n$ is even. He obtains a triangular decomposition for $\mathcal{B}_{\theta}$ for this subfamily where the components are each described as subspaces (see \cite{Wa1}, Remark 2.2.10). We see below that the Cartan part of the triangularization in \cite{Wa1} agrees with the Cartan subalgebra of this paper, and thus this subspace is actually a   subalgebra of $\mathcal{B}_{\theta}$.

\subsection{Overview}\label{section:ex_overview}

Assume that  $\mathfrak{g}, \mathfrak{g}^{\theta}$ is a symmetric pair of type AIII/AIV 
with $r= \lfloor(n+1)/2\rfloor$  where  $n\geq 2$ and $r$ is 
as in the proof of Theorem \ref{theorem:cases_take2} for Type AIII/AIV. 
The  involution $\theta $ on $\mathfrak{g}$ is defined by 
 $\theta(\alpha_i) = -\alpha_{n-i+1}$ for $i = 1, \dots, n$.
By Theorem \ref{theorem:nice_cartans} and Theorem \ref{theorem:cases_take2}, the set \begin{align*}
\mathfrak{h}_{\theta}\oplus \left(\bigoplus_{\beta\in \Gamma_{\theta}}\mathbb{C}(e_{\beta} +f_{-\beta})\right)
\end{align*} is a Cartan subalgebra for $\mathfrak{g}^{\theta}$ where 
$\mathfrak{h}_{\theta}={\rm span}_{\mathbb{C}}
\{ h_i-h_{n-i+1}|\ i = 1, \dots,r\}$ , 
$\Gamma_{\theta} = \{\beta_1,\cdots, \beta_{r}\}$, and 
$$\beta_j = \alpha_{j} + \alpha_{j+1} + \cdots + \alpha_{n-j+1}$$
  for $j=1, \cdots,r$.

The right coideal subalgebra $\mathcal{B}_{\theta}$ is generated over 
$\mathbb{C}(q)$ by \begin{align*}
B_i = F_i + E_{n-i+1}K_i^{-1}
\end{align*}
and 
 \begin{align*}K_iK_{n-i+1}^{-1}
 \end{align*}
 for $i=1,\dots, n$.
(In the notation of Section \ref{section:Definitions},  we are taking $c_i = -1$ and $s_i=0$  for each $i$. Other choices of parameters yields analogous results.  Note that $\mathcal{S}=\{\alpha_r\}$ if $r$ is odd and is empty otherwise, so that only time $s_i$ could take on nonzero values is when $r$ is odd and $i=r$.)
The $B_i$, $i=1,\dots, n$ satisfy the following relations (see for example \cite{Ko}, Theorem 7.4):
\begin{itemize}
\item For $i , j \in \{ 1, \dots, n\}$ with $a_{ij} = 0$, we have 
\begin{align*} B_i B_{j} - &B_{j} B_i 
=\delta_{n-i+1, j}(q-q^{-1})^{-1}(K_{n-i+1}K_i^{-1} - K_{n-i+1}^{-1}K_i)
\end{align*}
\item For $i , j \in \{ 1, \dots, n\}$ with $a_{ij} = -1$, we have 
\begin{align*}
B_i^2 B_j - &(q+q^{-1}) B_iB_j B_i + B_jB_i^2 =
-\delta_{i,n-i+1} qB_j \cr&+\delta_{i, n-j+1} (q + q^{-1})(qK_{n-i+1}K_i^{-1}+ q^{-2}K_{n-i+1}^{-1}K_i)B_i.
\end{align*}
\end{itemize}
Note that  $a_{ij} = -1$ and $i = n-i+1$ if and only if $n$ is odd, $i = (n+1)/2=r$ and  $j = i\pm 1$.  Also, $a_{ij} =-1$ and  $i = n-j+1$ if and only if $n$ is even
and  $\{i,j\} = \{{{n}\over{2}}, {{n}\over{2}} + 1\}=\{r, r+1\}$.

  We also have the following relations:
\begin{align*}
(K_{n-i+1}K_i^{-1})B_j(K_{n-i+1}K_i^{-1})^{-1} = q^{(\alpha_{n-i+1}-\alpha_i,-\alpha_j)}B_j
\end{align*}
for all $i,j$.   It follows that $B_1, \dots,  B_{r-1}, B_{r+1}, \dots, B_n$ generates an algebra isomorphic to $U_{q}(\mathfrak{sl}_{r-1})$  if $n$ is odd and $B_1, \dots, B_{r-1}, B_{r + 2}, \cdots B_n$ generates an algebra isomorphic to $U_{q}(\mathfrak{sl}_{r-1})$  if $n$ is even.

\subsection{Case 1: $n=2$}\label{section:nequals2}  In this case, $\Gamma_{\theta} = \{\beta\}$ where $\beta = \alpha_1 + \alpha_2$. Note that $\alpha_{\beta} = \alpha_1$ and 
\begin{align*} (1-q^2)^{-1}[ ({\rm ad}\ F_2 F_1) K_{-2\nu} ]K_{-\beta+2\nu}= (F_2F_1-qF_1F_2)
\end{align*}
where $\nu=\nu_1$ is the fundamental weight associated to $\alpha_1$.  Also, $\kappa(F_2F_1 - q F_1F_2)$ is a nonzero scalar multiple of $(E_1E_2 - q E_2 E_1)K_1^{-1}K_2^{-1}$.  Set 
\begin{align}\label{HDEFN}H = B_2 B_1 - qB_1B_2 -(q-q^{-1})^{-1}[q^{-1}K-K^{-1}] - (1+q)
\end{align}
where $K = K_1K_2^{-1}$.
  A straightforward computation shows that
 \begin{align*}H=&(F_2F_1 - q F_1 F_2) + q^{-2}(E_1E_2 -q E_2 E_1)K_1^{-1}K_2^{-1}  \cr &+(1+q)(K_1^{-1}K_2^{-1}-1) +(q^{-1}-q)F_1E_1K_2^{-1}\end{align*}
  and hence $H$ satisfies the conclusions of Theorem \ref{theorem:main}.  The Cartan subalgebra for $\mathcal{B}_{\theta}$ in this case is $\mathcal{H}= \mathbb{C}(q)[K, K^{-1}][H]$.
 By (\ref{HDEFN}), we can use $B_2B_1 - qB_1B_2$ instead of $H$ as a generator for $\mathcal{H}$.

   It follows from (3.3) of  \cite{AKR} that $CK^{1/3}$ is in the center of the simply connected version $\check{\mathcal{B}}_{\theta}$ (see Section \ref{section:Definitions}) where 
  \begin{align*}
  C &= B_2B_1 - qB_1 B_2 + {{1}\over{(q-q^{-1})}}K^{-1} + {{(q+q^{-1})}\over{(q-q^{-1})}}K\cr
  &=H + {{(q+2q^{-1})}\over{(q-q^{-1})}} K+(1+q).
  \end{align*}
 Thus we have the following equivalent descriptions of $\mathcal{H}$ as a polynomial ring in one variable over the commutative ring $\mathbb{C}(q)[K,K^{-1}]$:
 \begin{align*}
 \mathcal{H} =  \mathbb{C}(q)[K, K^{-1}][H] = \mathbb{C}(q)[K, K^{-1}][B_2B_1 - qB_1 B_2] = \mathbb{C}(q)[K, K^{-1}][C]. 
 \end{align*}
 
Note that $\mathcal{B}_{\theta}$ admits a second quantum Cartan subalgebra that satisfies the conclusions of Theorem \ref{theorem:main}.  This second subalgebra is constructed using the same maximum orthogonal system $\Gamma_{\theta}=\{\beta\}$, but instead of choosing $\alpha_{\beta} = \alpha_1$, we pick  $\alpha_{\beta} = \alpha_2$.   
This second Cartan subalgebra is $\mathcal{H}' = \mathbb{C}(q)[K,K^{-1}][H']$ where $H'$ takes the same form as $H$ with $B_1$ and $B_2$ interchanged. In analogy to 
$\mathcal{H}$ above, we see that
 \begin{align*}
 \mathcal{H}'=  \mathbb{C}(q)[K, K^{-1}][H'] = \mathbb{C}(q)[K, K^{-1}][B_1B_2 - qB_2 B_1] = \mathbb{C}(q)[K, K^{-1}][C'] 
 \end{align*}
 where $C'=H' + (q+2q^{-1})/(q+q^{-1})K^{-1} + (1+q)$ and $C'K^{-1/3}$ is a central element.  
 
 In \cite{AKR}, Aldenhoven, Koelink and Roman analyze the finite-dimensional modules for $\mathcal{B}_{\theta}$ in this case.  They use the algebra generated by  $\mathcal{H}$ and $\mathcal{H}'$ together  as the Cartan subalgebra for $\mathcal{B}_{\theta}$.  Using the relationship between $C$, $C'$ and the central elements $CK^{1/3}$, $C'K^{-1/3}$, it is straightforward to see that  this larger algebra is commutative.   Aldenhoven, Koelink and Roman show that any finite-dimensional simple $\mathcal{B}_{\theta}$-module is generated by a vector $v$, referred to as a highest weight vector, that is an eigenvector with respect to the action of $K$ and satisfies $B_1v=0$. They characterize finite-dimensional simple modules based on the eigenvalue of the highest weight vector and the dimension of the module (see \cite{AKR}, Proposition 3.1). Given a finite-dimensional simple module for $\mathcal{B}_{\theta}$, they explicitly describe the action of this large Cartan subalgebra on a basis (see \cite{AKR}, Corollary 3.3).

\subsection{Case 2: $n=3$}\label{section:nequals3}
 In this case $\Gamma_{\theta} = \{\beta_1, \beta_2\}$ where $\beta_1= \alpha_2$ and $\beta_2 = \alpha_1 + \alpha_2 + \alpha_3$.  The Cartan element associated to 
$\beta_1$ is simply $H_1 = B_2 = F_2 + E_2K_2^{-1}$.   Set $\beta = \beta_2$ and 
let  $\nu = \nu_1 = (\alpha_3 + 2\alpha_2 +3\alpha_1)/4$, the fundamental weight associated to $\alpha_1$.
Set 
\begin{align*}Y_{-\beta} = [ F_3 (F_2 F_1 - qF_1 F_2)  - q(F_2 F_1 - q F_1 F_2) F_3]
\end{align*} and note that 
\begin{align*}
Y_{-\beta} = [({\rm ad}\ F_3F_2F_1 )K_{-2\nu_1}]K_{2\nu_1- \beta} 
\end{align*}
up to a nonzero scalar multiple which we ignore.  Set 
\begin{align*}
X_{\beta} =  [E_1 (E_2 E_3 - qE_3 E_2)  - q(E_2 E_3 - q E_3 E_2) E_1]K_{-\beta}.
\end{align*} Note that  $\kappa(Y_{-\beta}) = X_{\beta}$ (up to multiplication by a nonzero scalar).

Set 
\begin{align}\label{H2}H_2 = [B_3 &(B_2 B_1 - qB_1 B_2)  - q(B_2 B_1 - qB_1 B_2) B_3] + B_2K_1K_3^{-1}.
\end{align}
A straightforward (lengthy) computation shows that 
\begin{align*}
H_2 = 
 &Y_{-\beta} +X_{\beta} - (q-q^{-1})FE_1K_3^{-1}
 +q(q-q^{-1}) (qF_1EK_3^{-1}K_2^{-1})\end{align*}
 where $F = (F_2F_1-qF_1F_2)$ and $E = (E_1E_2 - qE_2 E_1)$.  Thus $H_2$ satisfies the conclusions of Theorem \ref{theorem:main}.  The Cartan subalgebra for $\mathcal{B}_{\theta}$ in this case is 
 $\mathcal{H} = \mathbb{C}(q)[K, K^{-1}][H_1, H_2]$ where $K = K_1K_3^{-1}$.   Since $H_1=B_2$ and $H_2$  commute with each other,  we see that $B_2$ commutes with 
 \begin{align}\label{Bterm}
 (B_3 (B_2 B_1 - qB_1 B_2)  - q(B_2 B_1 - qB_1 B_2) B_3).
 \end{align}  Moreover, it follows from the explicit formulas for $H_1$ and $H_2$ that 
 \begin{align*}
 \mathcal{H}= \mathbb{C}(q)[K, K^{-1}][B_2, (B_3 &(B_2 B_1 - qB_1 B_2)  - q(B_2 B_1 - qB_1 B_2) B_3)].
 \end{align*}

\subsection{The  General Case}\label{section:ex_general}

We show that the general case for this family of examples is very similar to the low rank cases $n=2,3$. Given two elements $A, B$ in $U_q(\mathfrak{g})$ and a nonzero scalar $a\in \mathbb{C}(q)$, we define the $a$-commutator $[A,B]_a$ by 
\begin{align*}
[A,B]_a = AB-aBA.
\end{align*}
Define elements  $H'_j$ in $\mathcal{B}_{\theta}$ by 
\begin{align*}H_j' = [B_{n-j+1}, [B_{n-j}, \cdots, [B_{j+1}, B_j]_q]_q\cdots ]_q
\end{align*}
for all $j = 1, \dots, r$. 
Note that the elements  $H_j'$ satisfy the recursive relations
\begin{align}\label{recursion}
H_j' = [B_{n-j+1}, [H_{j+1}', B_j]_q]_q.
\end{align}
for $j = 1, \dots, r-1$.  
 In Theorem \ref{theorem:example} below, we prove that the Cartan subalgebra $\mathcal{H}$ of 
$\mathcal{B}_{\theta}$ defined by Theorem \ref{theorem:main} based on the choice of $\Gamma_{\theta}$ of Theorem \ref{theorem:cases_take2} satisfies 
\begin{align*}
\mathcal{H} = \mathbb{C}(q)[\mathcal{T}_{\theta}][H'_1, \dots, H'_r].
\end{align*}
For $n$ even, the elements $H_j'$  are the same as the  $h_j'$ defined by (5) of \cite{Wa1}  (up to reordering of the subscripts). An immediate consequence of Theorem \ref{theorem:example} is that the Cartan part in \cite{Wa1} is the same as the Cartan subalgebra presented here for the family of coideal subalgebras considered in \cite{Wa1}.

Note that the lowest weight term of $H_j'$ is 
\begin{align*}
W_j = [F_{n-j+1}, [F_{n-j}, \cdots, [F_{j+1}, F_j]_q]_q\cdots]_q = c[({\rm ad}\ F_{n-j+1}\cdots F_{j+1})F_jK_j ]K_{-\beta_j}
\end{align*}
for some nonzero scalar $c$ where $-\beta_j = -\alpha_j -\alpha_{j+1}-\cdots -\alpha_{n-j+1}$. In particular, the lowest weight term of $H_j'$ agrees with the lowest weight term of the Cartan element of $\mathcal{H}$ associated to weight $\beta_j$ as defined in Theorem \ref{theorem:main}.

\begin{lemma} \label{Hcommutativity} 
The elements $H_1', \dots, H_{r}'$  pairwise commute with each other. \end{lemma}
\begin{proof}
Assume first that  $n$ is odd.   Note that $H_k'$ is in the subalgebra of $U_q(\mathfrak{g})$ generated by the set  $\{B_i|\   i= k,k+1 \dots, n-k+1\}$,  We argue that $B_i$ commutes with $H_j'$ for all $i=j+1, \dots, n-j$, and so $H_k'$ commutes with $H_j'$ for all $k>j$.  This will prove the lemma when  $n$ is odd. 

 Note that $H'_{r} = B_{r}$ and $$H_{r-1}' =(B_{r+1} (B_r B_{r-1} - qB_{r-1} B_r)  - q(B_{r} B_{r-1} -qB_{r-1} B_r) B_{r+1}).$$ This is just (\ref{Bterm}) with $1,2,3$ replaced by $r-1, r, r+1$ respectively.  In particular, the results of Section \ref{section:nequals3} ensure that $H_r'=B_r$ commutes with $H_{r-1}'$. 
Now assume  that $j<r-1$ and $B_i$ commutes with $H_{j+1}'$ for all  $i=j+2, \dots, n-j-1$.  Since $B_i$ commutes with $B_j$ and $B_{n-j+1}$ for all $j+2\leq i\leq n-j-1$, it follows from (\ref{recursion}) that $B_i$ also commutes with $H_j'$.    Thus it is sufficient to show that $B_{j+1}$ and $B_{n-j}$ commutes with $H_j'$.  

Let $\mathcal{A}$ be the algebra generated by  $U_q(\mathfrak{sl}_{ r-1})$ and a new element $X$ subject only  to the following additional relations
\begin{itemize}
\item $F_iX-XF_i = E_iX-XE_i = 0$ for all $i\in \{1, \dots, r-2\}$.
\item $K_iXK_i^{-1} = X$ for all $i=1, \dots, r-1$.
\end{itemize}
Since $U_q(\mathfrak{sl}_{r-1})$ is a subalgebra of $\mathcal{A}$, $\mathcal{A}$ admits an $({\rm ad}\ U_q(\mathfrak{sl}_{r-1}))$ action. 
The above relations ensure that $X$ has weight $0$ with respect to this action and 
$({\rm ad}\ F_i)X = ({\rm ad}\ E_i)X = 0$
for all $i=1, \dots, r-2$.
Recall that $j\leq r-2$. Set 
\begin{align*}
V_1 &= ({\rm ad}\ E_{r-1}\cdots E_{j+1})E_j\cr
V_2 &=[({\rm ad}\ F_j\cdots F_{r-2})F_{r-1}K_{r-1}]K_j^{-1}\cdots K_{r-1}^{-1}\cr
V_3 &= [V_2, [X, V_1]_{q^{-1}}]_{q^{-1}}.
\end{align*}
Note that $V_1$ is the lowest weight vector in the $({\rm ad}\ U_{\{\alpha_{r-1}, \cdots, \alpha_{j+1}\}})$-module generated by $E_j$. Indeed, this follows from Lemma \ref{lemma:lowest_weight} with $\pi'= 
\{\alpha_{r-1}, \cdots, \alpha_{j+1}\}$ and $i=j$. Since $(\alpha_{j+1}, \alpha_{r-1}+\alpha_{r-2}+\cdots + \alpha_j)=0$, $V_1$ is also a trivial vector with respect to the action of $({\rm ad}\ U_{\{\alpha_{j+1}\}})$.  Hence
$({\rm ad}\ E_{j+1})V_1= ({\rm ad}\ F_{j+1})V_1=0.$
Similar arguments yield
$({\rm ad}\ F_{j+1})V_2 = ({\rm ad}\ E_{j+1})V_2 =0.$
It follows that $({\rm ad}\ F_{j+1})V_3= ({\rm ad}\ E_{j+1})V_3 =0.$  Since $V_3$ has weight $0$ with respect to the action of $({\rm ad}\ U_q(\mathfrak{sl}_{r-1})$, we have
\begin{align}\label{Vrein1}
[F_{j+1},V_3]= 0 \quad {\rm and}\quad [E_{j+1},V_3]=0.
\end{align}

It follows from the relations for $\mathcal{B}_{\theta}$ that the map defined by 
\begin{align*} &E_i\mapsto B_i,\quad  F_i\mapsto B_{n-i+1}, \quad K_i \mapsto K_{n-i+1}^{-1}K_i, \cr &K_{n-i+1}\mapsto K_{n-i+1}K_i^{-1},\quad X\mapsto B_r,  \quad q\mapsto q^{-1}
\end{align*} for $i=1, \dots, r-1$ defines a $\mathbb{C}$ algebra homomorphism.   Moreover, this map sends $V_3$ to $H_j'$.  Applying this homomorphism to (\ref{Vrein1}) yields
\begin{align*}
[B_{n-j}, H_j'] = 0 \quad {\rm and}\quad [B_{j+1}, H_j']=0
\end{align*}
as desired.

Now assume that $n$ is even. We have $H_r'= [B_{r+1}, B_r]_q$ and 
\begin{align*}
H_{r-1}' = [B_{r+2}, [B_{r+1}, [B_r, B_{r-1}]_q]_q]_q. 
\end{align*}
Note that both elements commute with all elements in $\mathcal{T}_{\theta}$.  Using the relations of $\mathcal{B}_{\theta}$, one can 
show that  Cartan element $H_{r-1}$ associated to the root 
$\beta_{r-1} = \alpha_{r+2} + \alpha_{r+1} + \alpha_r+\alpha_{r-1}$ satisfies
\begin{align}\label{Hrminus1}
H_{r-1} = H_{r-1}' + uH_r' + v
\end{align}
for appropriate elements $u, v\in \mathbb{C}(q)[\mathcal{T}_{\theta}]$.  This is very similar to the result for $H_2$ when $n=3$ given in Section \ref{section:nequals3}.    Since $B_r, B_{r+1}$ both commute with $H_{r-1}$, we must have $H_r'=[B_{r+1}, B_r]_q$ commutes with $H_{r-1}$.  Hence the expression for $H_{r-1}$ given in (\ref{Hrminus1}) ensures that 
$[B_{r+1}, B_r]_q$ commutes with $H_{r-1}'$.  Arguing using induction as in the $n$ is odd case, we see that $B_i,$ for all $ i\in \{j+1, j+2, \dots, r-1, r+2, \dots, n-j\}$ and $[B_{r+1}, B_r]_q$
commute with $H_j'$ for all $j$.  Note that for all  $k>j$, the element $H_k'$ is in the subalgebra generated by $B_i,$ for $ i\in \{j+1, j+2, \dots, r-1, r+2, \dots, n-j\}$ and $[B_{r+1}, B_r]_q$.  Hence $H_k'$ commutes with $H'_j$ for all $k>j$ and the lemma follows. 
\end{proof}

The next result shows that $\mathcal{H}'= \mathcal{H}$, thus establishing a nice formulas for a set of generators of $\mathcal{H}$ for this family of symmetric pairs. 

\begin{theorem} \label{theorem:example} Let $\mathfrak{g}, \mathfrak{g}^{\theta}$ be a symmetric pair of type AIII/AIV such that $r = \lfloor (n+1)/2\rfloor$. Let $\Gamma_{\theta} = \{\beta_1, \dots, \beta_r\}$ where $\beta_j = \alpha_j + \alpha_{j+1} + \cdots + \alpha_{n-j+1}$ and set $\alpha_{\beta_j} = \alpha_j$ for each $j$.  Then 
\begin{align}\label{calHeqn}
\mathcal{H} = \mathbb{C}[\mathcal{T}_{\theta}][H_1', \dots, H_r']
\end{align}
where $\mathcal{H}$ is the quantum Cartan subalgebra as defined in Theorem \ref{theorem:main} and 
\begin{align*}
H_j' = [B_{n-j+1}, [B_{n-j}, \dots, [B_{j+1}, B_j]_q]_q\cdots]_q
\end{align*}
for all $j=1, \dots, r$.
\end{theorem}
\begin{proof} If $n$ is odd, then $H_r = B_r = H_r'$ which is clearly in the right hand side of (\ref{calHeqn}).   If $n$ is even, then 
by (\ref{HDEFN}) of  Section \ref{section:nequals2}, we see that $H_r$ is in the right hand side of (\ref{calHeqn}).  

Now consider $j<r$.   Recall that the construction of the quantum Cartan element $H_j$
 in Theorem \ref{theorem:main} is based on  Proposition \ref{proposition:liftbeta} which in turn uses  Proposition \ref{proposition:BUiso}.  In the notation of Proposition \ref{proposition:BUiso}, $H_j =\tilde b$ where $H_j' = b$. 
Thus, we can write
\begin{align*}
H_j = H_j' + \sum_{{\rm wt}(I) <\beta_j}B_Ia_I
\end{align*}
where each $a_I\in \mathcal{T}_{\theta}$.   
Moreover, by the discussion at the end of Section \ref{section:decomp_proj}  and  Lemma \ref{lemma:decomp}, we see that $H_j$, $H_j'$, $H_j - H_j'$ are all elements in the set
\begin{align}\label{expansion}
\sum_{\lambda, \gamma, \gamma'}G^-_{-\beta_j+\lambda + \gamma}U^+_{\theta(-\lambda)-\gamma}K_{-\beta_j+2\gamma'}
\end{align}
where the sum runs over $\lambda, \gamma, \gamma'$  in $Q^+(\pi)$ such that $0< \lambda\leq \beta_j$, $\gamma\leq \theta(-\lambda)$, 
and $\gamma'\leq \gamma$.  Set $\pi_j = \{\alpha_{j+1}, \dots, \alpha_{n-j}\}$.

Let $\lambda+\gamma$ be minimal such that 
$H_j-H_j'$ admits a nonzero biweight summand of $l$-weight $-\beta_j+\lambda + \gamma$, say $guK_{-\beta_j+2\gamma'} \in G^-_{-\beta_j+\lambda + \gamma}U^+_{\theta(-\lambda)-\gamma}K_{-\beta_j+2\gamma'}$ where
$g\in G^-$ and $u\in U^+$. As explained above, $H_j$ and $H_j'$ have the same lowest weight term $W_j$ and so $\lambda +\gamma>0$.   Since $H_j-H_j'\in\mathcal{B}_{\theta}$, it follows from Remark \ref{remark:decomp}
 that  $guK_{-\beta_j+2\gamma'}\in U^-\mathcal{M}^+\mathcal{T}_{\theta}$. Note that  $\mathcal{M}^+$ is just the identity element because $\pi_{\theta}=\emptyset$ and so $u = 1$.  By Theorem  \ref{theorem:main}, 
$\mathcal{P}(H_j) = W_j$.  On the other hand $\mathcal{P}(gK_{-\beta_j+2\gamma'}) = gK_{-\beta_j+2\gamma'}$ and so $gK_{-\beta_j+2\gamma'}$ does not appear as a  biweight summand of $H_j$ with respect to the expansion along the lines of  (\ref{expansion}).  Hence $gK_{-\beta_j+2\gamma'}$ must appear as a biweight summand of $H_j'$.

 Note that  $\mathcal{M}^+$ is just the identity element because $\pi_{\theta}=\emptyset$ and so $u = 1$.   Since 
\begin{align*}
gK_{-\beta_j+2\gamma'} \in U_{-\beta+\lambda +\gamma}^-K_{-\beta_j+2\gamma'}=G^-_{-\beta_j+\lambda +\gamma}K_{2\gamma'-\lambda-\gamma},
\end{align*} we must have $2\gamma'-\lambda -\gamma\in Q(\pi)^{\theta}$.  Also, the fact that $H_j$ and $H_j'$ commutes with all elements in $\mathcal{T}_{\theta} = \langle K_iK_{n-i+1}^{-1}|\ i=1, \dots, r\rangle$ ensures that 
$p(\lambda+\gamma) = \lambda +\gamma$.  By the definition of $\beta_j$ and $\pi_j$ it follows that either $\beta_j-\lambda-\gamma\in Q^+(\pi_j)$ or $\lambda+\gamma\in Q^+(\pi_j)$.
If $\lambda +\gamma\in Q^+(\pi_j)$ then it follows from  Lemma \ref{lemma:weight_cons1} that $\lambda = \gamma = 0$ which contradicts the fact that $\lambda + \gamma>0$.   Hence $\beta_j-\lambda -\gamma\in Q^+(\pi_j)$ and $g \in U_{\pi_j}.$

Assume first that $n$ is odd. By Theorem \ref{theorem:main} and the proof of Lemma \ref{Hcommutativity},  both $H_j$ and $H_j'$ commute with all elements of $U_{\pi_j}\cap \mathcal{B}_{\theta}$.
By Lemma \ref{lemma:adF},  $gK_{-\beta_j+2\gamma'} $ is an $({\rm ad}\ U_{\pi_j})$ lowest weight vector. By Section \ref{section:dual_vermas},    the restriction of $-\beta_j+2\gamma' $ to 
$\pi_j$ is $ -2\xi$ for some $\xi\in P^+(\pi_j)$  and 
$\beta_j-\lambda - \gamma$ restricts to $\xi-w_j\xi$ where $w_j=w(\pi_j)_0$.  Since $\beta_j-\lambda-\gamma\in Q^+(\pi_j)$, these two weights are equal, namely $\beta_j-\lambda -\gamma = \xi-w_j\xi$.   Hence $\beta_j-\lambda -\gamma \in Q^+(\pi_j)\cap P^+(\pi_j)$.  The only  weight less than $\beta_j$ in $Q^+(\pi_j)$ which is also in $P^+(\pi_j)$ is $\beta_{j+1}$.   Hence the weight of $g$ is 
\begin{align}\label{xireln}
\xi-w_j\xi= \beta_j-\lambda-\gamma =\beta_{j+1} = \alpha_{j+1}+\cdots +\alpha_{n-j}.
\end{align} 
It follows that $\lambda+\gamma = \alpha_j+\alpha_{n-j+1}$.   Since $\gamma'\leq \gamma$, we must have $\gamma'\in \{0,\alpha_j, \alpha_{n-j+1}, \alpha_j+\alpha_{n-j+1}\}$.  Using the fact that $\beta_j-2\gamma'$ restricts to  $2\xi$  and $\xi$ satisfies (\ref{xireln}), it is straightforward to check that $\gamma' =\alpha_j$ or $\gamma' = \alpha_{n-j+1}$.
 
As explained above, $gK_{-\beta_j+2\gamma'}$ appears as a biweight summand in the expansion of  $H_j'$  using   (\ref{expansion}).   From the definition of $H_j'$ and the fact that $g$ has weight $-\beta_{j+1}$, $gK_{-\beta_j+2\gamma'}$ must appear in the weight vector expansion  of the sum of the following two terms:
 \begin{align*}
 [E_{j}K_{n-j+1}^{-1}, [F_{n-j}, \dots, [F_{j+1}, F_{j}]_q]_q\cdots]_q
 \end{align*}
 and 
  \begin{align*}
 [F_{n-j+1}, [F_{n-j}, \dots, [F_{j+1}, E_{n-j+1}K_j^{-1}]_q]_q\cdots]_q.
 \end{align*}
 It follows that $g = [F_{n-j},[F_{n-j-1},  \dots, [F_{j+2},F_{j+1}]_q\cdots]_q]_qK_{\beta_{j+1}}$ up to some nonzero scalar. Since $\gamma' = \alpha_j$ or $\alpha_{n-j+1}$, we have 
\begin{align*}
gK_{-\beta_j+2\gamma'} = gK_{-\beta_{j+1}}(K_jK_{n-j+1})^{-1}K_{2\gamma'} =  [F_{n-j},[F_{n-j-1},  \dots, [F_{j+2},F_{j+1}]_q\cdots]_q]_qt
\end{align*} where $t = K_jK_{n-j+1}^{-1}$ or $t = K_j^{-1}K_{n-j+1}.$
 Thus 
 \begin{align*}
 H_j \in  H_j' + H_{j+1}'t + \sum_{{\rm wt}(I)<\beta_{j+1}}B_Ia_I.
 \end{align*}

Note that all three elements $H_j, H_j', H_{j+1}'$ commute with all elements in $U_{\pi_{j+2}}\cap \mathcal{B}_{\theta}$ and all elements in $\mathcal{T}_{\theta}$.  So, using a similar argument as above yields 
\begin{align}\label{Hformula3}
H_j = H_j' +H_{j+1}'t +H_{j+2}'s + \sum_{{\rm wt}(I)  < \beta_{j+2} }B_Ia_I
\end{align}
with $s\in \mathcal{T}_{\theta}$.
Hence, repeated applications of this argument yield \begin{align*}H_j\in \mathbb{C}(q)[\mathcal{T}_{\theta}][H_j', H_{j+1}', \dots, H_r'].\end{align*}  Since the coefficient of $H_j'$ in $H_j$ is $1$, we get
\begin{align}\label{Hformula4}
\mathbb{C}(q)[\mathcal{T}_{\theta}][H_j, H_{j+1}, \dots, H_r] = \mathbb{C}(q)[\mathcal{T}_{\theta}][H_j', H_{j+1}', \dots, H_r'].
\end{align}
This proves (\ref{calHeqn}) when $n$ is odd. 
 
 Now assume that $n$ is even.  Let $V^-$ be the subalgebra of $U^-$ generated by $$F_1, \dots, F_{r-1}, [F_{r+1}, F_{r}]_q, F_{r+2}, \dots, F_n.$$ Note that $V^-$ is isomorphic to 
 $U^-_q(\mathfrak{sl}_{n-1})$ (the algebra generated by $F_1, \dots, F_{n-1}$) via the map defined by 
 \begin{itemize}
 \item   $F_i\mapsto F_i$ for $i = 1, \dots, r-1,$
 \item $[F_{r+1}, F_r]_q \mapsto F_r$,
 \item  $F_i\mapsto F_{i-1}$ for $i=r+2, \dots, n$. 
 \end{itemize} Hence $V^-$ inherits an adjoint action from $U^-_q(\mathfrak{sl}_{n-1})$.  
Now $H_j-H_j'$ commutes with 
 all elements in $\mathcal{T}_{\theta}$ and all elements in the algebra generated by 
 \begin{align*}B_{j+1}, \dots, B_{r-1}, [B_{r+1}, B_r]_q, B_{r+2}, \dots, B_{n-j}.
 \end{align*} The theorem follows using an argument very similar to the odd $n$ case.  The key ingredient is an analysis of the lowest weight summand of $H_j-H_j'$ with respect to the  adjoint action of $V^-$.  This term must lie in $U^-\mathcal{T}_{\theta}$ and be invariant under the action of $({\rm ad}\ \mathcal{T}_{\theta})$. This forces the lowest weight summand to be in the algebra generated by $V^-$ and $\mathcal{T}_{\theta}$. Arguing as above yields an   analog of (\ref{Hformula3}).  
 Repeated applications of this type of argument yield  (\ref{Hformula4}) for $n$ even which completes the proof of the theorem. \end{proof}

\begin{remark}\label{remark:lastremark} There is a close relationship between the quantum Cartan subalgebra in this case and the center of $\mathcal{B}_{\theta}$.   In particular, by \cite{KL},  there exists $z_j$ in $ \mathcal{B}_{\theta}$ such that 
\begin{itemize}
\item[(i)] $z_j\in ({\rm ad}\ U_{\pi_j})K_{-2\nu_j}$ 
\item[(ii)] $z_j= [({\rm ad} F_{n-j+1}\cdots F_{j+1})F_jK_j]K_{-2\nu_j}$ + higher weight terms
\item[(iii)] $z_j$ is in the center of $U_{\pi_j}\cap \mathcal{B}_{\theta}$.   
\end{itemize} 
Note that (ii) implies that 
\begin{align*}z_j = H_j'K_{\beta_j}K_{-2\nu_j} + \sum_{{\rm wt}(I)<\beta_j}B_I\mathcal{T}_{\theta}.
\end{align*} One can expand $({\rm ad}\ U_{\pi_j})K_{-2\nu_j}$ as a sum of weight spaces in a manner analogous to the expansion of the $B_I$ in Lemma \ref{lemma:decomp}.  (See for example \cite{KL}, Section 1.5 (10).)  An inductive argument similar to the one used in the above theorem yields
$\mathcal{H} = \mathbb{C}(q)[\mathcal{T}_{\theta}][z_1, \dots, z_r]$.
\end{remark}

\begin{remark}\label{finalremark} In \cite{Wa1}, Watanabe uses the braid group automorphisms of \cite{KP} in order to determine the action of the Cartan  on highest weight generating vectors.  A natural question is: can we use these automorphisms to construct quantum Cartan subalgebras?  We see that the answer is no for examples considered in this section where here we take the most obvious choice of braid group automorphism to pass from one Cartan element to another.  In particular, consider the case where $n=3$ as in Section \ref{section:nequals3}.   Applying the braid group automorphism $\tau_1^-$ of  \cite{KP} to the Cartan element $H_1 = B_2$ we get the element
\begin{align*}
\tau_1^{-1}(H_1)= (iq^{1/2})([B_1,[B_3,B_2]_q]_q -q B_2K_1K_3^{-1})
\end{align*} of $\mathcal{B}_{\theta}$ where here we have adjusted the outcome to reflect a slightly different choice of generators for $\mathcal{B}_{\theta}$.  On the surface, this term looks similar to $H_2$ as defined in (\ref{H2}).  However $[B_1, [B_3,B_2]_q]_q$ is not a scalar multiple of $[B_3[B_2,B_1]_q]_q$ and so $H_2$ is not equal to a scalar multiple of $\tau_1^-(H_1)$.  Moreover,    $[B_1, [B_3,B_2]_q]_q$ does not commute with $B_2$ and so  $H_1$ does not commute with $\tau_1^{-1}(H_1)$.  Thus, in contrast to the Cartan subalgebras studied here, the subalgebra generated by $H_1, \tau_1^{-1}(H_1), K_1K_3^{-1}, K_3K_1^{-1}$ is not  commutative. \end{remark}

\begin{remark} Note that the set $\Gamma= \{\beta_1, \dots, \beta_r\}$ where $r=\lfloor (n+1)/2\rfloor$ and 
\begin{align*}
\beta_j = \alpha_j + \alpha_{j+1} + \cdots + \alpha_{n-j+1}
\end{align*}
for $j=1, \dots, r$ is also a strongly orthogonal $\theta$-system for the symmetric pair $\mathfrak{g}, \mathfrak{g}^{\theta}$ of type AI.  Choosing $\alpha_{\beta_j} = \alpha_j$ and $\alpha_{\beta}' = \alpha_{\beta_{n-j+1}}$ for each $j$, we see that $\Gamma$ satisfies conditions Theorem \ref{theorem:cases_take2} (i) - (iv) but does not fall into one of the fives cases of this theorem.  A natural question is:  can we still identify a quantum Cartan subalgebra of $\mathcal{B}_{\theta}$ defined by this strongly orthogonal $\theta$-system?    In the case when $n=3$, a straightforward computation shows that $H'_{\beta_2}= B_2$ commutes with $H'_{\beta_1} = [B_3,[B_2,B_1]_q]_q$
where $B_i = F_i +E_iK_i^{-1}$ for each $i$.    More generally, in analogy to the type AIII case,   the algebra $\mathcal{H}' = \mathbb{C}(q)[H'_{1},\dots, H'_{r}]$ is a commutative polynomial ring where 
\begin{align*}
H_j' =[B_{n-j+1}, [B_{n-j}, \cdots, [B_{j+1}, B_j]_q]_q\cdots ]_q
\end{align*}
for $j=1, \dots, r$.  Furthermore,  $\mathcal{H}'$ specializes to the enveloping algebra of a  Cartan subalgebra of $\mathfrak{g}^{\theta}$ as $q$ goes to $1$.   However, $\mathcal{H}$ is not  $\kappa_{\theta}$ invariant. This is because  $\kappa_{\theta}(B_j) = B_j$ for $j-1, \dots, n$ and so   \begin{align*}\kappa_{\theta}(H_j') = [B_j,[B_{j+1}, \cdots, [B_{n-j}, B_{n-j+1}]_q]_q\cdots]_q\neq H_j',\end{align*}  and, moreover,   $\kappa_{\theta}(H_j')\notin \mathcal{H}'$ for all $j < (n+1)/2$.  Hence Corollary \ref{cor:unitary1} does not apply to $\mathcal{H}'$.  Nevertheless, it is likely that one can use a  twist of $\kappa_{\theta}$ obtained by composing it with  an automorphism of $U_q(\mathfrak{g})$ induced by the nontrivial diagram automorphism for type $A_n$ in order to establish the semisimplicity of unitary $\mathcal{B}_{\theta}$-modules with respect to the action of $\mathcal{H}'$.  Thus $\mathcal{H}'$ could potentially serve as an alternative quantum Cartan subalgebra for $\mathcal{B}_{\theta}$ distinct from the quantum Cartan subalgebra in type AI defined by Theorems \ref{theorem:cases_take2} and \ref{theorem:main} (see Remark \ref{nonstandard}).    \end{remark}

\small
\section{Appendix: Commonly used Notation}

\noindent
{\bf Section 2 }:
$\mathbb{C}$, $\mathbb{R}$, 
$\mathbb{N}$, $\mathfrak{g}$, $\theta$,  $\mathfrak{g}^{\theta}$ 

\medskip
\noindent
{\bf Section 2.1 }:
$\mathfrak{n}^-$, $\mathfrak{h}$, $ \mathfrak{n}^+$, 
$\Delta$, 
$\Delta^+$, $h_i, e_{\alpha},f_{-\alpha}$,  
$\pi=\{\alpha_1, \dots, \alpha_n\}$,  
$\mathcal{W}$, 
$w_0$,  $e_i,f_i$, 
$h_{\alpha}$,  $(\ , \ )$,  
$\mathfrak{g}_{\pi'}$, $Q(\pi')$,  $Q^+(\pi')$,  
$\Delta(\pi')$,  $\mathcal{W}_{\pi'}$,  
$w(\pi')_0$,  $P(\pi)$,  $P^+(\pi)$,  
$P(\pi')$,  $P^+(\pi')$,
${\rm mult}_{\alpha}(\beta)$, 
${\rm Supp}(\beta)$,
${\rm ht}(\beta)$

\medskip
\noindent
{\bf Section 2.2:}
$\mathfrak{g}_{\mathbb{R}}$,
$\mathfrak{h}_{\mathbb{R}}$,
$\mathfrak{t}_{\mathbb{R}}$, $\mathfrak{p}_{\mathbb{R}}$,
$\mathfrak{a}_{\mathbb{R}}$, $Q(\pi)^{\theta}$, $P(\pi)^{\theta}$, 
$\Delta_{\theta}$,
${\rm Orth}(\beta) $,
${\rm StrOrth}(\beta) $, $\mathfrak{h}_{\theta}$, $\mathfrak{a}$

\medskip
\noindent
{\bf Section 2.4:}
$\pi_{\theta}$, $p$, 
$\Gamma_{\theta}$, 
$\alpha_{\beta}, \alpha_{\beta}'$

\medskip
\noindent
{\bf Section 3.1:} $U_q(\mathfrak{g})$,
$E_i, F_i, K_i^{\pm 1}$, 
$\epsilon$, 
$U^+$, 
$U^-$, 
$\mathcal{T}$, 
$U^0$, $G^+$, $G^-$, 
$\hat U$, 
$V_{\lambda}$, biweight,  
$l$-weight,  $K_{\xi}$,  
$\check{ \mathcal{T}}$,  
$\check U$,  
$M(\lambda)$,  $L(\lambda)$,  
$U_{\pi'}$, 
$L_{\pi'}(\lambda)$, 
$S^{U_{\pi'}}$

medskip
\noindent
{\bf Section 3.2:} ${\rm ad}_{\lambda}$, $G^-(\lambda)$,  
$\delta M(\lambda/2)$, 
$\deg_{\mathcal{F}}$, 
$\tilde\lambda$, $F(\check U)$, 
$F_{\pi'}(\check U)$ 

\medskip
\noindent
{\bf Section 3.3: } $T_i$, 
$T_w$, 
$\sigma$ 

\medskip
\noindent
{\bf Section 3.4: }
$\kappa$

\medskip
\noindent
{\bf Section 4.1:}
$\mathcal{M}$,
$\theta_q$,
$\mathcal{T}_{\theta}$, $\mathcal{B}_{\theta}$, $B_i$  ($\alpha_i\notin\pi_{\theta}$), $\bf{c}=(c_1,\dots, c_n)$, $\bf{s} = (s_1, \dots, s_n)$, 
$\mathcal{S}$,
$\kappa_{\theta}$,
$\check{\mathcal{T}}_{\theta}$, $\check{\mathcal{B}}_{\theta}$

\medskip
\noindent
{\bf Section 4.2:}
$\mathcal{M}^+$, $B_i$ ($\alpha_i\in \pi_{\theta}$),
$F_J$ , $B_J$,  
${\rm wt}(j_1, \dots, j_m) $

\medskip
\noindent
{\bf Section 4.3:}
$\mathcal{P}$

\medskip
\noindent
{\bf Section 5.1:}
$\bar L(\lambda)$,
$\bar L_{\pi'}(\lambda)$, $[\bar L(\lambda):\bar L_{\pi'}(\gamma)]$,
$\nu_{\alpha}$

\medskip
\noindent
{\bf Section 5.2:}
$[L(\lambda): L_{\pi'}(\gamma)]$

\medskip
\noindent
{\bf Section 5.3:}
$\varphi$, $\varphi'$

\medskip
\noindent
{\bf Section 6.1:}
$\mathcal{N}(D:C)$,
$[b,a]$

\medskip
\noindent
{\bf Section 6.2:}
${\rm ht}_{\tau}(\beta)$

\medskip
\noindent
{\bf Section 6.4:}
$G^-_{(\tau,m)}$, 
$U^+_{(\tau,m)}$

\end{document}